\documentclass{article}[11pt]
\usepackage{geometry}
\usepackage{tikz}
\usepackage{tikz-cd}

\usetikzlibrary{decorations.markings}
\usetikzlibrary{arrows, automata}
\usepackage{bbold, eufrak}
\usepackage{amsfonts,amssymb}
\usepackage{float}
\usepackage{hyperref}
\usepackage{enumerate}
\usepackage{enumitem}
\usepackage{amsmath, amsthm, amsfonts, mathabx}
\usepackage[capitalise]{cleveref}

\usepackage{enumitem}
\usepackage[official]{eurosym}

\usepackage{nomencl}
\usepackage{bm}

\usepackage{amsfonts}
\usepackage{fancyhdr}
\usepackage{amscd}
\usepackage[english]{babel}
\usepackage{xcolor}
\usepackage[capitalise]{cleveref}
\usepackage{mathtools}
\newtheorem{thm}{Theorem}[section]
\newtheorem{cor}[thm]{\scshape{Corollary}}
\newtheorem{prop}[thm]{\scshape{Proposition}}
\newtheorem{lemma}[thm]{\scshape{Lemma}}
\newtheorem{claim}[thm]{\scshape{Claim}}
\newtheorem{question}[thm]{\scshape{Question}}
\newtheorem{notation}[thm]{\scshape{Notation}}
\newtheorem{conj}[thm]{\scshape{Conjecture}}
\newtheorem{remark}[thm]{\scshape{Remark}}
\newtheorem{obs}[thm]{\scshape{Observation}}

\newtheorem{example}[thm]{\scshape{Example}}
\newtheorem*{thm*}{Theorem}
\newtheorem*{cor*}{Corollary}

\def\G{\Gamma}
\def\D{\Delta}

\def\={\cong}

\def\mc{\mathcal}
\def\mbb{\mathbb}

\newcommand{\N}[0]{\mathbb{N}}
\newcommand{\subg}[1]{\langle #1 \rangle}

\newcommand{\frp}[0]{\ast}
\newcommand{\F}[0]{\mathbb{F}}
\renewcommand{\G}[0]{\mathcal{G}}
\newcommand{\Z}[0]{\mathbb{Z}}
\newcommand{\GG}[0]{\mathbb{G}}

\newcommand{\nin}[0]{\notin}

\newcommand{\bvee}[3]{ \bigvee\limits_{#1}^{#3}#2}

\newcommand {\BZ}{\mathbb{Z}}

\newcommand{\dis}[1]{d(*, #1 *)}

\newcommand {\onto}{\twoheadrightarrow}

\newcommand{\gvv}{\{G_{v}\}_{v\in \Gamma}} %
\newcommand{\hww}[0]{\{H_{w}\}_{w\in V\Delta}}

\newcommand{\hh}[0]{\mathcal{H}}
\newcommand{\HH}[0]{\mathcal{H}}

\newcommand{\U}[0]{\mathcal{U}}

\newcommand{\W}[0]{\mathcal{W}}
\newcommand{\ff}[0]{\mathcal{F}}

\usepackage{verbatim}
\usepackage{amsmath}  \usepackage{scalefnt}

\DeclareMathOperator*{\bfrp}{\ast}
\DeclareMathOperator{\st}{star}
\DeclareMathOperator{\ext}{ext}
\DeclareMathOperator{\Fill}{Fill}
\DeclareMathOperator{\Th}{Th}

\DeclareMathOperator{\lk}{link}

\newcommand{\link}[1]{\lk(#1)}
\DeclareMathOperator{\Core}{Core}
\DeclareMathOperator{\minCore}{MinCore}
\DeclareMathOperator{\ECore}{ECore}
\DeclareMathOperator{\supp}{supp}

\usepackage{mathabx,epsfig}
\def\acts{\mathrel{\reflectbox{$\righttoleftarrow$}}}

\newenvironment{subproof}[1][\proofname]{%
	\begin{proof}[#1]%
	}{%
	\end{proof}%
}

\newcommand{\env}[1]{\supp(#1)}

\begin{document}
\title{On the elementary theory of graph products of groups}
\author{Montserrat Casals-Ruiz\footnote{Ikerbasque - Basque Foundation for Science and Matematika Saila,  UPV/EHU,  Sarriena s/n, 48940, Leioa - Bizkaia, Spain; Contact: \texttt{montsecasals@gmail.com}},  \ 
	Ilya Kazachkov \footnote{Ikerbasque - Basque Foundation for Science and Matematika Saila,  UPV/EHU,  Sarriena s/n, 48940, Leioa - Bizkaia, Spain; Contact: \texttt{ilya.kazachkov@gmail.com}}, \ Javier de la Nuez Gonz\'alez\footnote{Departamento de Matem\'aticas, Universidad Aut\'onoma de Madrid - Madrid, Spain; Contact: \texttt{jnuezgonzalez@gmail.com}\newline
		The authors were supported by the Basque Government Grant IT974-16 and the Spanish Government grants PID2019-107444GA-I00 and MTM2017-86802-P, partly with FEDER funds. The third author was also supported by the Spanish Government grant MTM2017-82690-P, partly 
		with FEDER funds. \newline
		Keywords: elementary equivalence, right-angled Artin group, graph product of groups.\newline
		AMS subject classification:  20E08, 20E06, 20F70, 03C07.
	}
}
\bibliographystyle{plain}

	\theoremstyle{definition}
	\newtheorem{definition}[thm]{Definition}
	\maketitle
	\begin{abstract}
		In this paper we study the elementary theory of graph products of groups and show that under natural conditions on the vertex groups we can recover (the core of) the underlying  graph and the associated vertex groups. More precisely, we require the vertex groups to satisfy a non-generic almost positive sentence, a condition which generalizes a range of natural ``non-freeness conditions" such as the satisfaction of a group law, having nontrivial center or being boundedly simple. 
		
		As a corollary, we determine an invariant of the elementary theory of a right-angled Artin group, the core of the defining graph, which we conjecture to determine the elementary class of the RAAG. We further combine our results with the results of Sela on free products of groups to describe all finitely generated groups elementarily equivalent to certain RAAGs. We also deduce rigidity results on the elementary classification of graph products of groups for specific types of vertex groups, such as finite, nilpotent or classical linear groups.
	\end{abstract}
	
	\tableofcontents
	\section{Introduction}
	
	Given a finite simplicial graph $\Gamma=(V(\Gamma),E(\Gamma))$ and a collection of groups $\{G_v \mid v\in \Gamma\}$ associated with the vertices of $\Gamma$, the graph product of the groups $G_v$, denoted by $\G=\G(\Gamma, \gvv)$ is defined to be the group $\G=F/R$ where $F$ is the free product of the groups $G_v$ and relations $R$ establish that elements of $G_u$ commute with elements of $G_v$ whenever $(u,v)\in E(\Gamma)$. This group construction is quite rich and many classical families of groups can be viewed as graph products, such as free groups, free abelian groups and right-angled Artin and Coxeter groups, among others.
	
	The aim of this paper is to study the elementary theory of graph products of groups and, more precisely, to explore how much of the underlying structure of the graph product, that is the graph $\Gamma$ and the associated vertex groups $G_v$, is remembered by the elementary theory of the group. Roughly speaking, we show that under some natural conditions on the graph (to be reduced) and the vertex groups (to satisfy a non-generic almost positive sentence), we can recover the optimal part of the graph (its core) and the associated vertex groups. More precisely, we prove the following:
	\begin{thm*}[Interpretation]
		Let $\G=\G(\Gamma, \gvv)$ be a graph product of groups. Suppose that $\Gamma$ is reduced and $G_v$ have non-generic almost positive theory for all $v\in \Gamma$.
		
		Then the core of $\Gamma$ and its associated groups, $\Core(\Gamma, \{G_v\}_{v\in \Core(\Gamma)})$, are interpretable in $\G$. Furthermore, the interpretation is uniform, in particular independent of the graph and the associated subgroups, in an appropriate subclass of graph products.
		
		In particular, if the vertex groups $G_v$, $v\in \Gamma \setminus \Core(\Gamma)$ are not elementarily equivalent to the infinite cyclic group, then $\Gamma$ and $G_v$, $v\in \Gamma$ are uniformly interpretable in $\G$.
	\end{thm*}
	
	\begin{cor*} [Elementary rigidity]
		Let $\G=\G(\Gamma, \gvv)$ and $\G'=\G(\Gamma', \{H_w\})$ be graph products of groups. Suppose that $\Gamma, \Gamma'$ are reduced and $G_v, H_w$ have non-generic almost positive theory for all $v\in \Gamma$ and $w\in \Gamma'$.
		
		If $\G \equiv \G'$, then there is a graph isomorphism $f:\Core(\Gamma) \to \Core(\Gamma')$ and $G_v\equiv H_{f(v)}$.
	\end{cor*}
	
	When one considers this type of questions, it is immediately apparent that in order to be able to recover at least part of the underlying structure of the graph product, some restrictions have to be imposed, especially on the class of groups. Indeed, if one restricts to the class of graph products of finite groups, then one can show, see \cite{CKR}, that the graph and the vertex groups can be fully recovered from the elementary theory of the group. On the other hand, all non-abelian free groups are elementarily equivalent and so in this case, the elementary theory does not remember the number of vertices of the defining graph (if we view free groups as graph products, where the underlying graph is edgeless and the vertex groups are infinite cyclic).
	
	The key difference that determines whether or not the underlying structure can be recovered is the fact that in a free product of finite groups, the set consisting of all vertex groups (and all their conjugates) is definable, since,  up to conjugacy, vertex groups is the set of elements of bounded exponent, while this is not the case  in the free group.
	
	Our approach is to determine conditions on the vertex groups that assure a phenomenon similar to the case of finite groups. However, when one tries to extend from finite to infinite groups, say to free abelian groups, one immediately runs into the difficulties with the infinite cyclic groups highlighted above  and arrives at the following dichotomy: if we consider free products of free abelian groups of rank at least $2$, then the elementary theory of the free product does remember the number of factors and the ranks of the vertex groups in the free product. However, if we allow for infinite cyclic groups, then, by a result of Sela, we cannot recover those factors as for instance $\mathbb Z^n \ast \mathbb Z \equiv \mathbb Z^n \ast F_n$ for all $n\ge 1$. This phenomenon shows that even if we consider vertex groups that satisfy strong ``non-freeness" conditions, we are only able to recover part of the graph $\Gamma$, called the core, see Definition \ref{defn:ecore}, which excludes some vertices whose associated groups are (elementary equivalent to the) infinite cyclic subgroup. (In particular, if none of the vertex groups are (elementary equivalent to the) infinite cyclic group, then we can interpret the graph $\Gamma$ and all the corresponding vertex groups.)
	
	The condition that we require on the vertex groups is to have a non-generic almost positive theory.
	
	The positive theory of a group, denoted by $\Th^+(G)$, is the fragment of the first order theory that contains the positive sentences satisfied by $G$. Recall that a positive sentence is one that admits prenex normal form of the following type:
	$$
	\forall x^{1}\exists y^{1}\forall x^{2}\dots\forall x^{m}\exists y^{m}\,\bvee{j=1}{\Sigma_{j}(x^{1},y^{1},x^{2}, \dots, x^{m},y^{m})=1}{k},
	$$
	where $x^i$ and $y^i$ are tuples of variables for all $1\leq i\leq m$ and $\Sigma_j$ is a conjunction of atomic terms (i.e. a conjunction of equations), for all $1\leq j \leq k$.
	
	We say that a group has nontrivial positive theory if it satisfies a positive sentence that it is not satisfied by a non-abelian free group. The class of groups with nontrivial positive theory is quite rich and contains, among others, the class of groups satisfying an identity, for instance nilpotent and solvable groups; groups with a verbal subgroup of finite width, for instance groups with finite commutator width such as profinite groups, some groups acting on rooted trees, such as Grigorchuck and Sidki groups, some groups of homeomorphisms and diffeomorphisms such as the Thompson's group, see Section \ref{sec: generalized positive theory} for other examples.
	
	The almost positive theory of a group, as the name suggests, is a fragment of the first-order theory that contains the positive theory. Roughly speaking, we consider sentences with some controlled inequalities (which interact well with the formal solutions of the positive part of the sentence), see Definition \ref{defn:generic ap}. At this point we will not discuss the precise definition in details, but the type of property that one can express with a non-generic almost positive sentence is, for instance, to have a nontrivial center: $\exists z \ne 1 \,  \forall x \, [x,z]=1$.
	
	\bigskip
	
	The second restriction that appears in the statement of our main result, refers to the graph - we require $\Gamma$ to be reduced. The reason behind this requirement comes from the fact that the decomposition of a group as a graph product may not be unique. Indeed, for instance, Baumslag provided an example of a nilpotent group $N$ which is both the direct product of $2$ and $k$ directly indecomposable subgroups. It follows from our results that, under the assumptions on the vertex groups, direct products are the only impediment for the rigidity of the graph product up to isomorphism. More precisely, we deduce the following:
	\begin{cor*}[Algebraic rigidity]
		Let $\G=\G(\Gamma, \gvv)$ and $\G'=\G(\Gamma', \{H_w\})$ be graph products of groups. Suppose that $\Gamma, \Gamma'$ are reduced and $G_v, H_w$ have non-generic almost positive theory for all $v\in \Gamma$ and $w\in \Gamma'$.
		
		If $\G \simeq \G'$, then there is a graph isomorphism $f:\Gamma \to \Gamma'$ and $G_v\simeq H_{f(v)}$.
	\end{cor*}
	
	We notice that one may be able to generalise the conditions on the vertex groups and still be able to interpret them inside the elementary theory of the graph product, but it is our belief that recovering the core of the graph is the best one can do. In this sense, our result is optimal.
	
	We then apply our main result to deduce rigidity results on the classification of some classes of graph products up to elementary equivalence, and to study the elementary classification of (certain) RAAGs.
	
	In the case of RAAGs, our main theorem shows that the core of the graph is an elementary invariant of the group. In this setting, we deduce the following
	\begin{cor*}
		Let $G=\G(\Gamma, \{ \mathbb Z\})$ and $G'=\G(\Gamma, \{ \mathbb Z\})$ be two right-angled Artin groups. If $G\equiv G'$, then the cores of $G$ and $G'$ are isomorphic.
	\end{cor*}
	
	In fact, we conjecture that the core determines the elementary class of a RAAG, that is, we propose the following conjecture:
	
	\begin{conj}
		Let $G=\G(\Gamma, \{ \mathbb Z\})$ and $G'=\G(\Gamma, \{ \mathbb Z\})$ be two right-angled Artin groups. Then, $G\equiv G'$ if and only if the cores of $G$ and $G'$ are isomorphic.
	\end{conj}
	
	Notice that if two RAAGs have isomorphic  cores, then they have the same universal and existential theories, see Lemma \ref{lem:univequiv}.
	
	Furthermore, combining our results with Sela's work on free products, \cite{sela2010diophantine}, we show that the conjecture holds for the class of Droms RAAGs, which is defined to be the class of RAAGs whose finitely generated subgroups are again RAAGs. In fact, we are able to describe all finitely generated groups elementarily equivalent to a given Droms RAAG. In Theorem \ref{thm:eqdromsraags}, we give a recursive description of an elementary graph tower over a Droms RAAG and show the following
	\begin{thm}
		Let $G=\G(\Gamma, \{ \mathbb Z\})$ be a Droms RAAG and $H$ be a finitely generated group. Then, $G \equiv H$ if and only if $H$ is an elementary graph tower over $G$.
	\end{thm}
	
	Finally, we mention some of the rigidity results that follow from our main result. First of all, we recover the classification of graph products of finite groups.
	
	\begin{cor}[cf. \cite{CKR}]
		Let $G=\mathcal G(\Gamma, \gvv)$ and $G'=\mathcal G (\Gamma', \hww)$ where $G_v, H_w$ are finite groups. Then
		$$
		G \equiv G' \Leftrightarrow G\simeq G'.
		$$
		In particular, two right-angled Coxeter groups are elementarily equivalent if and only if they are isomorphic.
	\end{cor}
	
	Combining our result with results on elementary classification of nilpotent and classical linear groups we deduce the following corollaries.
	
	\begin{cor*}[Rigidity of the class of graph products of (non-cyclic) finitely generated nilpotent groups]
		Let $G=\mathcal G(\Gamma, \gvv)$ and $G'=\mathcal G (\Gamma', \{H_w\})$, where $G_v, H_w$ are finitely generated nilpotent groups, but not the infinite cyclic group.  Then
		$$
		G \equiv H \Leftrightarrow \Gamma \simeq \Gamma' \hbox{ and } G_v \times \mathbb Z \simeq H_v \times \mathbb Z.
		$$
	\end{cor*}
	
	\begin{cor*}[Rigidity of the class of graph products of classic linear groups]
		Let $\textbf C_{L}$ be the class of classical linear groups {\rm(}$GL_n(R), SL_n(R), UT_n(R)$, $n\ge 2${\rm)}. Let $G=\mathcal G(\Gamma, \gvv)$ and $H=\mathcal G (\Delta, \{H_v\}_{v\in \Delta})$, where $G_v, H_w \in \textbf C_{L}$. Then
		\begin{gather}
			\begin{split}
				G \equiv H \Leftrightarrow & \,
				\phi:\Gamma \to \Delta \hbox{ is an isomorphism}, \\  & G_v,G_{\phi(v)} \hbox{ are linear groups of the same type,}\\ & \hbox{the same rank and over elementarily equivalent rings}, \\ & (\hbox{if  $G_v=SL_n(R)$, then $G_{\phi(v)}=SL_n(R')$} \hbox{and } R\equiv R').
			\end{split}
		\end{gather}
	\end{cor*}
	\section{Preliminaries}
	\subsection{Graph products}
	
	\subsubsection{Graphs}
	
	\newcommand{\reducible}[0]{decomposable }    
	By a \emph{graph} we mean a finite simplicial graph, that is an unlabeled, undirected, graph without loops (i.e.\ edges that begin and end at the same vertex) and multiple edges. The sets of vertices and edges of a graph $\Gamma$ are denoted by $V(\Gamma)$, $E(\Gamma)$, or simply $V$ and $E$, respectively. Sometimes we abuse the notation and write $v\in \Gamma$ instead of $v\in V(\Gamma)$. If $v,u\in V$ are connected by an edge, we write $(v,u)\in E$. An \emph{induced subgraph} $\Delta$ of $\Gamma$ is a graph such that $V(\Delta) \subset V(\Gamma)$ and $(v,u)\in E(\Delta)$ if and only if $(v,u)\in E(\Gamma)$ for all $u,v\in V(\Delta)$. If not stated otherwise, by a subgraph we mean an induced subgraph. 
	
	The set of maximal cliques of $\Gamma$ will be denoted by $\mathcal{K}(\Gamma)$, or just $\mathcal{K}$.
	Given a graph $\Gamma=(V,E)$, by the complement graph of $\Gamma$ we mean the graph $\bar{\Gamma}=(V,V^{2}\setminus (E\cup \{(v,v)\}_{v\in V}))$. 
	
	We say that a graph $\Gamma=(V,E)$ is \emph{\reducible} if it is a join, that is there is a partition of $V$ into two non-empty sets $V_{0}$ and $V_{1}$ such that there is an edge between every vertex of $V_{0}$ and every vertex of $V_{1}$. In this case, we write $\Gamma=\Gamma_{0}\oplus\Gamma_{1}$, where $\Gamma_{i}$ is the subgraph induced by $V_{i}$, $i=1,2$.
	
	We say that two subsets of vertices $W,W'$ are \emph{orthogonal} if they are disjoint and there is an edge between any $w\in W$ and $w'\in W'$. Given a vertex $v\in\Gamma$ (resp. subgraph $\Delta\subset\Gamma$)), we denote by $\link{v}$ (resp. $\link{\Delta}$) the subgraph spanned by all the vertices that are adjacent to $v$ (resp. adjacent to all the vertices in $\Delta$ and do not belong to $\Delta$). Note that $\link{v}$ (resp. $\link{\Delta}$) does not contain $v$ (resp. $\Delta$). We define the star of $v$ and denoted by $\st(v)$ (resp. the star of $\Delta$ denoted by $\st(\Delta)$) as $\st(v)=\link{v}\cup\{v\}$ (resp. $\st(\Delta)=\link{\Delta} \cup \Delta$).

	\begin{definition}[Property $AP_n$] \label{defn:propDn}
		We say that a graph $\Gamma$ \emph{the property} $AP_{n}$, for $n\in\mathbb{N}$, if the complement graph $\bar \Gamma$ does not contain paths of length $n$ as induced subgraphs. In other words: there do not exist vertices $v_{1},v_{2},v_{3}\cdots v_{n+1}$ such that  $(v_{i},v_{j}) \in E$ if and only if $|i-j|>1$.
	\end{definition}
	
	\begin{definition}[Property $K_n$] \label{defn:propKn}
		We say that a graph $\Gamma$ satisfies the property $K_{n}$, for $n\in\mathbb{N}$, if any complete subgraph has size at most $n$.
	\end{definition}
	\begin{remark}
		Propety $K_{n}$ implies property $AP_{2n}$ 
	\end{remark}
	
	\begin{definition}[Full graph morphism]\label{defn:full morphism}
		A full graph morphism $\phi: \Gamma \to \Delta$ is a map from $V(\Gamma)$ to $V(\Delta)$ such that either $\phi(u)=\phi(w)$ or $(v,w)\in E(\Gamma)$ if and only if $(\phi(v), \phi(w))\in E(\Delta)$. Equivalently, for all $v'\in \Delta$, for all $v_1,v_2 \in \phi^{-1}(v')$ and for all $u\notin \phi^{-1}(v')$ we have that $(v_1,u)\in E(\Gamma)$ if and only if $(v_2,u)\in E(\Gamma)$.
	\end{definition}

	\subsubsection{Graph products}
	\newcommand{\GP}[0]{\mc{G}(\Gamma,\gvv)}

	By a group-labelled graph we mean a pair $(\Gamma,\gvv)$ consisting of a graph $\Gamma$ and a collection of groups $\{G_v \mid v\in \Gamma\}$. The graph product associated to such group-labelled graph is denoted by $\mathcal{G}(\Gamma, \gvv)$, or simply $\mathcal{G}$. Recall that $\mathcal{G}$ is the quotient of the free product of the groups $G_{v}$ by the relations expressing that $g\in G_{u}$ and $h\in G_{v}$ commute whenever $(u,v)\in E(\Gamma)$. The graph $\Gamma$ is called the \emph{underlying graph} of $\mc{G}$.
	For a subgraph $\Delta$ of $\Gamma$ (or a subset of vertices of $V(\Gamma)$), the subgroup $\langle \{G_v \mid v \in \Delta \} \rangle$, which we denote by $\mc{G}(\Delta)$ is isomorphic to $\mathcal{G}(\Delta,\{G_{v}\}_{v\in \Delta})$. By convention, we set $\mc{G}(\emptyset)=1$.
	
	\begin{definition}[Basics definitions on graph products]\
		Let $\mathcal{G}=\GP$ be a graph product. By a \emph{subgraph product} we mean a conjugate of a group of the form $\mc{G}(\Delta)$ for some subgraph $\Delta$ of $\Gamma$.
		
		Given a subgroup $H\leq\mathcal{G}$, we say that $H$ is \emph{of smaller complexity} if it is a subgroup of some proper subgraph product of $\mc{G}$. We say that $H$ is \emph{singular} if it is a subgroup of some conjugate of $G_v$ for some $v\in \Gamma$. Similarly, we say that $g\in\mc{G}$ is singular if $\subg{g}$ is singular.
	\end{definition}

	We will need some basic properties of graph products. See \cite{Green} for proofs of the lemmas below.

	\begin{lemma}\label{l: prelims_conjugation_and_graphical_subgroups}
		Let $\Delta$ be a subgraph of $\Gamma$, and let $g\in \mc{G}$. Suppose that there exists a subgraph $\Theta$ of $\Gamma$ such that  $\mc{G}(\Delta )^g = \mc{G}(\Theta)$. Then  $\mc{G}(\Theta )^g =\mc{G}(\Delta)$ and $\Delta = \Theta$.
	\end{lemma}
	
	The following lemma deals with intersections of graphical subgroups of $\mc{G}$ and their conjugates.
		\begin{lemma}\label{l: intersection_subgroups}
		Let $\Delta_1$ and $\Delta_2$ be two subgraphs of $\Gamma$. Then, for all $g_1, g_2\in \mc{G}$, there exists some $g_{3}\in\mathcal{G}$ and some subgraph $\Delta_3$ of $\Delta_{1}\cap\Delta_{2}$ such that $\mc{G}(\Delta_1)^{g_1}\cap \mc{G}(\Delta_2)^{g_2}=\mc{G}(\Delta_{3})^{g_{3}}$. In case $g_1=g_2$ then $\D_3=\D_1\cap \D_2$.
	\end{lemma}
	
	\begin{definition}[Factors and minimal subgraph of a subgroup]\label{d: minimal_graph_subg_and_factors}
		
		Given a subgraph product $G$ there is a well defined maximal collection $\ff(G)$ of subgraph products, called \emph{factors of $G$}, such that $G=\prod_{\hh\in\ff(G)}\hh$. Indeed, if $G=\mc{G}(\Delta)^h$ and $\Delta=\Delta_1 \oplus \Delta_2 \oplus \dots \oplus \Delta_k$ where $\Delta_i$ is not a join for all $i$, then $\ff(G)=\{ \mc{G}(\Delta_i)^h\}_{i=1, \dots, k}$.
		We note that the size of this collection is bounded by the clique number of $\Gamma$.
		
		Given a subgroup $H$ of $\mathcal{G}$ there is a unique minimal subgraph product $G$ such that $H\leq G$. We write $\env{H}$ to denote the subgraph product $G$ above and extend the notation of factors to general subgroups $H$ of $\G$ by defining $\ff(H)=\ff(\env{H})$. We call each element of $\ff(H)$ a \emph{factor} of $H$. Notice that in this case $H$ is not the direct product of its factors but rather a subgroup of it.
		
		Given a tuple $(h_{1},\cdots, h_{k})$ of elements of $\G$ we write $\ff(h_{1},\cdots, h_{k})$ as an abbreviation for $\ff(\subg{h_{1},\cdots h_{k}})$. 
		
		We say that $H$ is \emph{directly indecomposable} if $|\ff(H)|=1$.
	\end{definition}
	
	\begin{remark}\label{rem:factors}
		Note that if $H$ and $G$ are subgroups such that  $H<G$, then for all $K\in \ff(H)$, there is a unique $K' \in \ff(G)$ such that $K\leq K'$. More precisely, if $K\in \ff(H)$, $K' \in \ff(G)$ and $K \not \subseteq (K')^\perp$, then $K\leq K'$.
	\end{remark}
	
	\begin{definition}[Orthogonality]
		We say that two subgraph products $G_{1}=\mathcal{G}(\Delta)^{h}$ and $G_{2}=\mathcal{G}(\Delta')^{h'}$ are \emph{orthogonal} if there exists $g\in \mc{G}$ such that $G_1^g = \G(\Delta_1)$, $G_2^g =\G(\Delta_2)$, and $\Delta_1 \subseteq \link{\Delta_2}$ (and so $\Delta_2 \subseteq \link{\Delta_1}$). We denote this by $G_{1}\perp G_{2}$.
		
		It is easy to see that for any subgraph product $G=\mc{G}(\Delta)^h\leq\G$ there is a maximal subgraph product orthogonal to $G$, namely $\mc{G}(\link{\Delta})^h$. We denote this subgraph group by $G^{\perp}$. As above, for a general subgroup $H\leq\G$ we let
		$H^{\perp}=\env{H}^{\perp}$.
	\end{definition}

	\begin{example}
		Consider the following right-angled Artin group:
		
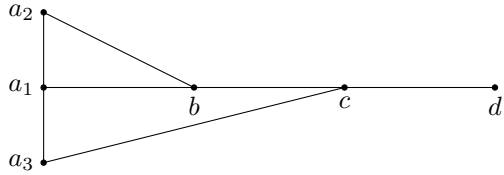
\begin{figure}[H]
			\begin{center}
				\begin{tikzpicture}
					\draw (1,0) -- (7,0);
					\filldraw (1,0) circle (1pt)  node[align=left, left] {$a_1$};
					\filldraw (3,0) circle (1pt)  node[align=left, below] {$b$};
					\filldraw (5,0) circle (1pt)  node[align=left, below] {$c$};
					\filldraw (7,0) circle (1pt)  node[align=left, below] {$d$};
					\filldraw (1,1) circle (1pt)  node[align=left, left] {$a_2$};
					\filldraw (1,-1) circle (1pt)  node[align=left, left] {$a_3$};
					\draw (3,0) -- (1,1);
					\draw (1,0) -- (1,1);
					\draw (1,0) -- (1,-1);
					\draw (1,-1) -- (5,0);
				\end{tikzpicture}
			\end{center}
		\end{figure}

		Let $H= \langle (a_1a_2)^c \rangle$ and $G= \langle (a_1a_2)^c, a_3, b \rangle$, hence $H<G$. Then the smallest subgraph group that contains $H$ is $\supp(H)= \langle a_1, a_2 \rangle^c$; $H$ has two factors: $\ff(H)=\{ \langle a_1\rangle^c, \langle a_2 \rangle ^c\}$ and $H^\perp = \langle b^c \rangle=\langle b \rangle$.
		
		Similarly, $\supp(G)= \langle a_1, a_2, a_3, b \rangle^c= \langle a_1^c, a_2^c, a_3, b \rangle$, $\ff(G)=\{ \langle a_2, a_3\rangle^c, \langle a_1\rangle^c, \langle b \rangle\}$   and $G^\perp = \{1\}$.
	\end{example}
	
	Orthogonality plays an obvious role in the description of centralizers of elements in $\G$.

	\begin{lemma}[see \cite{Barkauskas}]
		\label{l: description of centralizers}Let $g\in\mc{G}$. Then the centralizer of $g$ in $\G$ is isomorphic to the direct product of following groups:
		\begin{itemize}
			\item $\subg{g}^{\perp}$
			\item $C_{H}(\pi^{H}(g))$ for each singular  $H\in\ff(g)$
			\item $\subg{\pi^{H}(g)}$ for each non-singular $H\in\ff(g)$.
		\end{itemize}
	\end{lemma}
	
	\begin{definition}[Retractions]
		For each subgraph product of the form $G=\G(\Delta)$ there is a natural retraction $\pi^{G}: \mc{G}\onto G$ which sends to $1$ all elements of the vertex groups $G_v$ with $v \in \Gamma\setminus \Delta$. Given a subgraph product $G=\G(\Delta)^h$ there is also a retraction $\pi^{G}: \mc{G}\onto G$ where $\pi^G(x) = h^{-1}\pi^{\G(\Delta)}(x)h$ for all $x \in \mc{G}$, which is well-defined up to post-composition of an inner er automorphism of $G$. We will work as if one such representative has been fixed. This can be done in such a way that whenever a subgraph product is decomposable, the projections to sub-direct products of its set of factors are compatible in the sense that $\pi^{H_{1}}\circ\pi^{H_{2}}=\pi^{H_{1}}$ for two such sums $H_{1}\leq H_{2}$. 
	\end{definition}
	
	\begin{notation}
		We will often use the superscript $G$ to abbreviate the application of $\pi^{G}$ to elements, tuples and subgroups of $G$. That is, we will write $g^{G}$ instead of $\pi^{G}(g)$ and so on.
	\end{notation}
	
	\subsection{Logic}\label{s: logic}
	
	We assume that the reader is familiar with the basic notions of first-order logic, see \cite{Marker} as a basic reference. As it is customary in model theory, we will use unadorned lowercase letters to refer to either elements or tuples of elements of a structure or variables. In case a letter such as $x$ stands for a tuple of variables $(x_1,x_2,\dots, x_m)$, an expression such as $\forall x$ within a formula will serve as an abbreviation for $\forall x_1\,\dots, \forall x_m$.
	
	\newcommand{\ms}[3]{\mathcal{B}(#1,\{#2_{v}\})}
	
	\paragraph{\textbf{Languages.}} As usual we fix the language of groups to be  $(\cdot^{(2)},inv^{(1)},1^{(0)})$, where $\cdot$ stands for group multiplication, $inv$ for inversion of elements, and $1$  the identity element. The language of graphs is set to be $(E^{(2)})$, with $E$ standing for adjacency.
	
	Consider further the  language:
	$$
	\mathcal{L}_{B}=\{F^{(2)},\odot^{(3)},1^{(0)}\}.
	$$

	We will interpret any group-labelled graph  $(\Gamma,\gvv)$ as an $\mc{L}_{B}$-structure as follows. We convene that for any  $u,v\in \Gamma$ with $u\neq v$ one has $G_u\cap G_v=\{1\}$. Then we write $\mathcal{B}=\ms{\Gamma}{G}{V}$   to be the structure in the language $\mathcal{L}_{B}$ with universe $\bigcup_{v\in \Gamma} G_v$ and 
	\begin{align*}
		&F(\mc{B})= \bigcup_{(u,v)\in E(\Gamma)} (G_w\backslash \{1\}) \times (G_v\backslash \{1\}),\quad \quad &&\odot(\mc{B}) = \bigcup_{v\in \Gamma} \odot_v,\\
	\end{align*}
	where $\odot_v$ denotes the graph of the group operation of $G_v$. 
	
	\paragraph{\textbf{Uniform interpretability.}} For any class $\textbf C$ of graph products and any assignment to any $\G\in\textbf C$ of some
	family $\mathcal{A}(\G)\subseteq \mathcal{P}(\G)^{k}$ of tuples of subsets of $\G$, we say that $\mathcal{A}$ is \emph{uniformly definable in} $\textbf C$ if there exists a finite collection of $k$-tuples of formulae  $\{(\omega^i_{j}(x,y))_{j=1}^{k}\mid i\in I\}$ on a variable $x$ and a tuple of variables $y$, such that for all $\G \in \textbf C$ we have $\mathcal{A}(\G)=\{(\omega_{j}^{i}(\G,b))_{j=1}^{k}\,|\,b\in\G^{|y|},i\in I\}$, where for an $|y|+1$ variable formula $\theta(x,y)$ and $b\in \G^{|y|}$, we let $\theta(\G,b)$ be the set $\{g\in\G\,|\,\G\models \theta(g,b)\}$.

	Similarly, given an assignment of some structure $\mathcal{M}(\G)$ to each $\G\in\textbf C$ we say that
	$\mathcal{M}$ is \emph{uniformly interpretable} in $\textbf C$ if there is a collection of formulas in the language of groups
	yielding an interpretation of $\mathcal{M}(C)$ in $\G$ for any $\G\in\textbf C$.
	
	\begin{definition}[Classes of graph products] \label{d: K_N_and_C}
		
		We will denote by $\textbf{AP}_{n}$ and $\textbf{K}_{n}$ respectively the classes of graphs satisfying property $AP_{n}$ and $K_{n}$ respectively (see definitions \ref{defn:propDn} and \ref{defn:propKn}) and by $\textbf{F}_{n}$ the class of graphs with at most $n$ vertices.
		Given a family $\Phi$ of sentences in the language of groups we denote by $\textbf{C}_{\Phi}$  the class of groups which satisfy at least one of the sentences in $\Phi$.
		
		Fix some axiomatization $\{\phi_{j}\}_{j=1}^{\infty}$ of the theory of $\mathbb{Z}$  and $\{\phi_{j}'\}_{j=1}^{\infty}$ of the theory of the infinite dihedral group $D_{\infty}$ such that $\phi_{j}$ implies $\phi_{i}$ for $i<j$ and similarly for $\{\phi'_{j}\}$. Then we will write
		$\textbf{C}^{r}_{\Phi}$ to indicate the subclass of $\textbf{C}_{\Phi}$ consisting of groups which are either elementarily equivalent to $\mathbb{Z}$  (respectively $D_{\infty}$) or satisfy $\neg\phi_{r}$ (respectively $\neg\phi'_{r}$). 
		
		It will be useful to specify classes of graph products as follows: given a class  of graphs $\textbf{Ga}$ and a class of groups $\textbf{Go}$ we denote $\mc{G}(\textbf{Ga},\textbf{Go})$  the class of groups formed by all graph products of groups of the form $\mc{G}(\Gamma, \gvv)$ with $\Gamma\in \textbf{Ga}$ and $G_v\in \textbf{Go}$ for all $v\in \Gamma$.
	\end{definition}
	
	\subsection{Actions on trees}\label{sec:actions on trees}
	
	In this subsection we fix some terminology and provide some basic results regarding isometries and group actions on trees. The reader is referred to Chapter 3 of \cite{chiswell2001introduction} for a rigorous and detailed presentation of the subject at hand. Throughout the whole paper	$T$ denotes a simplicial tree with distance metric $d$.
	
	\begin{definition}[Basic definitions] \label{defn: basics actions on trees}
		An isometry $g$ of a simplicial tree $T$ that does not invert edges is either \emph{elliptic} if it fixes a vertex (also called \emph{point}), or \emph{hyperbolic} otherwise. In the first case there exists a unique maximal subtree $A(g)$ of $T$ that is fixed point-wise by $g$, and $g$ may be seen as `rotating' around $A(g)$. In the second case there exists an infinite line $A(g)$ that is preserved by $g$, and $g$ acts on $T$ as a translation along $A(g)$. Regardless of whether $g$ is elliptic or hyperbolic we call $A(g)$ the \emph{axis} of $g$. The \emph{translation length} of an isometry $g$, denoted $tl(g)$, is  defined as the infimum of the displacements by $g$ of a point, i.e.\ $tl(g)=\inf_{v\in T} d(v, gv)$.
		
		Given a hyperbolic element $h\in G$, denote the collection of all elements $g\in G$ preserving $A(h)$ set-wise by $E(h)$ and by $K(h)$ its point-wise stabilizer.
		We further partition $E(h)$ into the set $E^{+}(h)$ of all elements of $E(h)$ that preserve the orientation of $A(h)$ and the set $E^{-}(h)$ of those that swap it.
		
		Two isometries $g, h$ are said to intersect \emph{coherently} if $A(g) \cap A(h) \neq \emptyset$ and if both are hyperbolic then the translation direction of $g$ and $h$ coincides on $A(g) \cap A(h)$ (see page 100 in \cite{chiswell2001introduction}).
	\end{definition}

\begin{definition}[Weakly stable]
		We say that the hyperbolic element $h$ is \emph{weakly $\lambda$-stable} (or just \emph{weakly} \emph{stable}) if for all $g\in G$, if $|A(h)\cap gA(h)|> \lambda \cdot tl(h)$, then $g\in E(h)$.
		
		We say that the hyperbolic element $h$ is ($\lambda$-)\emph{stable} if it is weakly ($\lambda$-)stable and $K(h)=\{1\}$.
	\end{definition}
	
	\begin{remark}\label{r: stability and powers}
		The following observations will be used later on:
		\begin{itemize}
			\item If $\lambda < \mu$, then any $\lambda$-stable element is $\mu$-stable.
			
			\item  For any $n>0$, if a hyperbolic element $h$ is (weakly) $\lambda$-stable, then $h^{n}$ is (weakly) $\frac{\lambda}{n}$-stable.
			
			\item If $G$ acts on a simplicial tree $T$ with trivial edge stabilizers, then every hyperbolic element of $G$ is
			stable.
		\end{itemize}
	\end{remark}

	\begin{lemma}[see Lemma 2.11 in \cite{casals2019positive}]
		\label{l: forms of stability}Let $G$ be a group acting on a simplicial tree $T$. Consider the following three conditions, parametrized by $\lambda>0$ on a hyperbolic element $h$.
		\begin{enumerate}
			\item[$FS(\lambda)$] If an element fixes a subsegment of $A(h)$ of length greater than $\lambda tl(h)$, then it has to fix the whole $A(h)$.
			\item[$AI(\lambda)$] Given another hyperbolic element $g$ with $tl(g)\leq tl(h)$, then
			either $D(g,h)\leq\lambda tl(h)$ or $g\in E(h)$.
			\item[$WS(\lambda)$] $h$ is weakly $\lambda$-stable.
		\end{enumerate}
		The following implications hold:
		\begin{align*}
			FS(\lambda)&\Rightarrow AI(\lambda+2), WS(\lambda+2)\\
			WS(\lambda)&\Rightarrow FS(\lambda).
		\end{align*}
	\end{lemma}
	
	\begin{lemma}
		\label{l: one basepoint}Let $G$ be a group acting on a simplicial tree and assume we are given elements $c_{1},c_{2},\dots, c_{k}$ of
		$G$ and $K>0$ with the property that for all $1\leq j<j'\leq k$ we have $tl(c_{j}),tl(c_{j'}),tl(c_{j}c_{j'})\leq K$.
		Then there is a vertex $*\in T$ such that $d(c_{j}*,*)\leq 2K$ for all $1\leq j\leq k$.
	\end{lemma}
	
	\begin{cor}[{Lemma 2.8} of \cite{casals2019positive}]
		\label{c: irreducibility of powers}Let $p$ and $q$ be coprime integers and assume we are given a two-generated group $\subg{g,h}$ acting irreducibly
		on a simplicial tree $T$. Then there is $(s,t)\in\{p,q\}^{2}$ such that the subgroup $\subg{g^{s},h^{t}}$ also acts irreducibly on $T$.
	\end{cor}
	
	\begin{lemma}
		\label{l: hyperbolic from stable} Let $G$ be a group acting on a simplicial tree $T$ and $g,h\in G$, where $h$ is $\frac{1}{3}$-stable, $g\nin E^{-}(h)$ and $3tl(g)\leq tl(h)$. Then $gh$ is hyperbolic.
	\end{lemma}
	\begin{proof}
		Let $A=A(g)$ and $A'=A(h)$. Suppose towards contradiction that $gh$ is elliptic. Let $*$ be a vertex fixed by $gh$ and let $*'=h\cdot*$. The segment $[*,*']$ decomposes in two different ways as a concatenation of (possibly degenerate) segments.
		Firstly, as $J_{0}\cup J_{1}\cup J_{2}$ where $J_{2}$ is mapped onto $J_{0}$ by the action of $g$ and $J_{1}\subset A$ has length $tl(g)$.
		Secondly, as $J'_{0}\cup J'_{1}\cup J'_{2}$, where $J'_{0}$ is mapped onto
		$J'_{2}$ by $h$ and $J'_{1}\subset A'$ has length $tl(h)$. It follows that a segment of length at least $\frac{1}{3}tl(h)$ of $J'_{1}$ is contained in $J_{1}$ and sent by $g$ into $J_{1}$ with a reversal of orientation. However, this means that $g\in E^{-}(h)$, against our hypothesis.
	\end{proof}
	
	We are interested in the different ways a group $G$ may act by isometries on a tree $T$. 
	
	\begin{definition}[Abelian, dihedral and irreducible actions]\label{d: types_of_group_actions}
		Let $G$ be a group acting on a tree $T$ by isometries. Such an action is called  \emph{abelian} if $tl(gh)\leq tl(g) + tl(h)$ for all $g,h\in G$. It is called \emph{dihedral} if it is not abelian and $tl(gh)\leq tl(g) + tl(h)$ for all hyperbolic pair of elements $g, h\in G$. Finally, the action is said to be \emph{irreducible} if it is neither abelian nor dihedral. We refer to \cite[Chapter 3]{chiswell2001introduction} for further details and alternative characterisations of these notions.
	\end{definition}
	
	We will also use the following notion, which we introduced in \cite{casals2019positive}.
	
	\begin{definition}[Acylindrical pair]\label{d: acylindrical_pair}
		Let $G$ be a group acting by isometries on a tree. We will call a pair of elements $g,h$ \emph{acylindrical} if
		\begin{itemize}
			\item $\subg{g,h}$ acts irreducibly on $T$,
			\item $|A(g)\cap A(h)|\le \frac{1}{2}\max\{tl(g),tl(h)\}$ and
			\item if  $g$ (resp. $h$) is hyperbolic, then $g$ (resp. $h$) is  weakly $\frac{1}{3}$-stable.
		\end{itemize}
	\end{definition}
	
	\section{Graph products, property $AP_{n}$ and actions on trees}\label{sec:factorstrees}
	
	In this section we study the action of graph products on trees and classify the type of actions of finitely generated subgroups. More precisely, we define a natural action of a graph product on a tree, see Definition \ref{defn:tree associated to a vertex} and show that given a finitely generated subgroup $H=\langle h_1, \dots, h_n\rangle$ that acts non-elliptically on the tree, the kernel of the action on a minimal subtree is precisely the pointwise stabilizer of a ball whose radius is determined by the property $AP_n$ (see Definition \ref{defn:propDn}) of the graph. Furthermore, if the kernel is nontrivial, the minimal subgraph product that contains the kernel and the subgroup $H$, see Definition \ref{d: minimal_graph_subg_and_factors}, has a nontrivial decomposition as a direct product, see Corollary \ref{c: raw acylindricity}. Using this result we are able to classify in Corollary \ref{c: types of actions} the action of a finitely generated subgroup $H$ of a graph product, namely, if the minimal subgraph product that contains $H$ is directly indecomposable and not a vertex group, then $H$ is either cyclic, dihedral, or it acts irreducibly on the tree.
	
	\bigskip
	
	The definition of the extension graph of a right-angled Artin group was introduced by Kim and Koberda in analogy to the curve graph of a mapping class group. Here we extend that definition to graph products.
	
	\begin{definition}[Extension graph]\label{d: extension_graph}
		
		Let $\mathcal{G}=\mathcal{G}(\Gamma,\gvv)$ be a graph product of groups.  The extension graph $\Gamma^{e}=(V^{e},E^{e})$ of $\G$ is the graph whose vertices are in bijective correspondence with the set of singular subgraph products  of $\G$, that is each $v^e\in \Gamma^e$ is in bijective correspondence with a subgroup of the form $G_u^h$, for some $u\in \Gamma$  and some $h\in \G$; the edge relationship is given by orthogonality: two vertices $v^e,u^e\in \Gamma^e$ are adjacent in $\Gamma^e$ if and only if their corresponding subgroups commute in $\mc{G}$.
		
		There is a natural surjective simplicial quotient map $\pi^{e}:\Gamma^{e}\twoheadrightarrow\Gamma$: if $v^e\in \Gamma^e$ has associated the group $G_u^h$, then $\pi^e(v^e)=u$.
		Given $v^e\in \Gamma^e$ we will use the notation $H_{v^e}$ to refer to the group associated to $v^e$, that is,
		$H_{v^e}=H_{\pi^{e}(v)}^{h}$ for some $h\in\G$.

	\end{definition}

	\begin{prop} \label{prop:extensionAP}
		If $\Gamma$ satisfies property $AP_{n}$, so does $\Gamma^{e}$.
	\end{prop}
	\begin{proof}
		By \cite[Lemma 3.1]{KimKoberda}, the graph $\Gamma^e$ is obtained from $\Gamma$ by a subsequent process of doubling along the stars of vertices. More precisely $\Gamma^e = \bigcup_{i\in \mathbb N}\Gamma_i$ where $\Gamma_i$ is obtained from $\Gamma_{i-1}$ by doubling over a star of a vertex $v\in \Gamma_{i-1}$, i.e. $\Gamma_i$ is obtained from two formal copies of $\Gamma_{i-1}$, say $\Gamma_{i-1}$ and $\Gamma_{i-1}'$, and identifying the $star(v)$ in both graphs. Notice that from the definition we have that for all $w \in \Gamma_{i-1}\smallsetminus star(v)$ and for all $w'\in \Gamma_{i-1}'\smallsetminus star(v)$, $(w,w')\notin E(\Gamma_i)$.
		
		Therefore, it suffices to show that if $\Gamma$ is a graph with property $AP_n$, $v\in \Gamma$ and $\Gamma_1$ is obtained from $\Gamma$ by doubling along the star of $v$, then $\Gamma_1 = \Gamma\cup_{\st(v)}\Gamma'$ also has property $AP_n$.
		
		Let $p$ be a maximal path in the complement graph $\bar\Gamma_1$. If $p$ is contained in $\bar \Gamma$ or $\bar\Gamma'$, then the statement follows by induction. Otherwise, the statement follows since, as we noticed above, for any $w\in \Gamma\smallsetminus star(v)$ and $w'\in \Gamma'\smallsetminus star(v)$, there exists an edge $(w,w')\in E(\bar \Gamma_1)$, i.e. the distance between $w$ and $w'$ is one in the complement graph $\Gamma_1$.
	\end{proof}
	
	\begin{lemma}
		\label{l: blob lemma}
		Let $\G=\G(\Gamma,\gvv)$ be a graph product and assume that $\Gamma$ satisfies 	property $AP_{n}$.
		Suppose that we are given a collection $\mathcal{A}$ of directly indecomposable subgraph products, none of which is contained in the other.
		
		\newcommand{\ef}[0]{E_{\mathcal{F}}}
		Let $\ef$ be the following edge relationship in $\mc{A}$: $(H,H')\in\ef$ if and only if
		$H$ and $H'$ are not orthogonal.
		Then the diameter of any $\ef$-connected component of $\mathcal{A}$ is bounded by $n$.
	\end{lemma} 
	\begin{proof}
		Assume towards contradiction that there is an $E_{\mathcal{F}}$-connected component of diameter $n+1$. Then it is possible to find a sequence $G_{1},G_{2},\dots,G_{n+1}\in\mathcal{A}$ such that for any $1\leq i<j\leq n+1$ the groups $G_{i}$ and $G_{j}$ are orthogonal if and only if $j\ne i+1$. Therefore, for all $v,v'\in \Gamma^e$ such that $H_v < G_i$ and $H_{v'} < G_j$ for $j\ne i+1$ we have that $(v,v') \in E(\Gamma^e)$. Furthermore, as $G_i$ and $G_{i+1}$ are not orthogonal, there exists $H_v< G_i$ and $H_{v'}< G_{i+1}$ for some $v,v'\in \Gamma^e$ such that $(v,v')\notin E(\Gamma^e)$.
		
		Now, the fact that each $G_{i}$ is directly indecomposable implies that any to vertices $v,v'\in \Gamma^e$ such that $H_v, H_{v'}<G_i$ are connected in the complement $\overline{\Gamma^e}$ and all the subgroups associated to the vertices in the path are subgroups of $G_i$. 
		
		From the above observations, we deduce the existence of a sequence $v_{i,1},v_{i,2},\dots, v_{i,n_i+1}$ where the sequence $(v_{i,j})_{j=1}^{n_{i}}$, $H_{v_{i,j}} < G_i$, induces a path (so $v_{i, j+1}$ is adjacent to $v_{i,j}$ for all $j$) in the complement graph $\bar{\Gamma^{e}}$ and the end vertex $v_{i,n_i+1}$ either coincides or is joined by an edge of $\overline{\Gamma^e}$ to the initial vertex $v_{i+1 ,1}$. The concatenation of all paths in the sequence and additional edges connecting them yields a path in $\overline{\Gamma}^{e}$ of length at least $n+1$. Therefore $\Gamma^e$ would satisfy property $AP_{n+1}$ and so by Proposition \ref{prop:extensionAP}, so would $\Gamma$, contradicting our assumptions.
	\end{proof}
	
	\begin{obs}
		\label{o: generation}Given $C=\subg{c_{1},c_{2},\dots, c_{k}}\leq\G$, any $\HH\in\ff(C)$ is equal to $\G(\subg{\HH_{j}}_{j\in J})$ where $\{\HH_{j}\}_{j\in J}$ is the collection of all members of $\bigcup_{1\leq i\leq k}\ff(c_{i})$ contained in $\HH$.
	\end{obs}
	
	\begin{definition}[Tree associated to a vertex]\label{defn:tree associated to a vertex}
		
		Let $\Gamma$ be a graph. Given $v\in \Gamma$ such that $v^\perp\cup \{v\}\neq \Gamma$, the graph product $\mathcal{G}=\mathcal{G}(\Gamma,\gvv)$ decomposes as an amalgamated product:
		$$
		\left(G_{v}\times\mathcal{G}(v^\perp)\right)\frp_{\G(v^\perp)}\mathcal{G}(\Gamma \setminus\{v\}).
		$$
		We denote the Bass-Serre tree associated to the above decomposition as $T_{v}(\Gamma)$, or simply $T_v$ when no ambiguity ensues, and refer to it as the \emph{tree associated to the vertex $v$}.
	\end{definition}
	
	\begin{remark}\label{rem:ellTv}
		Observe that if $v^\perp\cup\{v\}\neq V(\Gamma)$, then any group elliptic in $T_{v}$ is of smaller complexity.
		
		Notice also that given any $\Delta\subset\Gamma$ and $v\in\Delta$, the action of $\G(\Delta)$ on $T_{v}(\Delta)$ is equivariantly isomorphic to the action
		of $\G(\Delta)$ on its minimal tree in $T_{v}(\Gamma)$.
	\end{remark}

	\begin{obs}
		\label{o: decomposibility}Given any decomposition of $\Gamma$ as $\Gamma_{0}\oplus\Gamma_{1}$ where $v\in \Gamma_{1}$, the subgroup
		$\G(\Gamma_{0})$ is in the Kernel of the action on $T_{v}$. For any $g\in\G$, the image of  $g$ by the projection $\pi^{\mc{G}(\Gamma_1)}$ acts
		in the same way as $g$ on $T_{v}$.
	\end{obs}
	\begin{proof}
		Since $V(\Gamma_0) \subset V(\Gamma)\setminus \{v\}$ and $V(\Gamma_0) \subseteq v^\perp$, we have that $\mc{G}(\Gamma_0)$ is contained in the stabilisers of the two vertices in $\mathcal G(\Gamma)\backslash T_v$. Since $\mc{G}(\Gamma_0)$ is normal in $\mc{G}(\Gamma)\simeq \G(\Gamma_0)\times \G(\Gamma_1)$, it follows that $\mc{G}(\Gamma_0)$ belongs to the stabiliser of all vertices in $T_v$ and so in the kernel of the action. The second claim is an immediate consequence of the first one.
	\end{proof}
	
	The main technical result of this section is a weak form of malnormality for canonical subgroups of graph products. We will start with an easy observation.
	\begin{lemma}\label{lem:aux1}
		Let $\Delta$ be an induced subgraph of $\Gamma$. Suppose we are given $C< \mathcal G(\Delta)$ and $g_{1},g_{2},\dots, g_{k}\in \mathcal G$. Then either $\subg{C,g_{1},g_{2},\dots, g_{k}}$ is of smaller complexity   or
		there is $1\leq i\leq k$ such that $C\cap C^{g_{i}}$ can be conjugated into $\G(\Delta\cap(\bigcap_{v\in \Gamma\setminus \Delta}v^\perp))$.
		
	\end{lemma}
	\begin{proof}
		Notice that if $\Delta= \Gamma$, then the second alternative holds. Let us assume that $\Delta \ne \Gamma$.
		
		If $v\in \Gamma\setminus \Delta$ and $\Gamma= v\cup v^\perp$, then, since $\G(\Gamma)= G_v\times \G(\Gamma\setminus \{v\})$, for all $1\le i\le k$ we have that $C^{g_i} = C^{h_i}$ for some $h_i\in \G(\Gamma\setminus \{v\})$. Therefore, we have that $C\cap C^{g_i}= C\cap C^{h_i} < \G(\Gamma\setminus \{v\})= \G(v^\perp)$.
		
		Let $v\in \Gamma\setminus \Delta$ be such that $\Gamma\ne v\cup v^\perp$. Consider the associated tree $T_{v}$. The group $C$ fixes a vertex $x$ of $T_{v}$, since $C\leq \mc{G}(\Delta) \leq \mc{G}(\Gamma\setminus\{v\}) $. If each of the $g_{i}$ fixes this vertex as well, then $\subg{C,g_{1},g_{2},\dots,g_{k}}$ is of smaller complexity  by Remark \ref{rem:ellTv}. If $g_{i}$ does not fix $x$ for some $i=1, \dots, k$, then $C\cap C^{g_{i}}$ must fix the segment $[x,g_{i}^{-1}x]$ that starts at $x$ and ends at $g_{i}^{-1}x$. In particular, this means that $C\cap C^{g_{i}}$ can be conjugated into the edge group $\G(v^\perp)$. 
		
		The result then follows from an iterated application of Lemma \ref{l: intersection_subgroups}.
	\end{proof}
	
	\begin{notation}[Positive ball]
		Given elements $g_{1},g_{2},\dots, g_{k}$ of a group, denote by $\mathcal{B}_{n}(g_{1}, \dots, g_k)$ the collection of all the elements of the form $g_{i_{1}}g_{i_{2}}\dots g_{i_{\ell}}$ with $0\leq \ell\leq n$ and $i_{j}\in\{1,2,\dots, k\}$ for all $1\leq j\leq \ell$.
	\end{notation}

	\begin{thm}
		\label{l: conjugate intersection}
		Let $\mathcal{G}=\mathcal{G}(\Gamma,\gvv)$ be a graph product with $\Gamma$ satisfying property $AP_n$. Let $g_{1},g_{2},\dots, g_{k}$ be elements of $\G$, let $C=\G(\Delta_{0})\leq\G$ be a subgraph product and let $E=\bigcap_{h\in\mathcal{B}_{n}(g_{1},g_{2},\dots, g_{k})}C^{h}$.
		
		Then, up to conjugation,  $\supp(\subg{E,g_{1},g_{2},\dots,g_{k}})$  is of the form $\G(\Delta)$ for some graph $\Delta$ which satisfies that
		$\Delta = \Delta_{1}\oplus\Delta_{2}$ for some graphs $\Delta_1$ and $\Delta_2$ and $E\leq\G(\Delta_{1})\leq C$.
	\end{thm}
	
	\begin{proof}
		
		Notice that we can assume that $\supp(\subg{E,g_{1},g_{2},\dots, g_{k}})$ is not of smaller complexity. Indeed, if $H=\supp(\subg{E,g_{1},g_{2},\dots, g_{k}})$ is of less complexity, we can consider the graph product $\mc{G}'=H$ and $C'=C \cap H$. In this case $E$ remains the same. The result for this setting implies the original statement and so without loss of generality we assume that $H$ is not of of smaller complexity  and so $\Gamma= \Delta$.
		
		Furthermore, we can assume that $\Delta_{0}\neq\Gamma$. Indeed if $\Delta_0= \Gamma$, then $E$ is $\mc{G}$ and $\Delta=\Delta_1=\Gamma$ and so $E=\mc{G}(\Gamma)=C$.
		
		Let $C_{i}=\bigcap_{h\in\mathcal{B}_{i}(\bar{g})}C^{h}$, so that $C_{0}=C$ and $C_{n}=E$. Assume by induction that $C_i < G(\Delta_i)$. By Lemma \ref{lem:aux1}, we have that $C_{i+1}$ can be conjugated into $\G(\Delta_{i+1})$ where 
		
		$$
		\Delta_{i+1}=\Delta_i \cap (\bigcap_{v\in \Delta_{i-1}\setminus\Delta_i}v^\perp) = \Delta_i \cap (\bigcap_{v\in \Gamma \setminus \Delta_i}v^\perp).
		$$
		
		If for some $i$ we have that $\Delta_{i+1} =\Delta_{i}$, then, from the definition of $\Delta_{i+1}$ we have that $\Gamma \cong \Delta_{i} \oplus \Delta_{i}^\perp$ with $E$ conjugate into $\G(\Delta_{i}) < \G(\Delta_0)$ and so the result holds.

		If for all $0\le i\le n$ we have that $\Delta_{i+1} \lneq \Delta_i$, then we obtain a proper descending chain of graphs $\Delta_i$ of length $n+1$. Hence, it is possible to choose a sequence of vertices
		$v_{0},v_{1},\dots, v_{n}$ of $\Gamma$ where $v_{0}\in\Gamma\setminus\Delta_{0}$ and  $v_{i}\in\Delta_{i}\setminus\Delta_{i-1}$ for $i\leq 1$.
		However, from the definition of the graphs $\Delta_i$, such a sequence is a path in the complement graph $\bar \Gamma$  and thus contradicts the fact that $\Gamma$ satisfies property $AP_{n}$.
	\end{proof}
	
	The previous theorem has strong implications on the way the graph product $\G$ acts on the simplicial trees $T_{v}$.
	
	\begin{cor}
		\label{c: raw acylindricity}Let $\G=\G(\Gamma,\gvv)$ be a graph product of groups satisfying property $AP_{n}$ and let $g_1, \dots, g_k\in \G$. Assume  we are given
		$v\in \Gamma$, $*$ a vertex of $T_{v}$ stabilized by a conjugate of $\G(V\setminus\{v\})$
		and let $E$ be the point-wise stabilizer of $\mathcal{B}_n(h_{1},h_{2},\dots, h_{k})\cdot\{*\}$.
		If we let $\Delta$ be a subgraph of $\Gamma$ such that $\supp(E,h_{1},h_{2},\dots, h_{k}))=\G(\Delta)^h$ for some $h\in \mc{G}$, then there exists a decomposition $\Delta=\Delta_{1}\oplus\Delta_{2}$ such that $E=\G(\Delta_{1})\leq Stab(*)$.
		In particular, $E$ fixes the whole tree spanned by the orbit of $*$ under the action by $\G(\Delta)$.
		
	\end{cor}
	\begin{proof}
		The result follows from Lemma \ref{l: conjugate intersection} by letting $C=Stab(*)$. The last claim follows from Observation \ref{o: decomposibility} applied to $\G(\Delta)$.
	\end{proof}
	
	\begin{cor}
		\label{c: stable axis}
		Assume that the graph $\Gamma$ satisfies property $AP_{n}$ and let $\G=\G(\Gamma,\gvv)$ be a graph product. Then for every $v\in \Gamma$, any element $h\in\G$ which is hyperbolic with respect to the action of $\mc{G}$ on $T_{v}$ is also weakly $(n+2)$-stable with respect to the action of $\mc{G}$ on $T_v$.
		
		If, additionally, $\subg{h}$ is not of smaller complexity and $\Gamma$ is directly indecomposable, then $h$ is $(n+3)$-stable with respect to the action of $\mc{G}$ on $T_v$.
	\end{cor}
	\begin{proof}
		
		The first claim is  a direct consequence of Lemma \ref{l: forms of stability} via property $FS(n)$, which follows directly from Corollary \ref{c: raw acylindricity}.
		The second claim is a consequence of the fact that the kernel $K$ of the action on $A(h)$ of its set-wise stabilizer is trivial. Indeed, if $\subg{h}$ is assumed to be not of smaller complexity, then $K\neq\{1\}$ implies that the graph
		$\Gamma$ admits a nontrivial decomposition $\Gamma=\Gamma_1\oplus\Gamma_2$, contradicting the assumption.
	\end{proof}

	\begin{cor}
		\label{c: types of actions}
		Assume that the graph $\Gamma$ satisfies property $AP_{n}$ and let $\G=\G(\Gamma,\gvv)$ be a graph product. Let $H\leq\G$ be a finitely generated group and let $\mc G'=\mathcal G(\Delta)^h\in \ff(H)$ be a non-singular factor.
		
		Then one of the three situations occurs {\rm(}see {\rm Definition \ref{d: types_of_group_actions})}:
		\begin{enumerate}
			\item \label{sit_a}$\pi^{\mc G'}(H)\cong\Z$ and for all $v\in \Delta$, $\pi^{\mc G'}(H)$ acts on $T_v$ by translations on an axis;
			\item \label{sit_b}$\pi^{\mc G'}(H)\cong D_{\infty}$ and  for all $v\in \Delta$, $\pi^{\mc G'}(H)$ acts on $T_{v}$ dihedrally; or
			\item  \label{sit_c} for all $v\in \Delta$, $\pi^{\mc G'}(H)$ acts irreducibly on $T_{v}$.
			
		\end{enumerate}
	\end{cor}
	\begin{proof}
		Observe that, since $\mc{G}'\in \ff(H)$ is non-singular, from the definition of factor, see Definition \ref{d: minimal_graph_subg_and_factors}, we have that for all $v\in \Delta$, the action of $\pi^{\mc G'}(H)$ on $T_v$ is not elliptic. Since the statement is for trees $T_v$ for $v\in \Delta$ and the subgroup $\pi^{\mc G'}(H)$, we can assume without loss of generality that $\Delta=\Gamma$, $\Gamma$ is directly indecomposable (as $\Delta$ is a non-singular factor), $\pi^{\mc G'}(H)=H$ and $H$ is not of smaller complexity.
		
		In order to prove the statement, it is enough to show that one of the three alternatives holds for one vertex $v\in V$. Indeed,
		if the action of $H$ on $T_{v}$ is irreducible for some vertex $v$, then $H$ must have infinitely many ends and therefore cannot be isomorphic to either $\Z$ or $D_{\infty}$, and hence
		the action on $T_{v'}$ for any any other given $v'\in V$ must also be irreducible.
		
		Assume that the action of $H$ is not irreducible. As $H$ is not of smaller complexity in $T_v$, then $H$ has one or two ends. Corollary \ref{c: stable axis} implies that if the intersection of the axis of two hyperbolic elements has length larger than $(n+3)\min\{tl(g),tl(h)\}$, then the two axis must coincide. Since $H$ is finitely generated, the latter implies that $H$ cannot have just one end and so $H$ must preserve a bi-infinite line. Furthermore, since $\Gamma$ is directly indecomposable and $H$ is not of smaller complexity, we have by Corollary \ref{c: stable axis} that $H$ has a stable hyperbolic element and so the  action of $H$ on this line must have trivial kernel. Therefore, $H$ is either infinite cyclic or infinite dihedral for any $v\in V$.
	\end{proof}
	
	\begin{definition}[Linear, dihedral and irreducible actions on factors]\label{d: types_of_actions_on_subgraph_products}
		Let $H\leq\G$ and without loss of generality let $\G'=\G(\Delta)\in\mathcal{F}(H)$ be a non-singular factor. If the alternative (\ref{sit_a}) in Corollary \ref{c: types of actions} holds, then we say that $H$ is \emph{linear} in $\G'$, or that $H$ acts \emph{linearly} on $\mc{G}'$. If the alternative (\ref{sit_b}) holds, then we say that $H$ is \emph{dihedral} in $\G'$,  or that $H$ acts \emph{dihedrally} on $\mc{G}'$. If the alternative (\ref{sit_c}) holds,  then we say that $H$ is
		\emph{irreducible}, or that $H$ acts \emph{irreducibly} on $\mc{G}'$.
		We say that $H$ is \emph{small} if for any $\G'\in\mathcal{F}(H)$ either $\G'$ is singular or $H$ is linear or dihedral in $\G'$.
	\end{definition}
	
	\bigskip
	We finish this section with two observations about factors of subgroups and the relations between themselves.
	
	\begin{lemma}
		\label{l: linear expansions} Let $h,g\in\G$ and
		let $\hh\in\ff(h)$ be non-singular.   Let $\hh'\in\ff(g,h)$ be such that $\hh \leq \hh'$ {\rm(}such an $\hh'$ always exists due to {\rm Remark \ref{rem:factors})}. Suppose  $\subg{h,g}$ is small in $\hh'$. Then $\hh=\hh'$.
	\end{lemma}
	\begin{proof}
		If $\hh\neq\hh'$, then there is a vertex tree $T$ associated to $\hh'$ in which $\hh$ is elliptic and $\subg{g,h}$ is not. Indeed, otherwise, $\subg{g,h}$ would be elliptic in $T$ and thus of smaller complexity than $\hh'$ contradicting the definition of factor.
		
		Suppose that $\subg{g,h}$ has a small action on $T$. As we noticed, this action is not elliptic. Furthermore, it cannot be cyclic, otherwise $\ff(g,h)=\ff(h)$. Therefore, the action has to be dihedral and so $\subg{g,h}$ acts on a line in $T$. The element $h$ either fixes a vertex or acts hyperbolically on the line. The latter case, implies that $\hh=\hh'$. In the latter case, we have that the projection
		$h^{\hh'}$ of $h$ on $\hh'$ is an involution, which implies that $h^{\hh}$ is an involution as well and therefore $\hh$ is singular, contradicting the hypothesis of the statement.
	\end{proof}
	
	The next lemma relates the factors of an element and its powers.
	
	\begin{lemma}
		\label{l: powers} Let $\mathcal{G}=\mathcal{G}(\Gamma,\gvv)$ be a graph product and $p$ a non-zero integer.
		Then for any $g\in\G$ we have $\mathcal{F}(\subg{g^{p}})\subseteq\mathcal{F}(\subg{g})$ and
		$\mathcal{F}(\subg{g})\setminus\mathcal{F}(\subg{g^{p}})$ is the collection of all singular $\G'\in\mathcal{F}(\subg{g})$ such that the projection of
		$\subg{g}$ to $\mathcal{G}'$ has order $p$.
	\end{lemma}
	\begin{proof}
		If $\G'=\mathcal{G}(\Delta)\in\mathcal{F}(\subg{g})$ is non-singular, then the component of $g$ in $\mathcal{F}$ is hyperbolic in $T_{v}$ with respect to any $v\in \Delta$. But then so is the component of $g^{p}$, which means that $\G'\in\mathcal{F}(\subg{g^{p}})$.
		If $\G'$ is singular, then $\G'\in\mathcal{F}(\subg{g^{p}})$ if and only if the projection of $g$ in $\G'$ does not have trivial $p$-th power.
	\end{proof}

	\section{Generalized positive theory}\label{sec: generalized positive theory}
	
	In this section we first recall the notion of the (nontrivial) positive theory of a group, discuss its connection with formal solutions, and provide some examples of groups with nontrivial positive theory. We then introduce a larger fragment of a first-order theory, which we call the almost positive theory where one allows for some controlled inequalities in the sentences, see Definition \ref{defn:almost positive}. As in the case of positives formulae, we consider almost positive sentences that do not admit a (relative) formal solution and call them non-generic almost positive sentences, see Definition \ref{defn:generic ap}.
	
	\subsection{Positive theory and formal solutions}

	Roughly speaking, a positive formula in the language of groups is a formula in the first-order language of groups that does not use negations. More precisely,
	
	\begin{definition}[Positive theory]\label{defn:positive theory}
		A \emph{positive formula} $\phi(u_{1},u_{2},\cdots, u_{m})$ in the language of groups is equivalent (in the common theory of all groups) to one of the form:
		\begin{align*}
			\phi(u_1, \dots, u_m)\cong \forall x^{1}\exists y^{1}\forall x^{2}\dots\forall x^{m}\exists y^{m}\ \bigvee_{j=1}^{m} \Sigma_{j}(x^{1},\dots, x^{r},y^{1},\dots, y^{r},u_{1},u_{2},\dots, u_{m})=1,
		\end{align*} 
		where $x^i$ and $y^i$ are tuples of variables for all $i$; and where  $\Sigma_{j}=1$ is a system of equations for all $j=1, \dots, m$, i.e.\ $\Sigma_j=1$ is a formula of the form $$
		w_1(x^{1},\ldots, x^{r},y^{1},\ldots, y^{r},u_{1},u_{2},\ldots, u_{m})=1 \wedge  \ldots \wedge w_k(x^{1},\ldots, x^{r},y^{1},\ldots, y^{r},u_{1},u_{2},\ldots, u_{m})=1,
		$$
		where $k\geq 1$ and each $w_i(x^{1},\dots, x^{r},y^{1},\dots, y^{r},u_{1},u_{2},\dots, u_{m})$ is a term in the language of groups.
		
		The \emph{positive theory} is the fragment of the first-order theory in the language of groups  defined by the set of positive sentences (i.e.\ positive formulas without free variables) in the language of groups. Given a group $G$, we denote by $\Th^+(G)$ the set of positive sentences that are true in $G$, and we call $\Th^+(G)$ the \emph{positive theory of $G$}.
	\end{definition}
	
	In \cite{Makanin}, Makanin studied the positive theory of non-abelian free groups, and described the positive theory of such a group in terms of so-called formal solutions. More precisely, he proved the following
	
	\begin{thm}\label{thm:Makanin}Let $F$ be a non-abelian free group. Then,
		\begin{align*}
			F \models \forall x^{1}\exists y^{1}\forall x^{2}\dots\forall x^{m}\exists y^{m}\,\bvee{j=1}{\Sigma_{j}(x^{1},y^{1},x^{2}, \dots, x^{m},y^{m})=1}{k},
		\end{align*}
		where $x^i$, and $y^i$ are tuples of variables for all $1\leq i\leq m$ and $\Sigma_j=1$ is a system of equations for all $1\leq j \leq k$, if and only if there exist a tuple $\alpha=(\alpha^{1},\dots,\alpha^{m})$ where
		$\alpha^{\ell}\in\F(x^{1},x^{2},\dots, x^\ell)^{|y^{\ell}|}$ for all $1 \leq \ell \leq m$ such that for some $1\leq j\leq k$ all words in the tuple
		\begin{align*}
			\Sigma_{j}(x^{1},\alpha^{1}(x^{1}),x^{2},\alpha^{2}(x^{1},x^{2}), \dots,\alpha^{m}(x^{1},x^{2}, \dots, x^{m}))
		\end{align*}
		are trivial as elements in the free group $F(x^1, \dots, x^m)$.
	\end{thm}
	
	The tuple $\alpha$ is called a \emph{formal solution} as it provides a ``witness" to the validity of the sentence in the free group. As a consequence of Makanin's results one has that all the non-abelian free groups have the same positive theory.
	
	Notice that if a group $G$ satisfies a positive sentence, so does any quotient of $G$. From this remark it follows that any positive sentence satisfied by a non-abelian free group is satisfied by any group, that is, the common positive theory of non-abelian free groups is contained in the positive theory of any group. We say that a group $G$ has \emph{nontrivial} positive theory if its positive theory strictly contains the positive theory of a non-abelian free group, that is, $G$ satisfies a positive sentence that  is not satisfied by a non-abelian free group.
	
	In \cite[Corollary 8.1]{casals2019positive}, we studied groups that have the same positive theory as a free group. More precisely we showed that non-virtually  solvable  fundamental  groups of closed,  orientable,  irreducible 3-manifolds; groups acting acylindrically on a tree;  non-solvable generalised Baumslag-Solitar groups;  (almost all)  non-solvable  one-relator groups;  and  graph  products of groups  whose underlying graph is not complete have the same positive theory as a non-abelian free group.
	
	In this paper, we are interested in the other side of the spectrum, that is, in the class of groups that do not have the same positive theory as a free group, or, in other words, groups that have nontrivial positive theory.

	\begin{lemma}\label{ex:groups_non_trivial_positive_theory}
		The positive theory of the following groups is different from the positive theory of a non-abelian free group:
		\begin{enumerate}
			\item
			All verbal groups, that is, groups that satisfy an identity {\rm(}equivalently, groups that belong to a proper group variety{\rm)}. In particular, all finite, abelian, metabelian, nilpotent, solvable, Burnside, Engel, etc.\  groups.

			\item Groups with finite commutator width, that is, groups that satisfy $$
			\forall x_1, \dots, x_{2N+2} \exists y_1, \dots, y_{2N} \, \prod_{i=1, \dots, N+1}[x_{2i-1},x_{2i}]= \prod_{i=1, \dots, N}[y_{2i-1},y_{2i}].
			$$
			
			For instance:
			
			\begin{itemize}
				\item Some branch groups, such as the Grigorchuk groups, see {\rm\cite{LMU}}, and the Gupta-Sidki groups, see {\rm\cite{Groth}}.
				
				\item Finitely generated pro-$p$ groups and profinite groups, see {\rm\cite{Nikolov-Segal}}.
				
				\item Linear algebraic groups, see {\rm\cite{Merz67}}.
				
				\item Thompson's group, see {\rm\cite{CFP}}.
				
				\item Groups with finitely many conjugacy classes, see the examples in {\rm\cite{Osin10}} and {\rm\cite{Minasyan07}}.
			\end{itemize}
			
			\item More generally, groups  $G$ for which there exists a word $w(x_1, \dots, x_n) \in F(x_1, \dots, x_n)$, such that the verbal subgroup $w(G)=\langle \{ w(g_1, \dots, g_n)\mid  g_1, \dots, g_n \in G\}\rangle$ has finite width, i.e.\ there exists an $N$ such that each element in $w(G)$ is the product of at most $N$ elements from  $\{ w(g_1, \dots, g_n)^\varepsilon \mid  g_1, \dots, g_n \in G,\ \varepsilon=\pm 1\}$. For instance, if $w=x^2$, then the symmetric group $S_X$ where $|X|=\aleph_0$ has $w$-verbal width $2$ {\rm(}every element is the product of two squares{\rm)}.
			
			A group $G$ is said to be \emph{verbally elliptic} if $w(G)$ has finite width for any word $w$. The following are examples of verbally elliptic groups. In particular, the following have nontrivial positive theory: finitely generated virtually abelian-by-nilpotent groups, see {\rm\cite{SegalWords}} and references there;  finitely generated solvable groups of class 3, see {\rm\cite{Rhemtulla}}, in particular, metabelian groups; virtually solvable minimax groups; and virtually polycyclic groups, see {\rm\cite{SegalWords}}, where we refer the reader for further examples of verbally elliptic groups.
			
			\item  Boundedly simple groups, i.e. groups for which there is some $m>0$ such that any element can be expressed as a product of at most $m$ conjugates of any given nontrivial element and its inverse, as expressed by:
			$$
			\begin{array}{l}
				\forall x\forall y \,\,x=1\vee \\
				\exists z_{1},z_{2}\dots z_{m}\exists u_{1},u_{2}\dots u_{m}\,\,\bigwedge_{i=1}^{m}(u_{i}=1\vee u_{i}=x\vee u_{i}=x^{-1})\wedge y=u_{1}^{z_{1}}u_{2}^{z_{2}}\cdots u_{m}^{z_{m}}.
			\end{array}
			$$
			The class of boundedly simple groups includes many groups of automorphisms of homogeneous structures.
			\item Given an infinite family of groups satisfying the same nontrivial positive sentence, then the direct product and the direct limit of these groups  satisfies the same nontrivial positive sentence. For instance,  any product of finite simple groups {\rm(}since all elements in a simple group are commutators {\rm \cite{OreConj})}.
			
			\item The class of groups with nontrivial positive theory is closed under taking extensions, see {\rm Lemma \ref{lem:extensionpreservation}}. More precisely, if $N,Q$ are groups with nontrivial positive theory and $G$ is an extension of $N$ by $Q$, then $G$ has nontrivial positive theory. In particular, any extension involving any of the groups mentioned above has nontrivial positive theory, for instance, the lamplighter group.
			
		\end{enumerate}
	\end{lemma}
	
	\subsection{Almost positive theory}
	
	In this section we introduce a type of sentences which we call \emph{non-generic almost positive}. This fragment of the elementary theory contains the nontrivial positive sentences. While it is not as classical and well-known as the positive theory,  it is relevant for our results, as we aim at proving that if in a graph product of groups the vertices satisfy a simple non-generic almost positive sentence, then the (core of the) graph structure is interpretable in the group structure.
	
	Let us now introduce this fragment of the theory.
	
	\medskip
	
	\begin{definition}[Almost positive sentence and  almost positive sentence]\label{defn:almost positive}
		
		We say that a sentence is \emph{almost positive} if it is of the form:
		\begin{align*}
			\phi\equiv\exists z\,(\Theta(z)\neq 1\wedge\psi(z)),
		\end{align*}
		where $\psi(z)$ is a positive formula whose free variables are  $z$ and $\Theta(z)\neq 1$ is a nontrivial system of inequalities, i.e.\ $\Theta(z)\neq 1$ is a formula of the form $w_1(z)\neq 1 \wedge \dots \wedge w_n(z)\neq 1$ where $w_i$ is a word from $\mbb{F}(z)$ for all $i=1, \dots, n$.
		
		Such a sentence is said to be \emph{simple} if $\Theta\neq 1$ is of the form $w(z) \neq 1$ for some $w\in \mbb{F}(z)$.
	\end{definition}
	
	These more general sentences allow to express properties such as having nontrivial center or being boundedly generated.
	
	\begin{example}\
		\begin{itemize}
			\item The property of having nontrivial center can be axiomatized using an almost positive sentence:
			\begin{align*}
				\exists z\,\,z\neq 1 \wedge \forall x\,\,[x,z]=1
			\end{align*}
			
			\item A group $G$ is said to be \emph{boundedly centraliser generated} if there exist  $n\geq 1$ and $z_1, \dots, z_n \in G\setminus\{1\}$ such that $G=C(z_1) \cdots C(z_n)$. This property can be axiomatized by the following almost positive sentence:
			\begin{align*}
				\exists z_{1}\exists z_{2}\cdots\exists z_{k}\, &(\bigwedge_{i=1}^{k}z_{i}\neq 1)\wedge\\
				&\wedge(\forall x\exists y_{1}\exists y_{2}\cdots\exists y_{k}\,(\bigwedge_{i=1}^{k}[y_{i},z_{i}]=1)\wedge x=y_{1}y_{2}\cdots y_{k})
			\end{align*}
		\end{itemize}
	\end{example}
	
	As we discussed above, any positive sentence satisfied by a nonabelian free group is also satisfied by any group and for this reason, we call these positive sentences trivial. Notice that this fact is no longer true for almost positive sentences - there are almost positive sentences satisfied by a nonabelian free group which are not satisfied by some other group. For instance, the almost positive sentence $\exists z_1, z_2 \,\, [z_1, z_2]\ne 1 \, \forall x \exists y \, [x,y]=1$ is satisfied by a nonabelian free group but not by an abelian group.
	
	Our objective is to determine a set of sentences with the property that if the vertex groups of a graph product satisfy one of them, then one can interpret the underlying graph and the associated vertex group in the graph product. We will show that nontrivial positive sentences have this property. On the other hand, the fact that a vertex group in a graph product satisfies an almost positive sentence that is not satisfied by a nonabelian free group will not suffice in order to be able to interpret the vertex group. For instance, if a vertex group is of the form $G_v= \mathbb Z_p \ast \mathbb Z$,  where $\mathbb Z_p$ is a cyclic group of order $p$, then $G_v$ satisfies a nontrivial almost positive sentence, namely $\exists x \ x\ne 1 \wedge \ x^p =1$. However, one should not expect to be able to interpret the underlying graph and the corresponding vertex groups. Indeed, consider $G= G_v \ast \mathbb Z_q= (\mathbb Z_p \ast \mathbb Z) \ast \mathbb Z_q$ and $G'=\mathbb Z_p \ast \mathbb Z_q$. Then, it follows from a result of Sela, see \cite{sela2010diophantine}, that $G$ and $G'$ are elementarily equivalent graph products, but $G_v$ is not elementarily equivalent to $\mathbb Z_p$ or $\mathbb Z_q$. This example motivates the need to restrict the type of almost positive sentences in order to be able to interpret the vertex groups that satisfy them in a graph product. In some sense, one would like to extend the notion of ``nontriviality" to the set of almost positive sentences.
	
	A characterization of trivial positive sentences is obtained via formal solutions: a positive sentence is trivial if and only if it admits a formal solution. Following this characterization, we introduce the notion of a non-generic almost positive sentence - a generalisation of the notion of nontrivial positive sentence. Roughly speaking, an almost positive sentence is generic if there exists a (relative) formal solution that witnesses both the inequalities as well as the positive part of the sentence.
	
	In order to provide the necessary definitions  we first need to recall the notion of formal solution relative to a Diophantine condition.
	
	\begin{definition}[Formal solution relative to a Diophantine condition]\label{defn:formal solution}
		
		Let $\psi(z)$ be a positive formula in the language of groups \emph{without constants}\footnote{all formulas in this paper use no constants --- except the identity element}:
		\begin{align*}
			\psi(z)\equiv\forall x^{1}\exists y^{1}\forall x^{2}\dots\forall x^{m}\exists y^{m}\,\bvee{j=1}{\Sigma_{j}(z,x^{1},y^{1},x^{2}, \dots, x^{m},y^{m})=1}{k},
		\end{align*}
		where $z$, $x^i$, and $y^i$ are tuples of variables for all $1\leq i\leq m$, and $\Sigma_j=1$ is system of equations  for all $1\leq j \leq k$ (so $\Sigma_j$ is a conjunction of words).  Let $\phi(z)$ be a a positive existential formula, which we call \emph{Diophantine condition}, of the form $\exists v \Pi(v,z)=1$ for some system of equations $\Pi(v,z)=1$.
		
		Intuitively, a formal solution of $\psi(z)$ relative to $\phi(z)$ is a substitution of the $y$'s by words on $v$, $z$, and the $x$'s, so that under this substitution some  word in $\Sigma_j$ (for all $j=1, \dots, k$) is trivial if all words in $\Pi$ are.
		
		Formally, by \emph{formal solution} to $\psi(z)$ \emph{relative to} $\phi(z)$ we mean a tuple $\alpha=(\alpha^{1},\dots,\alpha^{m})$ where
		$\alpha^{l}\in\F(v,z,x^{1},x^{2},\dots, x^l)^{|y^{l}|}$ for all $1 \leq l \leq m$ such that for some $1\leq j\leq k$ all words in the tuple
		\begin{align*}
			\Sigma_{j}(z,x^{1},\alpha^{1}(v,z,x^{1}),x^{2},\alpha^{2}(v,z,x^{1},x^{2}), \dots,\alpha^{m}(v,z,x^{1},x^{2}, \dots, x^{m}))
		\end{align*}
		are in the normal closure of the words of $\Pi(v,z)$ in
		$\F(v,z)\frp\F(x^{1},x^{2},\cdots, x^{m})$.
		
		One can understand formal solutions in terms of homomorphisms between  groups.
		Consider the groups given by the following presentations:
		\begin{align*}
			G_{\Pi}(v,z)=&\subg{v,z\,|\,\Pi(v,z)=1} \\
			G_{\Sigma_{j}}(z)=&\subg{v,z,x^{1},y^{1},x^{2},y^{2},\dots, x^{m},y^{m}\,|\,\Sigma_{j}(z,x^{1},y^{1},x^{2},y^{2},\dots, x^{m},y^{m})=1}
		\end{align*}
	\end{definition}

	Formal solutions relative to $\phi$ are equivalent to homomorphisms $f:G_{\Sigma}\to G_\Pi *\subg{x^1,x^2, \dots, x^m}$ that restrict to the identity on each of the elements named by a variable from the tuple $z$ or from any of the tuples $x^{1}, \dots, x^m$, with the property that $f(y^{l})$ is contained in the group generated by $v,z,x^{1},x^{2}, \dots, x^{l}$, for all $1\leq l \leq m$.
	
	Notice that the existence of such $\alpha$ implies that $G\models\psi(a)$ for any group $G$ and any tuple $a\in G^{|z|}$ such that $G\models\phi(a)$. The classical notion of formal solution  \cite{merzlyakov1966positive,sela2006DiophantineII} is recovered by taking $\Pi = 1$, in which case we will speak of $\alpha$ simply as a \emph{formal solution}.
	
	We will denote the collection of all variables in the tuples $x^1, \dots, x^{m}$ simply by $x$. A similar notation will be used for $y^1, \dots, y^m$ and $y$.
	
	\medskip
	
	We now introduce the definition of non-generic almost positive sentence.
	
	\begin{definition}[Non-generic almost positive sentence]\label{defn:generic ap}
		Let
		$$
		\psi(z)\equiv \forall x^1, \exists y^1, \dots,\forall x^m, \exists y^m \, \bvee{j=1}{\Sigma_{j}(z,x^{1},y^{1},x^{2}, \dots, x^{m},y^{m})=1}{k}
		$$
		be a positive formula, and let $\phi$ be an almost positive sentence of the form $\exists z\,(\Theta(z)\neq 1\wedge\psi(z))$, where $\Theta(z)\neq 1$ is a system of inequalities. We say that $\phi$ is \emph{non-generic} if for any  Diophantine condition $\exists v\,\,\Pi(v,z)=1$ and any formal solution relative to $\exists v\,\,\Pi(v,z)=1$, there exists some word in $\Theta(z)$  which belongs to the normal closure of $\Pi(v,z)$ in $\F(v,z)$, or, equivalently, one of the words in $\Theta(z)$ is the trivial element in the group $\F(v,z)/ \langle \langle \Pi(v,z) \rangle\rangle$.
	\end{definition}
	
	As we mentioned in the previous section, see Theorem \ref{thm:Makanin}, a positive sentence $\psi$ is nontrivial if there does not exist a formal solution to it.
	
	\begin{remark}
		\label{r: positive_is_generic} The class of almost positive sentences generalizes that of positive sentences in the sense that every positive sentence $\phi$ can be rewritten as an equivalent almost positive sentence by adding $\exists z\,z\ne 1\wedge \phi$, where $z$ does not appear in $\phi$. {\rm(}Notice that the equivalence between the positive and the almost positive sentence holds in any group besides the trivial group{\rm)}. In this way the notion of simple non-generic almost positive sentence generalizes that of nontrivial positive sentence, since any formal solution for $\phi$ would automatically be a formal solution relative to $z$ for an empty diophantine condition.
	\end{remark}
	
	\begin{example}
		Let $\phi$ be the sentence that asserts that a group has nontrivial center, that is,
		$$
		\exists z \ z\ne 1 \  \forall x \ [x,z]=1.
		$$
		Then $\phi$ is a non-generic almost positive sentence. Indeed, suppose that the formula $\psi(z)\equiv \forall x \ [x,z]=1$ has a formal solution relative to $\exists v \mid \Pi(v,z)=1$, that is, $[x,z]$ belongs to the normal closure of $\Pi(z,v)$ in   $\mbb{F}(v,z, x)$. Hence $[x,z]$ projects to the identity element in the group $\mbb{F}(v,z,x)/\langle \langle \Pi(z,v)\rangle \rangle \cong \mbb{F}(x) * \mbb{F}(v,z)/\langle \langle \Pi(z,v)\rangle \rangle.$ This can only occur if $z$ is  trivial element in $\mbb{F}(v,z)/\langle \langle \Pi(z,v)\rangle \rangle$ {\rm(}by basic properties of free products{\rm)}. Therefore, $\phi$ is non-generic almost positive.
		
		On the other hand, the sentence
		$$
		\exists z \ z^p\neq 1 \ \forall x \exists y \ [xy^{-1},z]=1
		$$
		is almost positive but generic. Indeed, $y=x$ is a formal solution to the formula $\forall x \exists y \ [xy^{-1},z]=1$ {\rm(}relative to the Diophantine condition $1=1${\rm)}. However, the word $z^p$ is not trivial in $\F(v, z)$ and so the sentence is generic.
	\end{example}
	
	As noticed in Remark \ref{r: positive_is_generic} any positive sentence can be regarded as simple non-generic almost positive. Hence the class of groups satisfying a non-generic almost positive sentence contains all the examples given in Lemma \ref{ex:groups_non_trivial_positive_theory}. Furthermore, it also contains the following groups.
	
	\begin{lemma}
		The class of groups satisfying a non-generic almost positive sentence  contains the following families of groups:
		\begin{itemize}
			\item Groups with nontrivial center. Note that, in particular, the group $\mathbb Z \times F$ is large  {\rm(}hence it has trivial positive theory{\rm)}, but it satisfies a simple non-generic almost positive sentence.
			\item Groups with finite centraliser width, i.e.\ groups satisfying
			$$
			\exists x_1, \dots, x_n \ \forall y \ \exists z_1, \dots, z_n, \ x_i\ne 1 \wedge z_i \in C(x_i) \wedge y=z_1 \cdots z_n
			$$
			Note that this class includes boundedly generated groups as well as bounded abelian generated groups.
		\end{itemize}
	\end{lemma}
	
	The class of groups with nontrivial positive theory is closed under taking extensions.
	It turns out that this is also the case when replacing nontrivial positive theory by non-generic almost positive theory:
	
	\begin{thm}[see Theorem 6.7 in \cite{casals2019positive}]\label{lem:extensionpreservation}
		Let $\mathcal C$ be the class of groups which satisfy a non-generic almost positive sentence. Then $\mathcal C$ is closed under extensions. That is: if $E$ is a group, $K,Q\in \mc C$,  $K \triangleleft E$ and $E/K \simeq Q$, then $E\in \mc C$. In particular, the direct product of two groups in $\mathcal C$ is also in $\mathcal C$.
	\end{thm}
	
	Finally, the following result states that if a group satisfies a non-generic almost positive sentence, then it satisfies one of a simpler form whose positive part only has one alternation of quantifiers. Moreover, this simpler formula uses single system of equations, instead of a disjunction of systems of equations.  Recall that an $\forall\exists$-formula is a first-oder formula of the form $\forall x \exists y \phi(x,y,w)$, where $x,y,w$ are tuples of variables, and $\phi$ is a quantifier-free formula.
	
	\begin{thm}[Quantifier Reduction, Theorem 6.5 and Lemma 6.4 in \cite{casals2019positive}]
		\label{t: quantifier reduction} Given a non-generic almost positive sentence $\psi\equiv\exists z\,\Pi(z)\neq 1\wedge\phi(z)$, where $\phi(z)$ is a positive formula and $\Pi(z)\neq 1$ is a system of inequalities, one can effectively describe a non-generic almost positive sentence
		$\psi'\equiv\exists z\,\Pi(z)\neq 1\wedge\phi'(z)$ where $\phi'(z)$ is a positive $\forall\exists$-formula of the form:
		$$
		\forall x \ \exists y \ \Sigma(z,x,y)=1,
		$$
		where $\Sigma=1$ is a system of equations, with the property that the sentence $\forall z \ \phi(z) \rightarrow \phi'(z)$ holds in any  group. In particular,  the sentence $\psi \rightarrow \psi'$ holds in any group, and if a group has nontrivial positive theory,  then it must satisfy some nontrivial positive $\forall\exists$-sentence.
	\end{thm}
	
	\section{Small-cancellation and formal solutions}
	\label{sec:small cancellation and formal solutions}
	
	The goal of this section is to relate the existence of small-cancellation elements in a group with the existence of formal solutions. We begin by recalling the definition of small cancellation elements.
	
	\begin{definition}\label{d: small cancellation}
		\label{d: small cancellation section3}
		Fix an action of a group $G$ on a simplicial tree $T$. Consider a tuple of elements $a=(a_{1},\cdots, a_{m})\in G^{m}$.
		Given $N>0$, we say that $a$ is \emph{weakly $N$-small cancellation} (in $T$) if the following holds for some base point $*\in VT$:
		
		\begin{enumerate}[label=(\alph*)]
			\item \label{SCA} $\dis{a_{i}}>N$ for all $1\le i \le m$,
			\item \label{SCB} $\dis{a_{i}}\leq\frac{N+1}{N}\min_{1\le r \le m} tl(a_{r})$ for all $1\le i\le m$,  and
			\item \label{SCC}
			For all $g\in G$ and $1\leq i,j\leq m$, if $|[*,a_{i}*]\cap g[*,a_{j}*]|\geq\frac{1}{N}\min_{1\le r \le m} tl(a_{r})$, then $i=j$ and $g$ acts like the identity on  $[*,a_{i}*]\cap g[*,a_{i}*]$.
		\end{enumerate}
		Given $c_{1},c_{2},\cdots, c_{k}\in G$, we say that a tuple $a$ is weakly $N$-small cancellation \emph{over} $c_{1},c_{2},\cdots, c_{k}$ if $*$ can be chosen in such a way that
		$\min_{1\le r\le m}\dis{a_{r}}>N \cdot \dis{c_{j}}$, for all $1\leq j\leq k$.
		
		We say that $(a_{1},a_{2},\dots, a_{m})$ is \emph{$N$-small cancellation} (over $c$)
		if it is weakly small cancellation (over $c$) and if it further satisfies the following strengthening of Item \ref{SCC}:
		
		\begin{itemize}
			\item [($c'$)] for all $g\in G$ and $1\leq i,j\leq m$, if $|[*,a_{i}*]\cap g[*,a_{j}*]|\geq\frac{1}{N} \min_{1\le r \le m} tl(a_{r})$, then $i=j$ and $g=1$
		\end{itemize}
		The point $*$ is said to be a \emph{witness} of the tuple $a$ being (weakly or not) $N$-small cancellation.
	\end{definition}
	
	In \cite{casals2019positive} we established a relation between small cancellation tuples and formal solutions to positive formulas. More precisely, we showed that given a positive $\forall \exists$-formula $\psi(w)$ with free variables $w$, there exist formal solutions (relative to a finite set of Diophantine conditions), see Definition \ref{defn:formal solution}, such that if a group $G$ acting on a tree with small cancellation elements satisfies the sentence $\psi(c)$, for some $c\in G^{|w|}$, then the group satisfies one of the Diophantine conditions and so the formula $\psi$ admits a formal solution relative to this Diophantine condition. As a corollary, we deduced that groups acting on trees and having small cancellation elements with respect to this action have trivial positive theory. The following corollary is a special case of Corollary 3.15 in \cite{casals2019positive}.
	
	\begin{cor}
		[Special case of Corollary 3.15 in \cite{casals2019positive}]
		\label{c: formal solutions forall exists}
		Given a positive formula of the form
		\begin{align*}
			\psi(z)\equiv \forall x \exists y \,\,\Sigma(z,x,y)=1,
		\end{align*}
		where $z$, $x$ and $y$ are {\rm(}finite{\rm)} tuples of variables, there is some $N>0$ and a finite collection $\mathcal{D}$ of Diophantine conditions on free variables $z$ with the following properties:
		\begin{enumerate}[label=(\roman*)]
			\item For each $\Delta\in \mathcal{D}$ there exists a formal solution $\alpha_{\Delta}$ to $\psi(z)$ relative to $\Delta$.
			\item \label{item witness_2} For any action of a group $G$ on a tree $T$ and tuples $c\in G^{|z|}$, $a\in G^{|x|}$ and $b\in G^{|y|}$ such that
			\begin{enumerate}[label=(\alph*)]
				\item $a$ is $N$-small cancellation over $c$ and
				\item $\Sigma(c,a,b)=1$,
			\end{enumerate}
			there is a Diophantine condition $\exists v\,\xi(z,v)=1$ in $\mathcal{D}$ and $d\subset G$ such that $\xi(c,d)=1$ holds in $G$.
		\end{enumerate}
	\end{cor}
	
	\section{Non-compatible pairs and small-cancellation}\label{sec: from non-compatible to small cancellation}

	The main goal of the present section is to prove Corollaries \ref{c: small cancellation} and \ref{l: small cancellation multi}, which provide, given a finitely generated subgroup of a graph product, a uniform method for obtaining a finite set of tuple of elements in the subgroup such that, for any factor where the subgroup acts irreducibly, at least one tuple is small-cancellation when projected onto that factor.

	We refer to Section \ref{sec:actions on trees} for a discussion on notation and basic notions regarding groups acting on trees.
	
	We recall some lemmas from \cite{casals2019positive}:
	\begin{lemma}[see Lemma 4.7 in \cite{casals2019positive}]
		\label{l: multiple conjugation} Let $G$ be a group acting on a real tree $T$ and let $g,h\in G$ be an acylindrical pair {\rm(}see {\rm Definition \ref{d: acylindrical_pair})}. Consider the words $v(x,y)=x^{y^{x}}$ and $w^{cn}(x,y)=y^{v(x,y)}x$ {\rm(}where `cn' stands for `conjugation'{\rm)}.  Then the element $w^{cn}(g,h)$ is hyperbolic with $tl(w^{cn}(g,h))\geq\max\{tl(g),tl(h)\}$.  Moreover, $g \notin E(w^{cn}(g,h))$ and $A(g)$ and $A(w^{cn}(g,h))$ intersect coherently in case $g$ is hyperbolic.
	\end{lemma}
	
	Roughly speaking, the next key proposition combines the previous results to show that given a pair of \emph{stable} elements that act irreducibly on a tree and an arbitrary tuple $c$, we can uniformly obtain a tuple of small-cancellation elements \emph{over} $c$.
	
	\begin{prop}[Proposition 4.8 in \cite{casals2019positive}]
		\label{l: from good pair to small cancellation}
		
		For all $N\geq 1$ and $m\geq 1$ there exists an $m$-tuple of words $w^{sc,N}\in\F(x,y)^{m}$ with the following property {\rm(}`sc' stands for `small cancellation'{\rm)}. Suppose we are given an action of a group $G$ on a tree $T$, a hyperbolic element $h\in G$, and another element $g\in G$ such that:
		\begin{enumerate}
			\item $h$ is hyperbolic and stable {\rm(}respectively, weakly stable{\rm)},
			\item $g\notin E(h)$,
			\item \label{it: no cancellation}$A(g)$ and $A(h)$ intersect, $g$ does not invert a subsegment of $A(h)$ and if $g$ is hyperbolic, then $A(g)$ and $A(h)$ intersect coherently,
			\item $tl(g)\leq tl(h)$.
		\end{enumerate}
		Then $w(g,h)^{sc,N}$ is $N$-small cancellation {\rm(}respectively, weakly $N$-small cancellation{\rm)} as witnessed by any point in $A(g)\cap A(h)$, and for all $1\leq i\leq m$ we have
		$$
		tl(w_{i}^{sc}(g,h))\geq N \max\{tl(g),tl(h)\}=N tl(h).
		$$
	\end{prop}
	
	\begin{lemma}[see Lemma 4.6 in \cite{casals2019positive}]
		\label{l: acylindrical pairs} Let $G$ act on a tree $T$ and take two elements $g,h\in G$ such that $\subg{g,h}$ acts irreducibly on $T$. Suppose that either $g$ is elliptic or is the $m$-th power of a weakly $\lambda$-stable element, where $m\geq 10\max\{\lambda,1\}$. Make the same assumption on $h$. Then the pair $\{g,h\}$ is acylindrical.
	\end{lemma}
	
	\begin{prop}
		\label{l: general small cancellation} Let $k,N,m,\lambda,d$ be positive integers, and let $x,y,z_{1},z_{2},\dots,z_{k}$ be variables. Then there is a finite collection $\mathcal{W}\subseteq\F(x,y,z_{1},z_{2},\dots,z_{k})^{m}$ of $m$-tuples of words in the normal closure in $\F(x,y,z_{1},z_{2},\dots,z_{k})$ of
		$\langle x,y \rangle$ with the following property.
		
		Suppose we are given a group $G$ and a collection $\mathcal{T}$ of simplicial $G$-trees such that
		any element hyperbolic in some $T\in\mathcal{T}$ is weakly $\lambda$-stable and any element hyperbolic in every $T\in\mathcal{T}$ is $\lambda$-stable with respect to the action on any such $T$. Let $(a,b)\in G^{2},c\in G^{k}$ be such that $\subg{a,b}$ acts irreducibly on each $T\in\mathcal{T}$ and both $a$ and $b$ have roots of order $\geq 10\lambda$.
		
		Then for every $T\in\mathcal{T}$ there is $u\in\mathcal{W}$ such that the tuple $u(a,b,c)$ is $N$-small cancellation over $c$.
		
		Furthermore, we can assume that for all $u \in \mathcal W$, all components of $u$ belong to any prescribed verbal subgroup of $\F(x,y)$.
		
		In case $k=0$ we can take $\mathcal{W}$ to be a singleton.
	\end{prop}
	\begin{proof}
		The proof is essentially the same as that of \cite[Theorem 5.3]{casals2019positive}, with an additional argument to ensure the last condition. We will recall the general structure of the proof and point out the necessary adjustments.
		
		We know that $a$ and $b$ admit roots  $a_{0}$ and $b_{0}$ of order $s,t\geq 10\lambda$, respectively. Note that $a$ is hyperbolic if and only if $a_0$ is hyperbolic, hence by
		hypothesis, either $a_0$ is  elliptic or $\lambda$-stable.  The same considerations apply for $b_0$. Thus the hypotheses of Lemma \ref{l: acylindrical pairs} apply to $(a,b)$ and so $(a,b)$ is an acylindrical pair (see Definition \ref{d: acylindrical_pair}) with respect to every $T\in\mathcal{T}$.
		
		Let $w^{cn}$ be the word given by Lemma \ref{l: multiple conjugation}. Now by this  same lemma applied to the pair $(a,b)$ we have that, with respect to the action on any $T\in\mathcal{T}$, the element $w^{cn}(a,b)$ is hyperbolic with $tl(w^{cn}(a,b))\geq tl(a),tl(b)$; that $a \notin E(w^{cn}(a,b))$; and that the axis $A(a)$ of $a$ intersects with the axis $A(w^{cn}(a,b))$ of $w^{cn}(a,b)$ coherently in case $a$ is hyperbolic.
		
		Let $d_1 = w^{cn}(a,b)$. In particular, $\langle d_1^{10\lambda}, a \rangle$ acts irreducibly on each
		$T\subset\mathcal{T}$. Therefore if we define $d_2=w^{cn}(d_1^{10\lambda},a)$, then another application of Lemma \ref{l: multiple conjugation} to the pair $(d_1^{10\lambda},a)$  yields that, with respect to the action on each $T\subset\mathcal{T}$, both $d_1$ and $d_2$ are hyperbolic with $tl(d_2)\geq tl(d_1)$, $d_1\nin A(d_2)$ and $A(d_1)$ and $A(d_2)$ intersect coherently.
		
		The next step in the proof of \cite[Theorem 5.3]{casals2019positive} is the following statement:
		\begin{claim}
			\label{f: blackbox 5.3}
			Fix $\lambda>0$, let $N>3$, $k\geq 0$, and let $x,y,z_1, \dots, z_k$ be variables. Then  there exists a pair of words $(v_{1}^{k,N},v_{2}^{k,N})\in\F(x,y,z_{1},z_{2},\dots,z_{k})^{2}$ in the normal closure of $\langle x,y\rangle$ with the following property.
			Let $G$ be a group acting on a tree $T$ in which any hyperbolic element is $\lambda$-stable. Let $(a,b,c)$ with $a,b\in G,\ c \in G^k$ be such that $a,b$ are hyperbolic, $A(a)$ and $A(b)$  intersect coherently, and $tl(a)\leq tl(b)$. Then the pair $(v_{1}^{N,k}(a,b,c),v_{2}^{N,k}(a,b,c))$ is $N$-small cancellation, and there exists a vertex $*\in T$ such that for any $i\in\{1,2\}$ and any $j\in\{1,\dots,k\}$ we have
			$$
			tl(v^{N,k}_{i}(a,b,c))\geq N d(c_{j} *,*)
			$$
		\end{claim}
		We claim that in our context, the pair $(v_{1}^{N,k},v_{2}^{N,k})$ given in Claim \ref{f: blackbox 5.3} verify the same conclusion with respect to the action of $G$ on every tree $T\in\mathcal{T}$, that is  for every two elements $d_1, d_2\in G$, $c\in G^k$ such that $d_1, d_2$ act hyperbolic with respect to the action on each $T\in\mc{T}$ , and $A(d_1) \cap A(d_2)$ intersect coherently, the pair $(v_{1}^{N,k}(d_{1},d_{2},c), v_{2}^{N,k}(d_{1},d_{2},c))$ is $N$-small cancellation, and for each tree $T\in \mathcal{T}$, there is a base point with respect to which $tl(v^{N,k}_{i}(d_1,d_2,c))\geq N d(c_{j} *,*)$ for any $i\in \{1,2\}$ and any $j\in \{1, \dots, k\}$.
		
		Indeed, taking $d_1, d_2$ as above, $d_{1}$ and $d_{2}$ are $\lambda$-stable (not just weakly-stable) because we have shown them to be hyperbolic with respect to the action on every $T\in\mathcal{T}$ and so by assumption they are stable. From that point on we simply follow the proof of  Claim \ref{f: blackbox 5.3}
		in each $T\in\mathcal{T}$ separately. For a fixed $T\in \mc{T}$, every time such a proof  requires that a given hyperbolic element $g$ constructed from $d_{1},d_{2},c$ is $\lambda$-stable, this holds in the present context  as well, since the previous step of the present proof  shows that the element $g$ is hyperbolic in every $T\in\mathcal{T}$ and therefore $\lambda$-stable in each $T\in\mathcal{T}$ by assumption. This proves our claim. 
		
		Summarising, at this stage given $N>0$ and $k\geq 0$
		if we let 
		$$
		w_{i}^{N,k}(x,y,z_1, \dots, z_k)=v_{i}^{N,k}(w^{cn}(x,y),w^{cn}(w^{cn}(x,y),y),z_1, \dots, z_k)\quad (i=1,2),
		$$  then
		for any $a,b\in G$ and $c\in G^k$ as in the assumptions of Proposition \ref{l: general small cancellation}, and any action on any tree $T\in\mathcal{T}$, we have that $(w_{1}^{N,k}(a,b,c),w_{2}^{N,k}(a,b,c))$ is $N$-small cancellation and there exists $*\in V(T)$ such that $tl(w_{i}^{N,k}(a,b,c))\geq Nd(c_{l}*,*)$ for any $1\leq l\leq k$.  
		
		\begin{claim}
			We can assume here that the words $w^{N,k}_{i}$ belong to  any prescribed nontrivial verbal subgroup $H$ of $\F(x,y)<\F(x,y,z)$.
		\end{claim}
		\begin{proof}
			Fix $k$ and $N$ sufficiently large. The pairs $w^{N,k}_{i}$ are constructed by induction on $k$ by successive composition on the left, so it suffices to show that one may assume $w^{N,0}_{i}\in H$. And indeed,
			since the group generated by $w^{N,0}_{1}(a,b)$ and $w^{N,0}_{2}(a,b)$ acts freely (i.e.\ the stabiliser of  any vertex or of any edge is  trivial)  on every $T\in\mathcal{T}$, given any two fixed different words $u_{1}(x,y),u_{2}(x,y)\in H$, if we let $u'_{i}(x,y)=u_{i}(w_1^{N,0}(x,y), w_2^{N,0}(x,y))^{10\lambda}$, then the subgroup generated by the pair $(u'_{1}(a,b),u_{2}'(a,b))$ also acts freely on every tree $T\in\mathcal{T}$, since such a subgroup is contained in $\langle w^{N,0}_{1}(a,b), w^{N,0}_{2}(a,b)\rangle$. The conclusions of the first part of the proof then imply that if we let $u''_{i}=w^{N,0}_{i}(u'_{1}(x,y),u'_{2}(x,y))$ for $i=1,2$, then $(u'_{i}(a,b))_{i=1,2}$ acts as an $N$-small cancellation pair on every tree $T\in\mathcal{T}$ and thus we can take $(u'_{i})_{i=1,2}\in H^{2}$ in place of $(w^{ N,k}_{i})_{i=1,2}$.
		\end{proof}

		Let $v^{N,k}=(v_{1},v_{2})$, let $w^{sc,N}$ be the $m$-tuple of words provided by Lemma
		\ref{l: from good pair to small cancellation}, let
		$\alpha(x,y)=xx^{y}$, and let $\mc{W}$ be
		the following family of $m$-tuples of words:
		\begin{align*}
			\left\{w^{sc,3N}(v^{N,k})\right\} \cup \left\{w^{sc, N}\left(\alpha(v^{N,k}_{i},z_{j}),\alpha((v^{N,k}_{i})^{2},z_{j})\right)\mid 1\le i\le 2, 1\le j \le k\} \right\}.
		\end{align*}
		
		Notice that  $\alpha((v^{M,k}_{i})^{2},z_{j})(a,b,c)$ is $\lambda$-stable due to the fact that $(v^{M,k}_{i})^{2},z_{j}$ is $3$-small cancellation.
		
		The proof of Theorem 5.3 in \cite{casals2019positive} concludes by showing that $\mc{W}$
		satisfies the conclusion of  Theorem  5.3 in \cite{casals2019positive}  for $N\geq 3$ large enough.  The argument is a case-by-case analysis of the action of the components of the tuple  $u^{N,k}(a,b,c)$ (following the notation in the proof of Theorem  5.3 in \cite{casals2019positive})  on $T$, using the properties of the word $w^{sc,N}$ provided by Proposition \ref{l: from good pair to small cancellation}.  The same arguments apply separately to the action of $G$ on each $T\in\mathcal{T}$.
	\end{proof}

	The main Corollary of Proposition \ref{l: general small cancellation} is the following result.  Recall the definition of $\ff(H)$ in Definition \ref{d: minimal_graph_subg_and_factors}.
	\begin{cor}
		\label{c: small cancellation}  For any $m,k,n,N\in\N$ there exist a collection $\mathcal{S}=\mathcal{
			S}(n,k,m,N)$  of $m$-tuples of words in $\F(x,x',y_{1},y_{2},\dots, y_{k})^{m}$ such that the following holds.  Suppose that we are given a graph product $\mathcal{G}=\G(\Gamma,\gvv)$ with $\Gamma$ satisfying condition $AP_{n}$, and  elements $a,b, c_{1},c_{2},\dots, c_{k}\in G$. Assume furthermore that for any $d\in\{a,b\}$ the element $d$ has an $m$-th root  for some $m\geq 10(n+3)$.
		
		For any factor $\G'=\G(\Delta)\in\mathcal{F}(\subg{a,b})$ such that $ \subg{a,b}$  is irreducible in $\G'$ {\rm(}see {\rm Definition \ref{d: types_of_actions_on_subgraph_products})}  and for each tree $T_v$  ($v\in \Delta$), there is at least one $w\in \mathcal{S}$ such that, for the canonical projection $\pi_{\G'}: \G \to \G'$, the tuple $\pi_{\G'}(w(a,b,c))$ is $N$-small cancellation over $(\pi_{\G'}(c_1), \dots, \pi_{\G'}(c_k))$ with respect to the action of $\G'$ on $T_v$.

		Furthermore,  we can assume that all components of $w$ belong to any prescribed verbal subgroup of $\mathbb F(x,x',y_1, \dots, y_k)$. In case $k=0$, one can assume that $\mathcal{S}$ is a singleton and the conclusion then holds for all vertex trees $T_v$ associated with $\G'$.
	\end{cor}
	\begin{proof}
		We are going to show that the statement is a consequence of Corollary \ref{c: stable axis} and Proposition \ref{l: general small cancellation} with $\lambda=n+3$. Indeed, let $\G'=\G(\Delta) \in\mathcal{F}(\subg{a,b})$ be such that $\subg{a,b}$ is irreducible in $\G'$. By definition, the graph $\Delta$ is directly indecomposable and $\subg{a,b}$ is not of smaller complexity in $\G'$.
		
		Let $\pi_{\G'}: \G \to \G'$ be the canonical projection  and consider the tuple $\pi_{\G'}(c)= (\pi_{\G'}(c_1), \dots, \pi_{\G'}(c_k))  \in \G'$. By virtue of Corollary \ref{c: stable axis}, the collection of trees $\mathcal{T}'=\{T_{v}'\}_{v\in \Delta}$ associated to $\G'$ satisfies all the properties required by the statement of Proposition \ref{l: general small cancellation} with $\lambda=n+3$.
		
		Furthermore, the subgroup $\subg{a,b}$ acts irreducibly on each tree $T_v'$ and by assumption, $a$ and $b$ have $m$-roots with $m \ge 10(n+3)$. Hence, from Proposition \ref{l: general small cancellation}, we deduce that $\pi_{\G'}(w(a,b,c))$ is $N$-small cancellation over the tuple $\pi_{\G'}(c)$ with respect to the action of $\G'$ on $T_v'$, for some $w\in \mathcal{S}$. The last  part of the statement is due to the last conclusions of Proposition \ref{l: general small cancellation}.
	\end{proof}

	In Corollary \ref{c: small cancellation}, one deduces a collection of tuples of words such that at least one tuple is small cancellation given a 2-generated subgroup $K$ and an irreducible action on one of the factors $\mc{H}(K)$ (see Definition \ref{d: types_of_actions_on_subgraph_products}).
	
	We now address the general situation for arbitrary $k$-generated subgroups. More precisely, our goal is to prove the following
	
	\begin{cor*}[See Corollary \ref{l: small cancellation multi}]
		For any $m,k,n,N\in\N$ there exists a finite collection $\mathcal{W}$ of $m$-tuples $w\in \mathbb F(y_{1},\cdots, y_{k})^m$ such that the following holds.	
		
		Suppose we are given any graph product $\mathcal{G}=\G(\Gamma,\gvv)$ such that $\Gamma$ satisfies property $AP_n$. Let $c_{1},c_{2}\cdots c_{k}\in\G$ and denote by $C$ the subgroup $\subg{c_{1},c_{2},\cdots, c_{k}}$.
		
		Then for any $\G'=\G(\Delta) \in\ff(C)$ in which $C$ is irreducible, there is some $u\in\mathcal{W}$ such that the projection of $u(c)$ onto $\G'$ is $N$-small cancellation over the projection of $c$ in $\G'$ with respect to one of the associated $\G'$-trees.
	\end{cor*}
	
	Note that this corollary generalizes the previous Corollary \ref{c: small cancellation}.

	The key technical result for proving this, namely Theorem \ref{cl: basic claim}, gives a uniform way to find a collection of words $\mathcal U$ such that for each collection of elements $c_1, \dots, c_k$, and each factor $\G'$ of   $\mc{H}(\subg{c_1, \dots, c_k})$ such that $\subg{c_1, \dots, c_k}$  is irreducible in $\G'$, we have that $\G'$ is also a factor of the cyclic subgroup generated by a word $u(c)$, for $u\in \mathcal U$.
	
	Before stating and proving Theorem \ref{cl: basic claim}, we begin with some remarks and technical lemmas.
	
	\begin{lemma}\label{l: non-ellipticity}
		Suppose we are given a graph product $\mathcal{G}=\G(\Gamma,\gvv)$, $C=\subg{c_1,\dots, c_k}\leq\G$ and $\HH\in\ff(C)$ such that $C$ is irreducible in $\HH$. Then there is a word $v\in B_{2}(y_{1},y_{2},\dots, y_{k})$ such that some non-singular component from $\ff(v(c))$ is contained in $\HH$.
	\end{lemma}
	\begin{proof}
		Let $\pi: \G \to \HH$ be the canonical projection from $\G$ onto $\HH$. Since $C$ is irreducible in $\HH$, there exist $c_i$ and $c_j$ in $C$ such that $\pi(c_i)\ne 1$, $\pi(c_j)\ne 1$ and $\pi(c_j) \notin E(\pi(c_i))$.
		
		If $\pi(c_i)$ (resp. $\pi(c_j)$) is not elliptic, then some member of $\ff(\pi(c_i))$ is non-singular and contained in $\HH$ so it suffices to take $v=y_i$ (resp. $v=y_j$). If both $\pi(c_i)$ and $\pi(c_j)$ are elliptic, since $\pi(c_j) \notin E(\pi(c_i))$, we have that $\pi(c_ic_j)$ is hyperbolic and so $\ff(\pi(c_i), \pi(c_j))$ is non-singular and it is contained in $\HH$. Hence, in this case, we can take $v=y_iy_j$.
	\end{proof}
	
	\begin{notation}\label{notation}
		Let $w$ be the word given by {\rm Corollary \ref{c: small cancellation}} for $k=0$, $m=1$, $n, N\in \mathbb N$, which can be assumed to belong to the verbal subgroup $[[\F,\F],[\F,\F]]$, where $\F=\F(x,y)$. By {\rm Corollary \ref{c: small cancellation}}, the word $w$ satisfies that for any pair $a_{1},a_{2}\in \G$ where each $a_{i}$ is an $r$-th power for some $r\geq 10(n+3)$ and for any $\HH\in\ff(a_{1},a_{2})$ such that $\subg{a_{1},a_{2}}$ is irreducible in $\HH$, we have that $w(a_1, a_2)$ is $N$-small-cancellation with respect to the action of $\HH$ on each tree associated to $\HH$. In particular, $\HH \in \ff(w(a_1,a_2))$ and $w(a_1, a_2)$ is hyperbolic with respect to the action on each tree associated to $\HH$, hence by {\rm Corollary \ref{c: stable axis}},  $w(a_1, a_2)$ is $(n+2)$-stable for this collection of trees. Since $w(a_1, a_2)$ is $N$-small cancellation, we have that $(n+2)tl(a_1)< tl(w(a_1,a_2)$ with respect to the action of $\HH$ on each tree associated to $\HH$.
		It follows from {\rm Lemma \ref{l: hyperbolic from stable}}, that the word $w^*(x,y)=xw(x,y)\in x[[\F,\F],[\F,\F]]$ satisfies that $w^*(a_1,a_2)$ is hyperbolic in each tree associated to $\HH$ and so
		\begin{equation}\label{eq:factor}
			\HH\in\ff(w^*(a_{1},a_{2})).
		\end{equation}
		
		Let $c=(c_1, \dots, c_k) \in \G^k$ be such that each $c_i$ has an $r$-th power for some $r\geq 10(n+3)$.
		
		Given a word $v\in \mathbb F(y_{1},y_{2},\dots, y_{k})$, we will write  $\ff_{v}$ instead of $\ff(v(c))$ and $\ff_{i}$ in place of $\ff_{c_{i}}$. For every $1\leq j\leq k$ we use the expression $\Lambda_{j}(v)$ to denote the word $w^{*}(v^{p},y_{j})$.
		
		Given a sequence $\bar{i}=(i_{1},i_{2},\dots, i_{r})\in\{1,2,\dots,k\}^{r}$, we will abbreviate $$\Lambda_{i_{1}}(\Lambda_{i_{2}}(\dots(\Lambda_{i_{r}}(v))\dots))$$ simply by $\Lambda_{\bar{i}}(v)$.
		
		Given a set of words $\mathcal{V}$ we let $\Lambda_{\bar{i}}(\mathcal{V})=\{\Lambda_{\bar{i}}(v)\,|\,v\in\mathcal{V}\}$.
	\end{notation}
	
	\begin{definition}[Domination]
		\label{d: domination}Given a collection $\ff$ of subgroups of $\G$ and $H$ a subgroup of $G$  we say that $H\prec\ff$ if there exists $K\in\ff$ such that $H\leq K$.
	\end{definition}

	\begin{lemma}
		\label{l: merging}
		In the notation of {\rm Notation \ref{notation}}, let $\G'\in\ff_{v}$. If $\G'$ is not singular, then some $\G''\in\ff_{\Lambda_{i}(v)}$ contains $\G'$ together with all components of $\ff_{i}$ non-orthogonal to $\G'$.
	\end{lemma}
	\begin{proof}
		By Remark \ref{rem:factors}, let $\HH\in\ff(v(c),c_i))$ be such that $\G'\leq\HH$.
		
		Assume that $\subg{v(c),c_{i}}$ is not irreducible in $\HH$. Since $\G'$ is not singular, it follows by Corollary \ref{c: types of actions} that the action of $\subg{v(c),c_i}$ on any vertex tree associated to $\HH$ is either cyclic or dihedral. In particular, $\pi_{\hh}(\subg{v(c),c_{i}})$ is also either infinite cyclic or infinite dihedral.
		
		Since $w^{*}\in x[[\ff,\ff],[\ff,\ff]]$ and both cyclic and dihedral subgroups are metabelian, we have that
		$\pi_{\hh}(w^{*}(v(c)^{p},c_{i}))=\pi_{\hh}(v(c)^{p})$. Since $\G'$ is non-singular, that is $\subg{v(c)}$ is infinite cyclic in $\G'$, it follows that $\G'\prec\ff_{\Lambda_{i}(v)}$. On the other hand, non-singularity of $\G'\in \ff_v$ implies $\G'=\HH$ by Lemma \ref{l: linear expansions}. Putting these two facts together we conclude that $\G'\in \ff_{\Lambda_i(v)}$.
		
		If $\subg{v(c)^p,c_{i}}$ is irreducible on $\HH$, since by assumption $c_i$ have sufficiently large roots, then $\HH\in\ff_{\Lambda_{i}(v)}$ by Equation \ref{eq:factor}. In this case, if $\HH'\in\ff_{i}$ is not orthogonal to $\HH$, then by Remark \ref{rem:factors} we have that $\HH' < \HH$.
	\end{proof}

	\begin{cor}
		\label{c: reaching via a subsequence}Let $v$ be a word in $y_{1},y_{2},\dots, y_{k}$ and $\bar{j}=(j_{1},j_{2},\dots, j_{r})\in\{1,2,\dots,k\}^{r}$ be a (ordered) subsequence of some sequence $\bar{j'}=(j'_{1},j'_{2},\dots, j'_{s})\in\{1,2,\dots, k\}^{s}$. Suppose that there is some non-singular $\HH\in\ff_{v}$ and components $\HH_{i}\in\ff_{j_{i}}$ for $1\leq i\leq r$ such that $\HH$ is non-orthogonal to $\HH_{1}$ and $\HH_{i}$ is non-orthogonal to $\HH_{i+1}$ for $1\leq i\leq r-1$.
		
		Then $\subg{\HH,\HH_{i}}_{1\leq i\leq r}\prec\ff_{\Lambda_{\bar{j'}}(v)}$.
	\end{cor}
	\begin{proof}
		First, we claim that $\HH \prec \ff_{\Lambda_{(j_i', \dots,j_s')}(v)}$, for all $i=1, \dots, s'$. Indeed, if $s=1$, then it follows from Lemma \ref{l: merging}. Assume by induction that $\HH \prec \ff_{\Lambda_{(j_{i+1}', \dots,j_s')}(v)}$, that is $\HH< \G'\in \ff_{\Lambda_{(j_{i+1}', \dots,j_s')}(v)}$. Since $\HH$ is non-singular, so is $\G'$. Then if $v'=\Lambda_{(j_{i+1}', \dots,j_s')}(v)$, then by Lemma \ref{l: merging} we have that $\G' \prec \ff_{\Lambda_{j_i'}(v')}= \ff_{\Lambda_{(j_i', \dots,j_s')}(v)}$ and so $\HH \prec \ff_{\Lambda_{(j_i', \dots,j_s')}(v)}$.
		
		We now prove the corollary by induction on $r$.
		
		Let $r=1$. Without loss of generality, we can assume that $j_{1}=j'_{l}$. Since $\HH \prec \ff_{\Lambda_{(j_1', \dots,j_s')}(v)}$, there exists $\G'' \in \ff_{\Lambda_{(j_i', \dots,j_s')}(v)}$ such that $\HH<\G''$. Since $\HH_1$ is not orthogonal to $\HH$, neither is to $\G''$ and so by Remark \ref{rem:factors}, we have that $\HH_1<\G''$.
		
		Assume by induction that $\subg{\HH, \HH_1, \dots, \HH_{r-1}} <\G'' \in \ff_{\Lambda_{(\bar{j'})}(v)}$. By assumption $H_r$ is not orthogonal to $H_{r-1}$ and so it is not orthogonal to $\G''$. By Remark \ref{rem:factors}, it follows that $H_r< \G''$ and so the statement follows.
	\end{proof}
	
	We are now ready to state the main technical result.
	
	\begin{thm}
		\label{cl: basic claim}
		For any $k,n \in\N$ there is a finite collection $\mathcal{U}$ of words $u\in \mathbb F(y_1, \dots, y_k)$ such that  the following holds. Suppose we are given a graph product $\mathcal{G}=\G(\Gamma,\gvv)$ such that $\Gamma$ satisfies property $AP_{n}$ and elements $c_{1},c_{2},\cdots, c_{k}\in\G$ such that each $c_i$ is an $r$-th power  for some $r\geq 10(n+4)$; denote by $C$ the subgroup $\subg{c_{1},c_{2},\cdots,c_{k}}$.
		
		For all $\G' \in\ff(C)$ such that $C$ acts irreducibly on $\G'$, there is a word $u\in \mathcal U$ such that $\G'\in \ff(u(c))$.
	\end{thm}
	\begin{proof}
		Let $\tau$ be the result of concatenating $n$ copies of the sequence $(1,2,\dots, k)$ and let $\Lambda_\tau$ be as in Notation \ref{notation}.
		
		We show that the collection of words:
		$$
		\mathcal{U}=\Lambda_{\tau}(B_{2}(z)),
		$$
		where $B_2(z)$ is the set of words of length 2 in variables $z$, satisfies the conditions of the statement, that is for any $\G' \in\ff(C)$ such that $C$ acts irreducibly on $\G'$, there is some $v\in B_{2}(z)$ such that $\G'\in\ff_{\Lambda_{\tau}(v)}$.
		
		By Lemma \ref{l: non-ellipticity}, there exists $v\in B_2(y_1, \dots, y_k)$ such that some non-singular $D\in\ff(v(c))$ is contained in $\G'$.
		
		Let $\partial_{i}(D)$ (resp. $\partial(D)$) be the smallest sub-graph product containing $D$ and all factors of $\ff(c_{i})$ (resp. $\bigcup_{1 \leq i\leq k}\ff(c_{i})$) not orthogonal to $D$. Given $\bar{i}=(i_{1},i_{2},\dots, i_{r})\in\{1,2,\dots,k\}^{r}$, let $\partial_{\bar{i}}(D)=\partial_{i_{r}}(\partial_{i_{r-1}}(\dots\partial_{i_{1}}(D)))$. Notice that if $D$ is a directly indecomposable subgraph product, then so is $\partial_{i}(D)$ (resp. $\partial(D)$ for any $i\in\{1,2,\dots, m\}$. Hence, the same statements hold true for any $\partial_{\bar{i}}$ and any $\partial^{k}$.
		
		It follows from Remark \ref{rem:factors} that $\G(\partial^{(r)}(D))=\G(\bigcup_{\bar{i}\in\{1,2,\dots,k\}}\partial_{\bar{i}}(D))$. By Lemma \ref{l: blob lemma} and Observation \ref{o: generation}, the subgroup $\G(\partial^{(n)}(D))$ coincides with the unique $\HH\in\ff$ containing $D$.
		
		On the other hand, $\tau$ contains any tuple $\bar{i}\in\{1,2,\dots,k\}^{n}$ as a subsequence. By Corollary \ref{c: reaching via a subsequence}, since $D$ is a non-singular factor of $\mathcal{F}(\subg{v(c)})=\ff_{v}$ contained in $\HH$, we have that for any such $\bar{i}$ the subgroup $\partial_{\bar{i}}(D)$ is contained in some component of $\ff_{\Lambda_{\tau}(v)}$. By the discussion in the previous paragraph, the latter must coincides with the unique $\G'\in\ff(C)$ containing $D$. Hence, $\G' \in \ff_{\Lambda_{\tau}(v)}$.
	\end{proof}
	
	\begin{cor}\label{l: small cancellation multi}
		For any $m,k,n,N\in\N$ there exists a finite collection $\mathcal{W}$ of $m$-tuples $w\in \mathbb F(y_{1},\dots, y_{k})^m$ such that the following holds.	
		
		Suppose we are given any graph product $\mathcal{G}=\GP$ such that $\Gamma$  satisfies property $AP_{n}$. Let $c_{1},c_{2},\dots, c_{k}\in\G$ and denote by $C$ the subgroup $\subg{c_{1},c_{2},\cdots,c_{k}}$.
		
		Then for any $\G'=\G(\Delta) \in\ff(C)$ such that $C$ is irreducible in $\G'$ and for any tree $T_v$ associated to  $\mc{G}'=\mc{G}(\Delta)=\mc{G}(\Delta, \{G_v\}_{v\in \Delta})$, there is some $u\in\mathcal{W}$ such that  $\pi^{\G'}(u(c))$  is $N$-small cancellation over  $\pi^{\G'}(c)$ with respect to the action of $\G'$ on $T_v$.
	\end{cor}
	\begin{proof}
		Let us first notice that, without loss of generality we can assume that all the entries of the tuple $c$ are $r$-th powers for some $r\geq 10(n+4)$. Indeed, let us fix two integers, $p,q$ which are coprime and greater than $10(n+4)$. Assume that the Lemma holds for a $2k$-tuple $c'$ where each component is an $r$-th power and so there is a family $\mathcal{W}$ of $2k$-tuples that satisfies the statement of the Lemma, then the family $\mathcal{W}^*=\{w(z_{1}^{p},z_1^{q},z_{2}^{p},z_{2}^{q},\dots, z_{k}^{q})\}$ is a family of words that satisfies the statement of the Lemma for a general $k$-tuple $c$.
		
		\newcommand{\wm}[0]{w}
		Let $\W_{0}=\mathcal{W}_0(m,k,n,N) \subseteq\mathbb F(x,x',y_1, \dots, y_k)$ be the finite collection of words provided by Corollary \ref{c: small cancellation}.
		Consider the collection:
		\begin{align*}
			\mathcal{W}=\{w^q(u,z_{i},z_{1},z_{2},\dots, z_{k})\,|\,1\leq i\leq k,u\in\U,w\in\W_{0}\}.
		\end{align*}
		
		By our initial remark, we can assume that each $c_i$ in the tuple $c=(c_{1},c_{2},\dots, c_{k})$ is an $r$-th power. Let
		$\G'\in\ff(C)$ be such that $C$ is irreducible in $\G'$. By Theorem \ref{cl: basic claim}, there exists $u\in\U$ such that
		$\G'\in\ff(u(c))$.
		
		By Remark \ref{rem:factors}, since $\G'\in \ff(u(c)) \cap \ff(c)$, we have that $\G'\in \ff(u(c), c_i)$, for all $i=1, \dots,k$.
		
		Let $\pi_{\G'}:\G \to \G'$ be the canonical epimorphism and let $T_v$ be one of the $\G'$-trees associated with $\G'$. Since $C$ is irreducible in $\G'$, it follows that for some $1\leq i\leq k$ the element $\pi_{\G'}(c_i)$ does not preserve the axis of $u(c)$ on $T_v$. Hence we have that $\G'\in \ff(u(c), c_i)$ and $\subg{u(c), c_i}$ acts irreducibly on $\G'$.  Therefore, it follows from Corollary \ref{c: small cancellation} that the tuple $\pi_{\G'}(w(u(c),c_{i},c_{1},c_{2},\dots, c_{k}))$ is $N$-small cancellation over $\pi_{\G'}(c_{1},c_{2},\dots, c_{k})$ with respect to the action on $T_v$. This proves the Corollary.
	\end{proof}
	
	\section{Graphs}
	
	In this section we work in the category of group-labelled graphs $(\Gamma, \gvv)$. When the labels play no role or are clear from the context, we refer to the pair of the form $(\Gamma, \gvv)$ by its first component.
	
	One of the main goals of this paper is to interpret the underlying structure of a graph product of groups (in a certain class of graph products). This problem naturally brings to the question of which particular underlying graph product structures we can recovered with the elementary theory and whether or not the underlying structure is unique.
	
	In general, the graph product decomposition of a group is far from being unique even if we require that the vertex groups satisfy a simple non-generic almost positive sentence. Indeed, since the class of groups with nontrivial positive theory is closed under taking direct products, if a graph contains a subgraph such that the subgroup that it defines is a direct product of singular and dihedral groups that can be sent to a vertex by a full morphism, see Definition \ref{defn:full morphism}, then we can collapse this subgraph to a vertex keeping the structure of a graph product and the property that the vertex groups have nontrivial positive theory. Similarly, the almost positive theory is preserved under direct product with any other group (not necessarily satisfying a non-generic almost positive sentence). In that case, we can collapse the joins that can be sent by full morphisms to a vertex and one of the factors is singular or dihedral.
	
	Furthermore, the direct product decomposition of a group is not always captured by the elementary theory, and in fact, not even by the isomorphism class of the group. Indeed, Baumslag described examples of finitely generated nilpotent groups that admit direct product decompositions with different number of directly indecomposable factors, that is $N \simeq N_1 \times N_2 \simeq N_1'\times \dots \times N_k$, for $k >2$ and $N_i, N_j'$ directly indecomposable for $i=1,2$ and $j=1, \dots, k$. These examples show that one cannot expect to recover any or all the graph product decompositions of a group and it hints that the best we can aim is at recovering the ``smallest" graph product decomposition that the group admits keeping the properties on the vertex groups, namely satisfying a simple non-generic almost positive sentence.
	
	In this section we address this matter by introducing different notions of reduced graphs. Roughly speaking, the almost positive reduced graph gives the ``smallest" graph product decomposition of $\mc{G}$ while preserving the property that its vertex groups satisfy a simple non-generic almost positive sentence. More precisely, we show that the graph product $\mc{G}$ can be given a structure of graph product $\mc{G}(\Gamma', \{G_v'\})$, where $\Gamma'$ is in the class of almost positive reduced graphs and all vertex groups $G_v'$ satisfy a simple non-generic almost positive sentence. Similarly, we introduce the notion of positively reduced graph as and show that the positively reduced graph is the ``smallest" graph product decomposition of $\mc{G}$ while preserving the property that its vertex groups satisfy a nontrivial positive sentence. We show in Corollary
	\ref{cor:rigidity isomoprhism} that this minimal graph product decomposition of a group is unique and so an isomorphism invariant in the class. More precisely, if $\G \simeq \G(\Delta, \{G_v'\}_{v\in \Delta})\simeq \G(\Gamma, \gvv)$ and $G_v,H_v$ have nontrivial (almost) positive theory and $\Delta,\Gamma$ are (almost) reduced graphs, then $\Delta$ and $\Gamma$ are isomorphic and this isomorphism extends to an isomorphism of the corresponding vertex groups.
	
	Even when we consider graph product decompositions of groups where the graph $\Gamma$ is (almost) positive reduced, we are not going to recover the whole underlying structure induced by $\Gamma$, only a part of it, that is the underlying structure induced by the subgraph $\ECore(\Gamma)$, see Definition \ref{defn: graph I(Gamma)}. The main reason why we cannot recover the whole underlying structure is due to the fact that if some vertices of the graph, which we call weak vertices, have associated groups elementarily equivalent to $\mathbb Z$, then these subgroups (and their conjugates) cannot be distinguished from some cyclic subgroups generated by hyperbolic elements in the graph product. One can already this this phenomenon in the case of non-abelian free groups, for which the underlying graph cannot be fully interpreted as they are all elementarily equivalent. In order to define (and later to interpret) the graph $\ECore(\Gamma)$, in this section, we first introduce (and later interpret) some other graphs, namely the $d$ and $c$-completions of $\Gamma$, see Definitions \ref{defn:D-completion} and \ref{defn:Gamma^c}.

	\subsection{Reduced graphs and positive theory}
	
	Let us begin with some observations about the relation between graph epimorphisms and the induced epimomorphism of graph products.
	
	Let $\Gamma,\Delta$ be graphs and let $\phi:\Gamma\to \Delta$ be a graph epimorphism. Let $\mc{G}(\Gamma, \gvv)$ be a graph product of groups. Then $\phi$ induces an epimorphism $\phi^*$ of graph products: $\phi^*:\mc{G}(\Gamma, \gvv)\to \mc{G}(\Delta, \{G_{v'}\}_{v\in \Delta})$ where $G_{v'}$ is isomorphic to the subgroup of $\mc{G}(\Gamma, \gvv)$ generated by vertex groups $G_v$, where $v\in \phi^{-1}(v')$.
	
	We next see that that if the epimorphism between the graphs is full, then the epimorphism of graph products is in fact an isomorphism.
	
	\begin{lemma}\label{lem: full is iso}
		Let $\phi: \Gamma \to \Delta$ be an epimorphism of graphs, see {\rm Definition \ref{defn:full morphism}}. Let $\mc{G}(\Gamma, \gvv)$ be a graph product of groups. Then, $\phi$ is full if and only if the epimorphism $\phi$ induces an isomorphism $\phi^*$ of graph products: $\phi^*:\mc{G}(\Gamma, \gvv)\to \mc{G}(\Delta, \{G_{v'}'\}_{v'\in \Delta})$ where $G_{v'}'$ is isomorphic to the subgroup of $\mc{G}(\Gamma, \gvv)$ generated by vertex groups $G_v$, where $v\in \phi^{-1}(v')$.
	\end{lemma}
	\begin{proof}
		Let us first assume that $\phi$ is full. By definition of $\phi^*$, we have that $\phi^*$ is injective on every $G_v$, $v\in \Gamma$ and maps it into $G_{v'}'$, where $v'=\phi(v)$, $v'\in \Delta$. Furthermore, by definition, $G_{v'}'$ is the graph product generated by the vertices $\phi^{-1}(v')$.
		
		We need to show that there are no further relation between the vertex groups that are sent to different vertices. By the defining relations in graph products, the latter is equivalent to show that if $\phi(v) \ne \phi(w)$, then $(v,w)\in E(\Gamma)$ if and only if $(\phi(v), \phi(w))\in E(\Delta)$.
		
		Since, $\phi$ is an epimorphism, if $(v,w)\in E(\Gamma)$ then $(\phi(v), \phi(w))\in E(\Delta)$. Conversely, if $(\phi(v), \phi(w))\in E(\Delta)$, then as the homomorphism if full we have that $(v,w)\in E(\Gamma)$.
		
		Therefore, $\phi^*$ is injective and as it is surjective, it is an isomorphism.
		
		Conversely, if $\phi$ is not full, it follows from the definition, that there exist $v,w\in \Gamma$ such that, $\phi(v)\ne \phi(w)$, $(v,w)\notin E(\Gamma)$ and $(\phi(v), \phi(w))\in E(\Delta)$. From the defining relations in graph products, we have that $\subg{G_v, G_w}\simeq G_v \ast G_w$ and $\subg{\phi^*(G_v), \phi^*(G_w)}\simeq \phi^*(G_v)\times \phi^*(G_w)$ and so $\phi^*$ is not an isomorphism.
	\end{proof}
	
	\begin{definition}[Reduced graph] \label{defn:reducedgraph}
		A graph $\Gamma$ is \emph{reduced} if for all $v, v'\in \Gamma$ we have that $\st(v) \ne \st(v')$.
	\end{definition}
	
	\begin{definition}[Positive reduced graph] \label{defn:reduced graph}
		We say that a group-labelled graph $(\Gamma,\gvv)$ is \emph{positive reduced} if there are no induced subgraphs $(\Lambda, \{G_v\}_{v\in \Lambda})$ of $(\Gamma, \gvv)$  such that
		\begin{itemize}
			\item  $\Lambda$ can be written as $\Lambda = \Lambda_1 \times \dots \times \Lambda_r$ where $r\geq 1$, $|V(\Lambda)|>1$, and for all $i$, $\mc{G}(\Lambda_i, \{G_v\}_{v\in \Lambda_i})$ is either a singular subgroup or a non-singular dihedral group of $\GP$ (i.e. either $|V(\Lambda_i)|=1$ of $\Lambda_i$ consists of two disconnected vertices $u_{i_1}, u_{i_2}$ and $G_{u_{i_1}}\cong \mbb{Z}/2\mbb{Z}\cong G_{u_{i_2}}$).
			\item for all $v_1, v_2\in \Lambda$ we have that  $\st(v_1)\cap \Gamma\smallsetminus \Lambda=\st(v_2)\cap \Gamma\smallsetminus \Lambda$ (equivalently, $\st(v) < \Lambda \cup \link{\Lambda}$ for all $v\in \Lambda$).
		\end{itemize}
	\end{definition}

	\begin{lemma}[Positive graphs] \label{lem: graph is positive reduced}
		Let $\mc{G}=\GP$ be a graph product such that $\Gamma$ is finite and $G_v$ satisfy a nontrivial positive sentence for all $v\in \Gamma$. Then $\mc{G}$ is a graph product $\mc{G}(\Gamma', \{G_v'\})$ where $\Gamma'$ is a positive reduced graph and $G_v'$ satisfies a nontrivial positive sentence for all $v\in \Gamma'$. Furthermore, the positive reduce graph $\Gamma'$ is unique up to isomorphism.
	\end{lemma}
	\begin{proof}
		Suppose that $\Gamma$ is not positive reduced.
		
		Using the notation of Definition \ref{defn:reduced graph}, let $(\Lambda, \{G_v\}_{v\in \Lambda})$ be a maximal (by inclusion) induced subgraph of $(\Gamma, \gvv)$ that prevents $\mc{G}$ from being positive reduced.
		
		Let $\Gamma''$ be the graph obtained from $\Gamma$ by collapsing $\Lambda$ into a single vertex $v_\Lambda$. For all $v,w\notin \Lambda$, there is an edge $(v,w)$ in $\Gamma''$ if and only if $(v,w)$ is an edge of $\Gamma$. Set $(v_\lambda, w)$ be an edge of $\Gamma''$ if and only if $(v,w)$ is an edge of $\Gamma$, where $v\in \Lambda$ and $w\notin \Lambda$. Note that the edge set of $\Gamma''$ is well-defined since for any $v_1,v_2\in \Lambda$ we have $(v_1,w)\in E(\Gamma)$ if and only if $(v_2,w)\in E(\Gamma)$.
		
		The canonical projection  $\phi:\Gamma\to \Gamma''$ is a full epimorphism, hence it follows from Lemma \ref{lem: full is iso} that  $\phi^*$ induces an isomorphism of graph products $\phi^*:\mathcal G(\Gamma, \gvv)\to \mathcal G(\Gamma'', \{G_{v''}\}_{v\in \Gamma''})$ where $G_{v''}$ is generated by the vertex groups $G_v$ for $v\in \phi^{-1}(v'')$.
		
		Furthermore, since by definition $\Lambda$ is a join, we have that $\mc{G}(\Lambda)$ is a direct product of groups with nontrivial positive theory. It follows from Lemma \ref{lem:extensionpreservation} that all vertex groups of the graph product $\mc{G}(\Gamma'', \{G_{v''}\}_{v\in \Gamma''})$ satisfy some nontrivial positive sentences.
		
		Since the graph $\Gamma$ is finite, the number of vertices of $\Gamma''$ is less than that of $\Gamma$, and so we can recursively apply the above procedure a finite number of times and obtain a graph $\Gamma'$ which is positive reduced and such that all the vertex groups of the graph product $\mc{G}(\Gamma', \{G_v'\}_{v\in \Gamma'})$ satisfy a nontrivial positive sentence. This concludes with the existence of a positive reduced graph.
		
		Let us now show that the positive reduced graph is unique up to isomorphism. Let $\Lambda, \Lambda' < \Gamma$ be two maximal induced subgraphs of $(\Gamma, \gvv)$ that prevent $\mc{G}$ from being positive reduced. If $\Lambda$ and $\Lambda'$ are disjoint, then independently of the order in which we collapse $\Lambda$ and $\Lambda'$, we obtain the same graph.
		
		Let us assume that $\Lambda \cap \Lambda' = \Delta \neq \emptyset$. Let
		$$
		\mc{G}(\Lambda)=G_{v_1}\times \dots\times G_{v_k}\times D_\infty^{i_1}\times \dots \times D_\infty^{i_l}; \quad
		\mc{G}(\Lambda')=G_{v_1'}\times \dots\times G_{v_m'}\times D_\infty^{i_1'}\times \dots \times D_\infty^{i_n'}.
		$$
		We claim that in this case
		\begin{itemize}
			\item  $\{v_1,\dots, v_k\}=\{v_1', \dots, v_m'\}$ and
			\item for any $i_j'$ either there exists $i_p$ so that $D_\infty^{i_j'}=D_\infty^{i_p}$ or $D_\infty^{i_j'}=G_{w_1}*G_{w_2}$, where $G_{w_1}\simeq G_{w_2}\simeq\BZ_2$ and $D_\infty^{i_p}=G_{w_2}*G_{w_3}$ and $G_{w_3}\simeq \BZ_2$, moreover the subgraph spanned by $w_1, w_2$ and $w_3$ is edgeless or $D_\infty^{i_p}=G_{w_3}*G_{w_4}$, where $G_{w_3}\simeq G_{w_4}\simeq\BZ_2$, moreover the subgraph spanned by $w_1, w_2, w_3$ and $w_4$ is edgeless.
		\end{itemize}
		Indeed, the first statement follows since $\Lambda \cap \Lambda' = \Delta \neq \emptyset$ and $\Lambda, \Lambda'$ are maximal.
		
		Suppose that for a group $D_\infty^{i_j'}=G_{w_1}*G_{w_2}$ (and $D_\infty^{i_p}=G_{w_3}*G_{w_4}$ for all $p$),  neither of the conclusions of the second statement hold. If $[D_\infty^{i_j'},D_\infty^{i_p}]=1$ for all $p$, then we have a contradiction with the maximality of $\Lambda$. Hence, there exists $p$ so that $[D_\infty^{i_j'},D_\infty^{i_p}]\neq 1$. In this case either $D_\infty^{i_j'}=G_{w_1}*G_{w_2}$ and $D_\infty^{i_p}=G_{w_2}*G_{w_3}$ or $D_\infty^{i_j'}=G_{w_1}*G_{w_2}$ and $D_\infty^{i_p}=G_{w_3}*G_{w_4}$. In the former case, there are no edges $(w_1,w_2)$ or $(w_2,w_3)$; if $[G_{w_1}, G_{w_3}]=1$, then the links of $w_1$ and $w_2$ are different which is a contradiction and hence the graph spanned by $w_1, w_2$ and $w_3$ is edgeless. The other case is similar.
	\end{proof}

	The following corollary follows from the discussion above and Theorem \ref{t: quantifier reduction}.
	\begin{cor}    \label{c: collapsed positive}
		Let $\mc{G}=\mc{G}(\Gamma, \gvv)$ be a graph product of groups such that $\Gamma$ is finite and $G_v$ satisfies a nontrivial positive sentence for all $v\in \Gamma$. Then $\mc{G}$ is a graph product $\mc{G}(\Gamma', \{G_v'\})$ where $\Gamma'$ is a positive reduced graph and each $G_v'$ satisfies a nontrivial positive $\forall\exists$-sentence.
	\end{cor}
	
	In the next lemma we show that in fact the positive reduced graph is optimal in the sense that if we further consider a graph product decomposition of the group whose underlying graph is smaller than the positive reduced one, then there is at least one vertex group with trivial positive theory.
	
	\begin{lemma}[The reduced graph is optimal]
		Let $\mc{G}=\mc{G}(\Gamma, \gvv)$ be a graph product where $\Gamma$ is positive reduced and $G_v$ satisfy a nontrivial positive sentence for all $v\in \Gamma$.
		
		Let $\phi:\Gamma \to \Delta$ be a proper full epimorphism of graphs. Let $\phi^*:\mathcal G(\Gamma, \gvv) \to \mathcal G(\Delta, \{G_{v'}\}_{v\in \Delta})$ be the induced epimorphism of groups where $G_{v'}$ is the subgroup of $\mc{G}$ generated by $G_v$ where $v\in \phi^{-1}(v')$, for all $v'\in \Gamma'$. Notice that by {\rm Lemma \ref{lem: full is iso}}, $\phi^*$ is an isomorphism.
		
		Then there exists a vertex $v'\in \Delta$ such that $G_{v'}$ has trivial positive theory.
	\end{lemma}
	\begin{proof}
		Let $v'$ be a vertex in $\Delta$ such that $\phi^{-1}(v')$ has more than one vertex. As $\phi$ is a proper epimorphism, such a vertex $v'$ exists. We claim that $G_{v'}$ has trivial positive theory.
		
		Consider the direct product decomposition of $G_{v'}$ (it may be trivial). As the graph $\Gamma$ is positive reduced it follows that one of the components of $G_{v'}$ has an irreducible action on a tree admitting stable hyperbolic elements so $G_{v'}$ has trivial positive theory, see \cite[Theorem 7.9]{casals2019positive}.
	\end{proof}
	
	We now observe that the positive reduced graph is, in some sense, the maximal graph product structure than one can recover. Note however, that we will show that if we require further restrictions on the singular vertex groups, then one may recover the reduced graph.
	
	\begin{example}
		Let $\mc{G}=\mc{G}(\Gamma, \gvv)$ be a graph product, where $\Gamma$ is the graph with two vertices $v$ and $w$ and one edge $(v,w)$. Let $G_v\simeq G_w \simeq \mathbb Q$. As $\mc{G}=\mathbb Q \times \mathbb Q \equiv \mathbb Q$, the basic underlying group-labelled graph $(\Gamma, \gvv)$ is not interpretable in $\mc{G}$ (as it should be interpretable in $\mathbb Q$).
	\end{example}
	
	The next lemma generalises the above example.
	
	\begin{lemma}[The positive reduced graph is optimal]
		Let $\mc{G}=\mc{G}(\Gamma, \gvv)$ be a graph product, where $\Gamma$ is positive reduced and $G_v$ satisfy a nontrivial positive sentence for all $v\in \Gamma$.
		
		Then the $\minCore(\Gamma,\gvv)$ and the vertex groups $G_v$, $v\in \minCore(\Gamma,\gvv)$ are interpretable in $\mc{G}$.
		
		Furthermore, there exists $\mc{G}=\mc{G}(\Gamma, \gvv)$ in the class for which $\Gamma$ is the maximal underlying structure that can be interpreted, that is, if $\pi:\Gamma' \to \Gamma$ is a full epimorphism of graphs, then the underlying structure associated to $\Gamma'$ cannot be uniformly interpreted, that is, there exists a graph product $\mc{G}'=\mc{G}(\Gamma', G_{v'})$ and a vertex $v'\in \minCore(\Gamma')$ so that $G_{v'}$ is not interpretable in $\mc{G}'$.
	\end{lemma}
	\begin{proof}
		The interpretability of the $\minCore(\Gamma,,\{G_{v}\}_{v\in \Gamma)}$ is proven in Subsection \ref{sec: levels}.
		
		Consider an epimorphism $\pi$ which identifies two vertices with the same star. Consider the graph product of groups with underlying graph $\Gamma'$ and all vertex groups are $\mathbb Q$. By our results we know that $\mathbb Q\times \mathbb Q$ is interpretable in $\mc{G}$. However, since $\mathbb Q$ is elementarily equivalent to $\mathbb Q\times \mathbb Q$, we have that $G_{v'}=\mathbb Q$ is not interpretable in $\mc{G}$.
	\end{proof}
	
	\subsection{Almost reduced graphs and non-generic almost positive theory}
	The next definition is similar to the previous one, with the only difference being that now we ask that there are no induced subgraphs $\Lambda = \Lambda_1 \times \dots \times \Lambda_r$ where \emph{one} (instead of all) factors $\mc{G}(\Lambda_i)$ is singular or non-singular infinite dihedral.
	
	\begin{definition}[Almost positive reduced graphs]\label{defn:almost reduced}
		Let $(\Gamma, \gvv)$ be a group-labelled graph. We say that $(\Gamma, \{G_v\}_{v\in \Gamma}\})$  is \emph{almost positive reduced}   if there are no induced subgraphs $(\Lambda, \{G_v\}_{v\in \Lambda})$ of  $(\Gamma, \gvv)$  such that
		\begin{itemize}
			\item  $\Lambda$ can be written as $\Lambda = \Lambda_1 \times \dots \times \Lambda_r$ where $r\geq 1$, $|V(\Lambda)|>1$ and, for \emph{some} $i$, $\mc{G}(\Lambda_i,\{G_v\}_{v\in \Lambda_i})$ is either a singular or non-singular dihedral subgroup of $\mc{G}(\Gamma, \{G_v\}_{v\in \Gamma}\}$ (i.e. either $|V(\Lambda_i)|=1$ or $\Lambda_i$ consists of two disconnected vertices $u_{i1}, u_{i_2}$ and $G_{u_{i1}}\cong \mbb{Z}/2\mbb{Z}\cong G_{u_{i2}}$.
			\item for all $v_1, v_2\in \Lambda$ we have that and $\st(v_1)\cap \Gamma\smallsetminus \Lambda=\st(v_2)\cap \Gamma\smallsetminus \Lambda$ (equivalently, $\st(v) < \Lambda \cup \link{\Lambda}$ for all $v\in \Lambda$).
		\end{itemize}
	\end{definition}

	\begin{remark}
		Notice that any pair  $(\Gamma, \{G_v\}_{v\in \Gamma}\})$  that is almost positive reduced is positive reduced and in both cases, the graph $\Gamma$ is reduced.
	\end{remark}
	
	We want to show that given a graph product with finite underlying graph $\Gamma$, one can obtain an almost positive reduced graph by identifying vertices and correspondingly changing the vertex groups, but preserving the fact that vertex groups satisfy a simple non-generic almost positive sentence.
	
	In the case of nontrivial positive theory, the key fact that allows to go to a positive reduced graph is the fact that the nontrivial positive theory is closed under direct products.
	
	For non-generic almost positive theory, a stronger preservation property holds: the direct product of groups satisfies a simple non-generic almost positive sentence if one of the factors also does.
	
	\begin{lemma}
		\label{l: almost positive direct sums}
		Let $G, H$ be groups and let $\phi$ be a simple non-generic almost positive sentence $\phi$ such that $G \models \phi$, then there exists a simple non-generic almost positive sentence $\phi'$ so that $G\times H \models \phi'$.
	\end{lemma}
	\begin{proof}
		By Theorem \ref{t: quantifier reduction} we may assume that
		$$
		\phi\equiv\exists z\,\Theta(z)\neq 1\wedge\psi(z),
		$$
		where the formula $\psi(z)$ is of the form $\forall x\exists y\,\Sigma(x,y,z)=1$ (note that $x,y$ are in general tuples of variables). Take now fresh tuples of variables $x^{0}$, $x^{1}$ with $|x^{0}|=|x^{1}|=|x|$ and let $\phi'\equiv\exists z\,\Theta(z)\neq 1\wedge\psi'(z)$, where $\psi'(z)$ is the positive formula:
		\begin{align*}
			\forall x\forall u\exists x^{0}\exists x^{1}\exists y \,\,\,\Sigma(x^{0},y,z)&=1
			\wedge\bigwedge_{1\leq j\leq |x|}(\bigwedge_{1\leq k\leq |x|}[(x^{0}_{k})^{u},x^{1}_{j}]=1\wedge[\Theta(z)^{u},x^{1}_{j}]=1\wedge x_{j}=x^{0}_{j}x^{1}_{j}).
		\end{align*}
		
		We show first that $\phi'$ is valid in any group of the form $G\times H$, where $G\models \phi$. It suffices to show that given $c\subseteq G\leq G\times H$ such that $G\models\psi(c)$ we have $G\times H\models \psi'(c)$.
		Indeed, given any $a\in G\times H^{|x|}$ let $a^{0}$ be the projection of $a$ to $G$ and $a^{1}$ the projection of $a$ to $H$.
		Clearly, there exists $b\in G^{|y|}$ such that $\Sigma(a^{0},b,c)=1$, since $G\models\psi(c)$. Since $[g^{f},h]=1$ for any $f\in G\times H$, $g\in G$ and $h\in H$, this suffices to witness $G\times H\models\psi'(c)$.
		
		We are left to show that $\phi'$ is non-generic, i.e. that for any formal solution of $\psi'(z)$ relative to a Diophantine condition $\exists v\,\xi(z,v)$ on $z$ implies the failure of the inequality $\Theta(z)\neq 1$. This amounts to showing the impossibility of the existence of a free product $L=K\frp\F(x,u)$ and and interpretations in $L$ of the tuples of variables $z,x^{0},x^{1},y$, which abusing the notation we will denote using the same letters,
		such that $z\subseteq K$ satisfies $\Theta(z)\neq 1$ and the equations in the quantifier-free part of $\psi'(z)$ are satisfied in $L$.
		
		Assume the situation above was possible. We will consider two actions of $L$ on trees with trivial edge stabilizers: $L\acts T$ in which $u$ and every $x_{j}$ are Bass-Serre elements (i.e. viewing $L$ as multiple HHN extension) and $L\acts T'$ in which elliptic subgroups are conjugate to either
		$K,\subg{x}$ or $\subg{u}$ (i.e. viewing $L$ as a free product). We will need the following observation:
		\begin{obs}
			\label{o: commuting conjugates} In the above notation, let $g_{1},g_{2}\in L$ be such that $[g_{1}^{u},g_{2}]=1$. Then at least one of the following conditions holds:
			\begin{enumerate}
				\item $g_{1}=1$ or $g_{2}=1$
				\item \label{alternative 2}  $g_{1}g_{2}$ is hyperbolic with respect to $T'$
				\item \label{alternative 3} $g_{1},g_{2}\in\subg{u}$
			\end{enumerate}
		\end{obs}
		\begin{subproof}
			Assume that $g_1$ and $g_2$ are nontrivial.
			Note that two elliptic elements with disjoint fixed point sets can never commute. Since edge stabilizers in $T$ and $T'$ are trivial, no elliptic element can commute with a hyperbolic element.
			
			So if $g_2$ is elliptic in $T$, then $g_1^u$ must stabilize the exact same vertex $v$ as $g_2$ does. Since $u$ is a Bass-Serre element, $g_1$ stabilises $v'=u^{-1}\cdot v$. Hence $g_1$ and $g_2$ stabilise different vertices in $T$, that is, they belong to different conjugates of $K$. Therefore, $g_1$ and $g_2$ also stabilise different vertices in $T'$ and so it follows that $g_1g_2$ is hyperbolic with respect to the action of $L$ in $T'$. Assume now that $g_2$ is hyperbolic in $T$ and so its centraliser is cyclic. Then, either $g_1^u=g_1$ and so we have the third alternative, or $g_1$ and $g_2$ are hyperbolic elements that stabilise different axis in $T$, which implies that $g_1g_2$ is hyperbolic in $T$ and is not a power of a conjugate of a Bass-Serre element, so hyperbolic in $T'$.
		\end{subproof}
		
		We start by using the previous observation to prove the following Lemma:
		\begin{lemma} \label{lem:aux}
			Either for any $1\leq j\leq |x|$ we have $x^{0}_{j}=x_{j}$ {\rm(}and thus $x^{1}_{j}=1${\rm)} or for any $1\leq j\leq |x|$ we have $x^{1}_{j}=x_{j}$ {\rm(}and thus $x^{0}_{j}=1${\rm)}.
		\end{lemma}
		\begin{subproof}
			We apply Observation \ref{o: commuting conjugates} to the pair $x^{0}_{j},x^{1}_{j}$. Alternatives \ref{alternative 2} and \ref{alternative 3} are clearly ruled out by the fact that $x^{0}_{j}x^{1}_{j}=x_{j}$ and $x$ is elliptic with respect to the action of $L$ on $T'$.
			Hence for each $1\leq j\leq |x|$ either $x^{0}_{j}=x_{j}$ or $x^{1}_{j}=x_{j}$.
			We are left to show that which of the alternatives hold does not depend on $j$, i.e. that either $x=x^0$ or $x=x^1$.
			Assume that this is not the case, then without loss of generality we have $x^{0}_{j}=x_{j}$ and $x^{1}_{k}=x_{k}$ for some $j,k$. Since $x$ is elliptic in the action of $G$ on $T'$, so are $x^0_j$ and $x^1_k$ and they fix the same vertex as does $x$. Since  $[(x^{0}_{j})^{u},x^{1}_{k}]=1$ and $x_{j}x_{k}$ is neither hyperbolic in $T'$ nor lies in $\subg{u}$ we arrive to contradiction with Observation \ref{o: commuting conjugates}.
		\end{subproof}
		
		We are now in position to finish the proof and show that $\phi'$ is non-generic or, equivalently, that there cannot exist witnesses $z,x^0,x^1,y$ in $L$ such that $z\subseteq K$ satisfies $\Theta(z)\neq 1$ and the equations in the quantifier-free part of $\psi'(z)$ are satisfied. Indeed, assume that $\Theta(z)\neq 1$ in $L$ and $\Sigma(x^{0},y,z)=1$. By Lemma \ref{lem:aux}, we have that either $x=x^0$ or $x=x^1$. In the former case, our data would give us a formal solution to $\psi(z)$ relative to a Diophantine condition compatible with $\Theta(z)\neq 1$, contradicting the assumption that $\phi$ is non-generic. In the latter case, we have that $x^{1}_{j}=x_{j}$ and so $[\Theta(z),x^{1}_{j}]\neq 1$ in $L$ contradicting the fact that $L$ satisfies $\phi'$.
	\end{proof}
	
	\begin{remark}
		Notice that even if $\phi$ in {\rm Lemma \ref{l: almost positive direct sums}} is a positive sentence {\rm(}which is a particular case of simple non-generic almost positive sentences{\rm)}, the sentence $\phi'$ is simple non-generic almost positive. If $G$ has nontrivial positive theory, $G\times H$ may have trivial positive theory. Indeed, $\mathbb Z$ has nontrivial positive theory, however $\mathbb Z \times F_2$ has trivial positive theory but it satisfies a non-generic almost positive sentence, for instance, it has nontrivial center.
	\end{remark}
	
	\begin{lemma}[Almost positive reduced graphs]
		Let $\mc{G}=\mc{G}(\Gamma, \gvv)$ be a graph product such that $\Gamma$ is finite and $G_v$ satisfies a simple non-generic almost positive sentence for all $v\in \Gamma$. Then $\mc{G}$ is a graph product $\mc{G}(\Gamma', \{G_v'\})$ where $\Gamma'$ is an almost positive reduced graph and $G_v'$ satisfies a simple non-generic almost positive sentence for all $v\in \Gamma'$.
	\end{lemma}
	\begin{proof}
		The proof is analogous to the proof of Lemma \ref{lem: graph is positive reduced}.
	\end{proof}
	
	\smallskip
	
	The following corollary follows from the discussion above and Theorem \ref{t: quantifier reduction}.
	\begin{cor}
		\label{c: collapsed almost positive}
		Let $\mc{G}=\mc{G}(\Gamma, \gvv)$ be a graph product of groups such that $\Gamma$ is finite and $G_v$ satisfies a simple non-generic almost positive sentence for all $v\in \Gamma$. Then $\mc{G}$ is a graph product $\mc{G}(\Gamma', \{G_v'\})$, where $\Gamma'$ is an almost positive reduced graph and each $G_v'$ satisfies a simple non-generic almost positive sentence whose positive part is an $\forall\exists$-formula.
	\end{cor}
	
	In the next lemma we show that in fact the definition of almost positive reduced graph is optimal in the sense that if we consider a graph product decomposition of the group whose underlying graph is a proper quotient of an almost positive reduced one, then there is at least one vertex group that does not satisfy a simple non-generic almost positive sentence.
	
	\begin{lemma}[The almost positive reduced graph is optimal]
		Let $\mc{G}=\mc{G}(\Gamma, \gvv)$ be a graph product where $\Gamma$ is almost positive reduced and $G_v$ satisfy a simple non-generic almost positive sentence for all $v\in \Gamma$.
		
		Let $\phi:\Gamma \to \Delta$ be a proper full epimorphism of graphs. Let $\phi^*:\mathcal G(\Gamma, \gvv) \to \mathcal G(\Delta, \{G_{v'}\}_{v\in \Delta})$ be the induced epimorphism of groups where $G_{v'}$ is the subgroup of $\mc{G}$ generated by $G_v$ where $v\in \phi^{-1}(v')$, for all $v'\in \Gamma'$. Notice that by {\rm Lemma \ref{lem: full is iso}}, $\phi^*$ is an isomorphism.
		
		Then there exists a vertex $v'\in \Delta$ such that $G_{v'}$ does not satisfy a simple non-generic almost positive sentence.
	\end{lemma}
	\begin{proof}
		Let $v'$ be a vertex in $\Gamma$ such that $\pi^{-1}(v')$ is more than one vertex. As $\phi$ is a proper epimorphism, such a vertex $v'$ exists. We claim that $G_{v'}$ does not satisfy a simple non-generic almost positive sentence.
		
		Consider the direct product decomposition of $G_{v'}$ (it may be trivial). As the graph $\Gamma$ is almost positive reduced it follows that none of the factors of $G_{v'}$ is either singular or infinite dihedral in $\Gamma$, that is, all the components of $G_{v'}$ admit irreducible actions on the associated vertex trees. Therefore, by Lemma \ref{l: simple parameter case} we have that $G_{v'}$ cannot satisfy a non-generic almost positive sentence.
	\end{proof}
	
	\subsection{Cores, completions and extensions}

	\begin{definition}[Weak vertices]\label{defn:weak vertices}
		Let $(\Gamma,\gvv)$ be a group-labelled graph. Let $v$ and $w$ be two vertices of $\Gamma$. We set $v\le w$ if and only if $\st(v)\subseteq \st(w)$. A vertex $w$ is \emph{minimal} if there are no vertices $v$ so that $v\lneq w$.
		
		A vertex $v$ in the label graph $(\Gamma,\gvv)$ is called \emph{weak} if the following two conditions hold:
		\begin{itemize}
			\item there exist vertices $\{w_1,\dots,w_k\}$  such that $(v,w_i)\notin E(\Gamma)$ and $\cap_{i=1,\dots,k}\st(w_i) = \link{v}$. In this case, we say that $W=\{w_1, \dots, w_k\}$ is  \emph{a witness} for the weak vertex $v$;
			\item the group $G_v$ associate to $v$ is elementarily equivalent to $\mathbb Z$.
		\end{itemize}
	\end{definition}
	
	Notice the witness of a weak vertex group may not be unique. In fact, if $W$ is a witness for $v$, then so is $W\cup \{v\}$.
	
	\begin{definition}[Minimal Core Graph] \label{denf: core graph}
		
		Let $(\Gamma,\gvv)$ be a group-labelled graph. \emph{The minimal core} of $(\Gamma, \{G_v\}_{v\in \Gamma})$, denoted by $\minCore(\Gamma,  \{G_v\}_{v\in \Gamma})$ is the group-labelled graph defined as follows: the underlying graph $\minCore(\Gamma)$ is the induced subgraph of $\Gamma$ spanned by the non-weak vertices of $\Gamma$; and the vertex groups are the following ones: for all $v\in \minCore(\Gamma, \gvv)$ such that $G_v\not\equiv D_\infty$, we assign $G_v$ to the vertex $v$; and for all $v\in \minCore(\Gamma, \gvv)$ such that $G_v\equiv D_\infty$, we associate $D_\infty$ to $v$.
		
		Since the vertex groups of the group-labelled graph $\minCore(\Gamma,\gvv)$ are determined by $\gvv$ (although they are not induced), we abuse the notation and omit the explicit reference to them.
	\end{definition}
	
	\begin{example}
		Consider the group-labelled graph $(\Gamma, \gvv)$.
		
		\begin{itemize}
			\item If the graph $\Gamma$ is edgeless, the minimal core graph contains precisely the vertices whose associated groups are not elementarily equivalent to $\mathbb Z$. In particular, for the graph associated to a free group the core is trivial and for the one associated to a free product of free abelian groups of rank more than 1, the minimal core is the whole graph.
			
			\item In a complete graph, all vertices are minimal but not weak so the minimal core coincides with the graph.
			
			\item In a tree with associated vertex groups $\mathbb Z$, leaves are precisely the weak vertices and so the minimal core are the vertices that are not leaves, i.e. vertices with degree more than 1.
			
			\item In a cycle of length more than 4, all vertices are not weak and so the minimal core is the cycle itself. In the cycle of length 4, if the vertex groups are $\mathbb Z$, then all vertices are weak and so the minimal core is trivial.
			
			\item In {\rm Figure \ref{fig:expl}}, we mark minimal and weak vertices in two graphs with associated vertex groups $\mathbb Z$.
		\end{itemize}
		
		\tikzset{node distance=2.5cm, 
			every state/.style={minimum size=1pt, scale=0.3, 
				semithick,
				fill=black},
			initial text={}, 
			double distance=2pt, 
			every edge/.style={ 
				draw,  
				auto}}
		\begin{figure}[!ht] \label{fig:expl}
			\begin{center}
				\begin{tikzpicture}
					\footnotesize
					\node[state, color=red] (1) at (0,0.66) {};  \filldraw (0,0.66) node[align=right,right] {$a_1$};
					\node[state, color=red] (2) at (0,0) {};  \filldraw (0,0) node[align=right,below] {$a_2$};
					\node[state] (3) at (1,0) {};  \filldraw (1,0) node[align=left,above] {$b$};
					\node[state] (4) at (2,0) {};  \filldraw (2,0) node[align=left,above] {$c$};
					\node[state] (5) at (3,0) {};  \filldraw (3,0) node[align=left,above] {$d$};
					\node[state] (6) at (4,0) {};  \filldraw (4,0) node[align=left,above] {$e$};
					\node[state, color=red] (7) at (5,0) {};  \filldraw (5,0) node[align=left,above] {$f$};
					\node[state, color=blue] (8) at (2.5,0.5) {};  \filldraw (2.5,0.5) node[align=right, right] {$x$};
					\draw
					(1) edge (3)
					(3) edge (4)
					(4) edge (5)
					(5) edge (6)
					(6) edge (7)
					(2) edge (3)
					(4) edge (8)
					(5) edge (8);
				\end{tikzpicture}
				
				\vspace{0.5cm}
				\begin{tikzpicture}
					\footnotesize
					\node[state, color=red] (1) at (0,0.66) {};  \filldraw (0,0.66) node[align=right,right] {$a_1$};
					\node[state, color=red] (2) at (0,0) {};  \filldraw (0,0) node[align=right,below] {$a_2$};
					\node[state] (3) at (1,0) {};  \filldraw (1,0) node[align=left,above] {$b$};
					\node[state] (4) at (2,0) {};  \filldraw (2,0) node[align=left,above] {$c$};
					\node[state] (5) at (3,0) {};  \filldraw (3,0) node[align=left,above] {$d$};
					\node[state] (6) at (4,0) {};  \filldraw (4,0) node[align=left,above] {$e$};
					\node[state, color=red] (7) at (5,0) {};  \filldraw (5,0) node[align=left,above] {$f$};
					\node[state, color=red] (8) at (2.5,0.5) {};  \filldraw (2.5,0.5) node[align=right,right] {$x$};
					\node[state, color=red] (9) at (2.5,-0.5) {};  \filldraw (2.5,-0.5) node[align=right,right] {$y$};
					
					\draw
					(1) edge (3)
					(3) edge (4)
					(4) edge (5)
					(5) edge (6)
					(6) edge (7)
					(2) edge (3)
					(4) edge (8)
					(5) edge (8)
					(4) edge (9)
					(5) edge (9);
				\end{tikzpicture}
				\caption{\small Weak vertices in the graphs are marked red; minimal non-weak vertices are marked blue.} \label{fig:expls}
			\end{center}
		\end{figure}
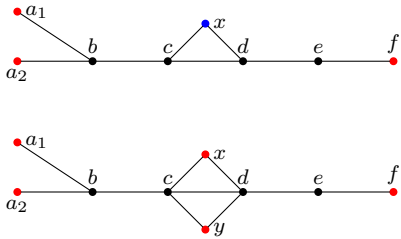
	\end{example}
	\begin{remark}
		Any weak vertex of $(\Gamma, \gvv)$ is necessarily minimal with respect to $\leq$. Indeed, suppose that $W$ is a witness for a weak vertex $v$ and that $v$ is not minimal. Since $v$ is not minimal, there exists a vertex $u\neq v$ such that $\st(u)\subsetneq \st(v)$. By definition, $u\in \link{v}= \st(v)\cap \cap_{w\in W} \st(w)$. This implies that for all $w\in W$ we have that either $u=w$ or $(u,w)\in E(\Gamma)$. The first case cannot occur because $u\in \link{v}$ but $(w, v)\not\in E(\Gamma)$ for all $w\in W$. So assume the second alternative, that is $(u,w)\in E(\Gamma)$ for all $w\in W$. Note that if $v\in W$, then there exists $w\in W$ such that $v\ne w$. If $(u,w)\in E(\Gamma)$, then $w\in \link{u}$ and since $\st(u)\subseteq \st(v)$, we have $(v,w)\in E(\Gamma)$ contradicting the definition of the witness $W$.
	\end{remark}
	
	\begin{remark}
		Notice that if in a graph $\Gamma$ there are two vertices $v$ and $v'$ with the same star, then these vertices are identified in the reduced graph. In particular, in a reduced graph all vertices are different in the partial order.
		
		Notice that if two weak vertices are equal in the partial order $\leq$, their image in the positive reduced graph becomes non-weak. Indeed, by definition the groups associated to the weak vertex groups are elementarily equivalent to $\mathbb Z$ and so the group associated to the vertex group {\rm(}corresponding to the collapsed weak vertices{\rm)} would be elementarily equivalent to $\mathbb Z^n$, where $n>1$ is the number of weak vertices collapsed, and so the group would not be elementarily equivalent to $\mathbb Z$.
	\end{remark}

	\begin{definition}[Core of a graph] \label{defn:ecore}
		Let $(\Gamma, \gvv)$ be a group-labelled graph. A vertex $v\in \Gamma$ is called \emph{redundant} if it is weak and there exist vertices $\{w_1,\dots,w_k\}$ such that $v\ne w_i$, $i=1,\dots, k$ and $\cap_{i=1,\dots,k}\st(w_i) = \link{v}$.

		We define the \emph{group-labelled core of} $\Gamma$, denoted by $\Core(\Gamma, \gvv)$, as follows: \emph{the core of} $\Gamma$, denoted by $\Core(\Gamma)$, is a maximal induced subgraph of $\Gamma$  without redundant vertices. Notice that by definition, $\minCore(\Gamma)\le \Core(\Gamma)$. The graph $\Core(\Gamma)$ is well-defined up to isomorphism, see Lemma \ref{lem:ecore}.

		The vertex groups associated to the vertices of $\Core(\Gamma)$ are defined as follows: for all $v\in \minCore(\Gamma, \gvv)$, we associate the corresponding $G_v$ in $\minCore(\Gamma, \gvv)$; for all weak vertex $v\in \Gamma$, $v\notin \minCore(\Gamma)$, we assign to it the infinite cyclic group $\mathbb Z$. (Notice that as the vertex is weak, the associated group in $\Gamma$ is elementarily equivalent to $\mathbb Z$).
		
		As we did with the $\minCore(\Gamma, \gvv)$, since the vertex groups of the group-labelled graph $\Core(\Gamma,\gvv)$ are determined by $\gvv$, we abuse the notation and omit the explicit reference to them.
	\end{definition}

	\begin{example}
		\,
		Consider $(\Gamma, \gvv)$ where $G_v\simeq \mathbb Z$.
		\begin{itemize}
			\item Let $\Gamma$ be the edgeless graph with $k$ vertices. The core is an edgeless graph with 2 vertices. Indeed, the minimal core of $\Gamma$ is trivial as all vertices are weak. The vertices in the graph with 2 vertices are not redundant; on the other hand, the graph with 3 vertices has one redundant vertex.
			
			\item More generally, let $\Gamma$ be the disjoint union of graphs $\Gamma_i$. If there are at least two $\Gamma_i$ with more than one vertex, then all the isolated vertices are redundant. If there is only one component with more than one vertex, then at most one isolated vertex is non redundant. By the previous example, if all the $\Gamma_i$ are isolated vertices, then the core graph is precisely 2 isolated vertices.
			
			\item Let $\Gamma$ be a graph with diameter greater than 2 and let $\Gamma'$ be the disjoint union of $\Gamma$ with a finite set of isolated vertices. Then the isolated vertices are redundant in $\Gamma'$.
			
			\item In all cycles, the core is the cycle itself. Indeed, if the cycle is of length more than 4, then the minimal core (and so the core) is the cycle. If the cycle is of length less than 4, then it is complete and the minimal core is the cycle. If the cycle is of length 4, then all the vertices are weak but non-redundant, so the minimal core is trivial but the core is the cycle.
			
			\item If $\Gamma$ is a tree of diameter more than 2, then the core is obtained from the minimal core by joining each leaf of the minimal core with a new vertex, which is a leaf of $\Gamma$. In particular, the core of a path is the path. If $\Gamma$ has diameter 2, then the minimal core is a unique vertex and the core is the path with 3 vertices.
		\end{itemize}
	\end{example}
	
	\begin{lemma}\label{lem:ecore}
		The core of a positive reduced graph exists and it is unique up to isomorphism.
	\end{lemma}
	\begin{proof}
		We show that one can remove redundant vertices one-by-one and that the graph we obtained is unique (up to isomorphism) independent of the order we remove the redundant vertices.
		
		It suffices to show that given any two vertices $v$ and $v'$ which are redundant in $\Gamma$, then we can either remove both $v$ and $v'$ or if removing  $v$ makes $v'$ not redundant, then the converse also holds and the resulting graphs (removing $v$ and removing $v'$) are isomorphic.
		
		If $v$ is redundant in $\Gamma'=\Gamma \smallsetminus \{v'\}$, then the statement is clear. So, suppose that  $v$ is not redundant in $\Gamma'$.

		By definition, since redundant vertices are weak, there exists $w_1',\dots, w_l'$ and $w_1,\dots, w_k\in\Gamma$ so that
		$$
		\lk_\Gamma(v')=\bigcap_{i=1,\dots,l}\st_{\Gamma}(w_i')=\hbox{ and }
		\lk_\Gamma(v)=\bigcap_{i=1,\dots,k}\st_{\Gamma}(w_i).
		$$
		Since $v$ is not redundant in $\Gamma'$, we have that $w_i=v'$ for some $i=1,\dots, k$. Therefore, $\lk_\Gamma(v)\subseteq \st_\Gamma(v')$.
		
		We consider two cases. Firstly, suppose that $(v,v')\notin E(\Gamma)$. Then,
		$$
		\lk_\Gamma(v)=\bigcap_{i=1,\dots,l}\st_{\Gamma}(w_i')\cap \bigcap_{i=1,\dots,k}\st_{\Gamma}(w_i).
		$$
		If $w_j'\neq v$ for all $j=1,\dots, l$, then
		$$
		\lk_{\Gamma'}(v)=\bigcap_{i=1,\dots,l}\st_{\Gamma'}(w_i')\cap \bigcap_{i=1,\dots,k}\st_{\Gamma'}(w_i),
		$$
		which contradicts our assumption of $v$ being redundant, hence this case is impossible and $w_j'=v$ for some $j$.
		
		We now have $\lk_\Gamma(v)=\bigcap_{i=1,\dots,k}\st_{\Gamma}(w_i)\cap \st_\Gamma(v')$ and $\lk_\Gamma(v')=\bigcap_{i=1,\dots,l}\st_{\Gamma}(w_i')\cap \st_{\Gamma}(v)$. Therefore,  $\link{v'}= \link{v}$ and there exists an automorphism of $\Gamma$ sending $v$ to $v'$.
		
		Suppose that $(v,v')\in E(\Gamma)$. Then,
		$$
		\link{v}= \st_\Gamma(v')\cap \bigcap_{i=1,\dots,k}\st_{\Gamma}(w_i)=(\bigcap_{i=1,\dots,l}\st_{\Gamma}(w_i')\cup\{v'\}) \cap \bigcap_{i=1,\dots,k}\st_{\Gamma}(w_i).
		$$
		If $w_j'\neq v$ for all $j=1,\dots, l$, then $\lk_{\Gamma'}(v)=\bigcap_{i=1,\dots,l}\st_{\Gamma'}(w_i')\cap \bigcap_{i=1,\dots,k}\st_{\Gamma'}(w_i)$,
		which contradicts our assumption of $v$ being redundant, hence this case is impossible and $w_j'=v$ for some $j$.
		
		We now have
		the    $$
		\lk_\Gamma(v)=\bigcap_{i=1,\dots,k}\st_{\Gamma}(w_i)\cap (\lk_{\Gamma}(v') \cup \{v'\})=\bigcap_{i=1,\dots,k}\st_{\Gamma}(w_i)\cap (\bigcap_{i=1,\dots,l}\st_{\Gamma}(w_i') \cup \{v'\})
		$$ and
		$$
		\lk_\Gamma(v')=\bigcap_{i=1,\dots,l}\st_{\Gamma}(w_i')\cap (\bigcap_{i=1,\dots,k}\st_{\Gamma}(w_i) \cup \{v\}).
		$$
		It follows that $\st(v)=\st(v')$. We note that in this case there exists an automorphism of $\Gamma$ sending $v$ to $v'$, but, in our setting, $\st(v)=\st(v')$ is impossible, since we assume that the graph $\Gamma$ is reduced.
	\end{proof}
	
	The following lemma follows directly from definitions.
	\begin{lemma}
		Let $(\Gamma, \gvv)$ be a group-labelled graph and let $\Core(\Gamma, \gvv)$ be the group-labelled core of $\Gamma$. Then, $(\Gamma, \gvv)$ and $\Core(\Gamma, \gvv)$ have the same minimal core and the same core, that is $\minCore(\Gamma, \gvv)=\minCore(\Core(\Gamma), \gvv)$ and  $\Core(\Core(\Gamma),\gvv)=\Core(\Gamma, \gvv)$.
		
	\end{lemma}
	
	\begin{definition}[$d$-completion]\label{defn:D-completion}
		Let $\G(\Gamma, \gvv)$ be a graph product. We define the \emph{$d$-completion} of $\G$ as the group-labelled graph $(\Gamma^d, \{E_v\}_{v\in \Gamma^d})$ where:
		\begin{itemize}
			\item The set of vertices  $V(\Gamma^d)$ is in bijective correspondence with the set consisting of the following sets:
			\begin{itemize}
				\item singular subgroups $H_v$ for $v\in \Gamma^e$;
				
				\item The set $D_{\max}$ of maximal (by inclusion) non-singular dihedral and infinite cyclic subgroups  $H$ of $\G$ such that $|\ff(H)|=1$; more precisely, let
				\begin{align*}
					D=&\{ \subg{g} \mid g\in \G, g \notin \bigcup_{v\in \Gamma^e} H_v, \  |\ff(g)|=1\} \cup \\ &\cup\{ \subg{g,h}\mid g^2=h^2=1,\ [g,h]\ne 1,\ |\ff(g,h)|=1,\ \subg{g,h} \not< \bigcup_{v\in \Gamma^e}H_v\},
				\end{align*}
				then  $D_{\max}$ is the subset of maximal (by inclusion) sets from $D$.
				
				\item the set $C_2$ of singular involutions dominated by a dihedral group, that is, $C_2$ is the set of $g\in \G$ such that
				\begin{itemize}
					\item (involutions) $g\ne 1$ and $g^2=1$;
					\item (singular) $|\ff(g)|=1$;
					\item (contained and dominated by a dihedral group) if $g\in H_v$ for some $v\in \Gamma^e$, then either $H_v \simeq \subg{g}$ and there exists $E_w \in D_{\max}$ such that $C(g) < E_w \times E_w^\perp$; or $H_v \equiv D_\infty$.
				\end{itemize}
				
			\end{itemize}  Given $v\in \Gamma^d$, we denote by $E_v$ the singular, dihedral or cyclic (of order 2 or infinite) subgroup associated to $v$.
			
			\item $E(\Gamma^d)$ are determined by subgroup commutation: $(v,v')\in E(\Gamma^d)$ if and only if $[E_v, E_{v'}]=1$.
		\end{itemize}
	\end{definition}
	
	\begin{lemma}
		If $\Gamma$ satisfies the property $AP_n$, so does the graph $\Gamma^d$.
	\end{lemma}
	\begin{proof}
		By Proposition \ref{prop:extensionAP}, the graph $\Gamma^e$ has property $AP_n$.
		
		Let $p$ be an induced path in the complement graph of $\Gamma^d$. Our goal is to show that $p$ has length bounded by $n$. We prove it by induction on the number of vertices $v\in \Gamma^d\setminus 
		\Gamma^e$ in the path $p$. If there are none, then $p$ is a path in the complement of $\Gamma^e$ and since $\Gamma^e$ satisfies property $AP_n$, the length of the path is bounded by $n$.
		
		Let $v$ be a vertex of $p$ so that $v\in \Gamma^d\smallsetminus \Gamma^e$. 
		
		Assume first that $E_v$ is a group of order $2$. From the definition of the $d$-completion, $E_v$ is singular and so $E_v < H_{v'}$ for some $v'\in \Gamma^e$. In this case, we can be replace the vertex $v$ in $p$ by the vertex $v'$ from $\Gamma^e$. The resulting path $p'$ still defines an induced path in the complement graph of $\Gamma^d$, has the same length as $p$, but contains strictly less vertices from $\Gamma^d\setminus \Gamma^e$ and so by induction, we conclude that its length is bounded by $n$.
		
		Suppose now that $E_v$ corresponds to either a maximal non-singular cyclic or dihedral subgroup of $\G(\Gamma)$. Let $\G(\Delta)^g=\supp(E_v)$. Since $\Gamma^d$ is invariant up to conjugation, we can assume without loss of generality that $g=1$, that is $\supp(E_v)=\G(\Delta)$ for some $\Delta < \Gamma$. 
		
		We claim that there exist vertices $v_1,\dots, v_k \in \Delta$, which define a path in the complement of $\Gamma$ and so that the induced subgraph in the complement of graph of $\Gamma^d$ defined by the vertices $(V(p)\setminus \{v\}) \cup \{v_1,\dots, v_k\}$ is a path. Thus, we obtain a path $q$ whose length is bigger than or equal to that of $p$ and with less vertices in $\Gamma^d\setminus \Gamma^e$ than $p$. Therefore, by induction, we deduce that the length of $q$ (and so the length of $p$) is bounded by $n$.
		
		Let us prove the claim. Observe that if $v'\in \Delta$ and $w\in \Gamma^d$ are so that $(v',w)\notin E(\Gamma^d)$, then $(v,w)\notin E(\Gamma^d)$ as $\supp(E_v)=\G(\Delta)$ and $star_{\Gamma^d}(v) = \bigcap\limits_{v'\in \Delta} star_{\Gamma^d}(v')$. Let $u_1$ and $u_2$ be the two vertices of $p$ (only $u_1$ if $v$ is a leaf of $p$) connected to $v$, that is $(v,u_1), (v,u_2)\notin E(\Gamma^d)$. Since $(v,u_1), (v,u_2)\notin E(\Gamma^d)$, there exist $v_1, v_k\in \Delta$ (not necessarily different) such that $(u_1,v_1), (u_2,v_k)\notin E(\Gamma^d)$.  Since $|\ff(g)|=1$, the complement of $\Delta$ is connected and so it follows that there exist vertices $v_2, \dots, v_{k-1}\in \Delta$ so that $v_1,\dots, v_k$ define a path in the complement of $\Delta$ and so in the complement of $\Gamma$. Finally, since for all $u$ in $p$, $u\ne u_1, u_2$ we have that $(u,v)\in E(\Gamma^d)$, by the description of the $star(v)$ it follows that $(u,v_i)\in E(\Gamma^d)$ for all $i=1, \dots, k$. Therefore, the induced subgraph defined by $V(p)\setminus \{v\} \cup \{v_1, \dots, v_k\}$ is a path in the complement of $\Gamma^d$ that has less vertices in $\Gamma^d\setminus \Gamma^e$ than $p$ and so the Claim is proven.
	\end{proof}
	
	\begin{definition}[Partial preorder in $\Gamma^d$] \label{d:d-order}
		We define a partial preorder in $\Gamma^d$, denoted by $\preceq^d$ as follows: $v \preceq^d v'$ if and only if $E_v \times E_v^\perp \subseteq E_{v'}\times E_{v'}^\perp$. We write $v \prec^d v'$ if $E_v \times E_v^\perp \subsetneq E_{v'}\times E_{v'}^\perp$ and $v=^d v'$ if $E_v \times E_v^\perp = E_{v'}\times E_{v'}^\perp$.
	\end{definition}
	
	\begin{example}
		Let $\G$ be the right-angled Coxeter group defined by a pentagon with generators $a_i$, $i=0, \dots, 4$ and relations $[a_i,a_{i+1}]=1 \, (mod \, 5)$ . Then $\subg{a_0,a_2}$ is a maximal non-singular dihedral subgroup and $|\ff(a_0,a_2)|=1$, hence it corresponds to a vertex $v\in \Gamma^d$. Let $v'$ be the vertex associated to the singular vertex group generated by $a_1$. Then $v' =^d v$.
	\end{example}

	\begin{remark}\label{rem:properties Gammad}\,
		
		\begin{itemize}
			
			\item  Notice that the extension graph $\Gamma^e$ is an induced subgraph of $\Gamma^d$.
			
			\item For all $v\in \Gamma^d$, we have that $E_v^\perp$ is a subgraph product, i.e. if $E_{v'} < E_v^\perp$, then either $v'\in \Gamma^e$ or $E_v$ is a dihedral group of the from $G_w \ast G_{w'}$ where $w,w'\in \Gamma^e$.
			
			\item if $v' \preceq^d v$ for some $v,v'\in \Gamma^e$, and $E_v$ is infinite cyclic and $E_{v'}$ is singular, then there exists $w\in \Gamma^e$ such that $E_w$ is singular and $v' \prec^d w$. Indeed, if $E_v\times E_v^\perp < E_{v'} \times E_{v'}^\perp$ and $E_{v}$ is a non-singular infinite cyclic subgroup, then $E_{v'}\times E_{v'}^\perp < E_{v}^\perp$. Let $\{\G(\Delta)^h\} = \ff(E_{v})$. Then it suffices to take as $w$ any vertex in $\Delta^h < \Gamma^e$ since $E_{v}^\perp < E_w^\perp$ and so $E_{v'}\times E_{v'}^\perp < E_w \times E_w^\perp$.
			
			\item if $v' \preceq^d v$, $E_v$ is dihedral and $E_{v'}$ is singular, then either $E_{v'}<E_v$ or there exists $w\in \Gamma^e$ such that $E_w$ is singular and $v' \prec^d w$.
		\end{itemize}
	\end{remark}

	\begin{definition}
		Let $\mathcal G=\mathcal G(\Gamma, \gvv)$ be a graph product. A vertex $v\in \Gamma^d$ is called $\preceq^d$-maximal if it is maximal in the preorder $\preceq^d$, that is, there is no $w\in \Gamma^d$ such that $v \prec^d w$.
	\end{definition}
	
	We next introduce an induced subgraph of $\Gamma^d$, the $c$-completion of $\Gamma$.
	
	\begin{definition}[$c$-completion]\label{defn:Gamma^c}
		Let $\Gamma^c=(V^c,E^c)$ be the induced subgraph of $\Gamma^d$ that contains $\Gamma^e$ and the vertices $v\in \Gamma^d$ such that $E_v$ is either infinite non-singular cyclic or cyclic of order 2.
	\end{definition}
	
	Recall that our goal is to find an elementary invariant of an appropriate class of graph product, the $\Core(\Gamma, \gvv)$. However, in order to interpret $\Core(\Gamma, \gvv)$ as a group-labelled graph, we actually interpret a larger graph, the extended core, that we introduce in the next definition.
	
	Roughly speaking, the extended core is a graph that can be obtained from the core by doing two operations; firstly, we replace each singular vertex group elementarily equivalent to an infinite dihedral group by a dihedral group. Secondly, we  consider the different equivalence classes of weak vertex groups and non-singular cyclic subgroups (from vertices in $\Gamma^c$) that have the same link up to conjugacy. If the equivalence class is not represented in the core of $\Gamma$ by a singular vertex group, we add a vertex for that equivalence class and connect it to the common orthogonal of all the cyclic subgroups in the class.
	
	For instance, if we consider $(\Gamma, \{\mathbb Z\}_{v\in \Gamma})$ where $\Gamma$ is the path with 4 vertices (which we denote by $a,b,c,d$ and $a$ and $d$ are the leaves), then the core of $\Gamma$ is precisely $\Gamma$. The equivalent classes of cyclic subgroups (either singular or corresponding to vertices in $\Gamma^c$) are classified by their orthogonal (and up to conjugacy) as follows: the class whose orthogonal is generated by either $b$ or $c$ or it is trivial. The class whose orthogonal is $b$ is represented by the weak vertex groups generated by $a$ in $\Core(\Gamma, \gvv)$; similarly, the class whose orthogonal is generated by $c$ is represented by the weak vertex generated by $d$. However, the class of cyclic subgroups with trivial orthogonal is not represented by any singular vertex group in $\Core(\Gamma, \gvv)$. Therefore, we add a vertex to $\Core(\Gamma, \gvv)$ which is isolated (as the orthogonal of the class is trivial) and assign as a vertex group and infinite cyclic group.
	
	Although the extend core could be described as above, we decided to define it in a somehow more convoluted way that follows the steps used for its interpretation in the class of graph products we consider.
	
	The interpretation of the extended core follows this steps. First we show that the (labelled) graph $\Gamma^c$ is interpretable; we then consider the quotient of $\Gamma^c$ by two equivalence relations, one defined by conjugation and the second one on (groups elementarily equivalent to) infinite cyclic subgroups that have the same orthogonal. This may produce a quotient that is too small - indeed if we consider the free group of rank $n>1$, all cyclic subgroup have trivial orthogonal so they all belong to the same equivalence class. In this case, the quotient would be a unique vertex group while the core contains two vertices. This phenomenon occurs when in the graph there are vertices which although being weak (and so they do not belong to the $\minCore$), they are not redundant (so they stay in the core). We will identify this equivalence classes using the first-order theory and for them we will add two vertices representing this class. However, for other classes, for instance the class of cyclic subgroup with trivial orthogonal in the example of the path above, if they have a witness for the weak vertices that belongs to the $\minCore$ of $\Gamma$, then one only adds one vertex to represent this class.
	
	There is yet another technical problem apparent in the interpretation - to distinguish using first-order sentences the  singular dihedral groups from the non-singular ones (for instance, the ones defined by two singular involutions that belong to vertex groups that do not commute). In order to interpret the singular dihedral groups, we use properties of their involutions. Roughly speaking, this brings us to have to make visible the involutions that belong to singular vertex groups (which makes the graph $I(\Gamma)$ in the definition non-reduced as some of them could be collapse to a dihedral group) in order to detect them. This is the reason why we introduce the set $C_2$ in the definition of $\Gamma^d$. At this point one may wonder why, when we reduce the graphs, we collapse some vertex groups of order 2 into a dihedral group in the first place. The reason is that if a singular vertex group is elementarily equivalent to an infinite dihedral group, its involutions will satisfy the same properties as the ones from a singular dihedral group but the group may not split as a free product. For instance, the free product of the ultraproduct of $D_\infty$ (which is freely indecomposable) and a cyclic group of order 2 is elementarily equivalent by Sela to the free product of three groups of order 2, so the underlying structure that one can expect to recover for the elementary class is a group-labelled graph with two vertices, not three, so one does need to reduce the graph in order to be able to interpret it. We now turn to the formal definition of the extended core.
	
	\begin{definition}[The Extended Core]\label{defn: graph I(Gamma)}
		
		Let $\mc{G}=\mc{G}(\Gamma, \gvv)$ be a graph product and let $\Gamma^c$ be as in Definition \ref{defn:Gamma^c}. We define the following equivalence relation on vertices of $\Gamma^c$: $v\sim_1 w$ if and only if there exists $g\in \mathcal G$ such that $E_v^g = E_w$.
		
		Notice that by definition, the equivalence relation induces a relation in $\Gamma^e$, that is if $v\sim_1 w$ and $v\in \Gamma^e$, then $w\in \Gamma^e$. The quotient of $\Gamma^e$ by $\sim_1$ is precisely $\Gamma$.
		
		Let $\Lambda$ be the quotient of $\Gamma^c$ by $\sim_1$, that is $\pi:\Gamma^c \to \Lambda$ and $\pi(v)=\pi(v')$ if and only if $v\sim_1 v'$. As we noticed above, $\Gamma < \Lambda$. By definition, the vertices in $\Lambda$ that do not belong to $\Gamma$ have associated vertex groups which are either non-singular cyclic (as singular groups are identified with vertices of $\Gamma^e$) or singular cyclic groups of order 2. To each vertex $v\in \Gamma <\Lambda$, we assign the corresponding $G_v$ and denote it by $\bar E_v$. To each vertex $v\in \Lambda \setminus \Gamma$, we assign the group $\bar E_v = E_{v'}$, where $v'\in \Gamma^c$ such that $\pi(v')=v$ and $\ff(E_{v'}) < \mathcal G(\Delta)$ for some $\Delta< \Gamma$ minimal (notice that $\Delta$ is well defined).
		
		We further define an equivalence relation on the graph $\Lambda$ as follows: $v\sim_2 w$ if and only if $v,w \notin \minCore(\Gamma)$, $|\bar E_v|, |\bar E_w|>2$ and $\link{v}=\link{w}$. Notice that, by definition, the relation $\sim_2$ only identifies vertices whose associated group is elementarily equivalent to an infinite cyclic group.
		
		Let $[v]$ be an equivalence class for $v\in \Lambda \setminus \minCore(\Gamma)$ such that $\bar E_v$ is elementarily equivalent to an infinite cyclic group. If there exists $v'\in [v]$ such that $v'$ is weak and there is a witness $W$ for $v'$ such that for all $w\in W$, $\link{v'} \subsetneq \link{w}$, then we set $|[v]|=1$. Otherwise, we set $|[v]|=2$.
		
		Let $\Lambda_0 < \Lambda$ be the induced subgraph of $\Lambda$ defined by all the vertex groups such that $\bar E_v$ is not elementarily equivalent to $D_\infty$.
		
		We define a quotient $I(\Gamma,\gvv)$ of $\Lambda_0$ as follows: $V(I(\Gamma,\gvv))$ is the union of the following sets of vertices:
		
		\begin{itemize}
			\item $V(\minCore(\Gamma,\gvv))$
			\item $\bigcup_{v\notin \minCore(\Gamma,\gvv), |\bar E_v|=2} \{v\}$
			\item $\bigcup_{v\notin \minCore(\Gamma,\gvv), \bar E_v \equiv \mathbb Z, |[v]|=1} \{v\}$
			\item $\bigcup_{v\notin \minCore(\Gamma,\gvv),\bar E_v \equiv \mathbb Z |[v]|=2} (\{v_1\} \cup \{v_2\})$.
		\end{itemize}
		
		The edges are naturally induced by the quotient map.
		
		We define the \emph{extended core} of $\Gamma$, denoted by $\ECore(\Gamma)$, as the positive reduced graph obtained from $I(\Gamma, \gvv)$.
		
		We defined the \emph{labelled extended core} of $\Gamma$, denoted by $\ECore(\Gamma, \gvv)$ as follows: the underlying graph is $\ECore(\Gamma)$; for all $v\in \minCore(\Gamma, \gvv) < \ECore(\Gamma)$, see Remark \ref{rem:mincore subgroup of ecore}, we associate the same group as in $\minCore(\Gamma, \gvv)$; for all $v\in \ECore(\Gamma)$ such that $|\bar E_v|=2$, we associate $\bar E_v$; otherwise, we associate the group $\mathbb Z$.
	\end{definition}
	
	As we did we the other cores, we omit the explicit mention to the corresponding vertex groups as they are determined by those of $\Gamma$.
	
	\begin{remark}[$\minCore(\Gamma, \gvv) < \ECore(\Gamma, \gvv)$]\label{rem:mincore subgroup of ecore}
		
		Notice that as extended core is an induced subgraph of a quotient of $\Gamma^c$, if a vertex $v$ has an associated group elementarily equivalent to $D_\infty$ (but not isomorphic to $D_\infty$), then $v\in \Gamma^e$. From the definition of $\Gamma^d$, see the set $C_2$, we have that for each involution $g$ that belongs to a singular group $G_v$ elementarily equivalent to a dihedral group, there is a vertex $w_g\in \Gamma^c$ such that $E_{w_g}=\subg{g}$. Then, up to conjugacy, there are precisely two vertices $w_e$ and $w_{e'}$ in $\Lambda$ associated to involutions of $G_v$. Since $E_{w_e}^\perp = E_{w_{e'}}^\perp=E_v^\perp$, the vertices $E_{w_e}$ and $E_{w_{e'}}$ can be collapse to a vertex $w_{e,e'}$ in the positive reduction of $I(\Gamma)$ to the extended core $\ECore(\Gamma)$. The vertex $w_{e,e'}$ has as an associated group $D_\infty$.
		
		Then we can define a map $f:\minCore(\Gamma) \to \ECore(\Gamma)$ as follows: for all $v\in \Lambda_0$, we set $f(v)$ to be the image of $v$ in the quotient $\ECore(\Gamma)$; for all $v\in \Lambda\setminus \Lambda_0$, as $G_v \equiv D_\infty$, we set $f(v)$ to be the vertex $w_{e,e'}$ defined above. It follows from the construction that $f$ induces a graph isomorphism and $\bar E_{f(v)} = \bar E_v$ if $E_v\not\equiv D_\infty$ and $\bar E_{f(v)}= D_\infty$ if $\bar E_v \equiv D_\infty$.
	\end{remark}
	
	\begin{example}
		We describe the graph $(\Gamma, \{\mathbb Z\}_{v\in \Gamma})$, the core $\Core(\Gamma,\{\mathbb Z\}_{v\in \Gamma})$ and the extended core $\ECore(\Gamma,\{\mathbb Z\}_{v\in \Gamma})$.
		
		\tikzset{node distance=2.5cm, 
			every state/.style={minimum size=1pt, scale=0.3, 
				semithick,
				fill=black},
			initial text={}, 
			double distance=2pt, 
			every edge/.style={ 
				draw,  
				auto}}
		\begin{figure}[!ht]
			\begin{center}
				\begin{tikzpicture}
					\footnotesize
					\node[state] (1) at (0,0.5) {};  \filldraw (0,0.66) node[align=right,right] {$a_1$};
					\node[state] (2) at (0,0) {};  \filldraw (0,0) node[align=right,below] {$a_2$};
					\node[state] (3) at (1,0) {};  \filldraw (1,0) node[align=left,above] {$b$};
					\node[state] (4) at (2,0) {};  \filldraw (2,0) node[align=left,above] {$c$};
					\node[state] (5) at (3,0) {};  \filldraw (3,0) node[align=left,above] {$d$};
					\node[state] (6) at (4,0) {};  \filldraw (4,0) node[align=left,above] {$e$};
					\node[state] (7) at (5,0) {};  \filldraw (5,0) node[align=left,above] {$f$};
					\node[state] (8) at (2.5,0.5) {};  \filldraw (2.5,0.5) node[align=right, right] {$x$};
					\draw
					(1) edge (3)
					(3) edge (4)
					(4) edge (5)
					(5) edge (6)
					(6) edge (7)
					(2) edge (3)
					(4) edge (8)
					(5) edge (8);
				\end{tikzpicture}
				\caption{\small Graph $\Gamma$}
				
				\vspace{0.5cm}
				\begin{tikzpicture}
					\footnotesize
					\node[state] (3) at (1,0) {};  \filldraw (1,0) node[align=left,above] {$b$};
					\node[state] (4) at (2,0) {};  \filldraw (2,0) node[align=left,above] {$c$};
					\node[state] (5) at (3,0) {};  \filldraw (3,0) node[align=left,above] {$d$};
					\node[state] (6) at (4,0) {};  \filldraw (4,0) node[align=left,above] {$e$};
					\node[state] (8) at (2.5,0.5) {};  \filldraw (2.5,0.5) node[align=right, right] {$x$};
					\draw
					(3) edge (4)
					(4) edge (5)
					(5) edge (6)
					(4) edge (8)
					(5) edge (8);
				\end{tikzpicture}
				\caption{\small Graph $\minCore(\Gamma, \gvv)$}

				\vspace{0.5cm}
				\begin{tikzpicture}
					\footnotesize
					\node[state] (2) at (0,0) {};  \filldraw (0,0) node[align=left,above] {$a_2$};
					\node[state] (3) at (1,0) {};  \filldraw (1,0) node[align=left,above] {$b$};
					\node[state] (4) at (2,0) {};  \filldraw (2,0) node[align=left,above] {$c$};
					\node[state] (5) at (3,0) {};  \filldraw (3,0) node[align=left,above] {$d$};
					\node[state] (6) at (4,0) {};  \filldraw (4,0) node[align=left,above] {$e$};
					\node[state] (7) at (5,0) {};  \filldraw (5,0) node[align=left,above] {$f$};
					\node[state] (8) at (2.5,0.5) {};  \filldraw (2.5,0.5) node[align=right, right] {$x$};
					\draw
					(3) edge (4)
					(4) edge (5)
					(5) edge (6)
					(6) edge (7)
					(2) edge (3)
					(4) edge (8)
					(5) edge (8);
				\end{tikzpicture}
				\caption{\small Graph $\Core(\Gamma, \gvv)$}
				
				\vspace{0.5cm}
				\begin{tikzpicture}
					\footnotesize
					\node[state] (9) at (2,-0.66) {};  \filldraw (2,-0.66) node[align=right,right] {$bdx$};
					\node[state] (10) at (3,-0.66) {};  \filldraw (3,-0.66) node[align=right,right] {$cxe$};
					\node[state] (12) at (6,0) {};  \filldraw (6,0) node[align=left,above] {$a_2bcxdef$};
					
					\node[state] (2) at (0,0) {};  \filldraw (0,0) node[align=left,above] {$a_2$};
					\node[state] (3) at (1,0) {};  \filldraw (1,0) node[align=left,above] {$b$};
					\node[state] (4) at (2,0) {};  \filldraw (2,0) node[align=left,above] {$c$};
					\node[state] (5) at (3,0) {};  \filldraw (3,0) node[align=left,above] {$d$};
					\node[state] (6) at (4,0) {};  \filldraw (4,0) node[align=left,above] {$e$};
					\node[state] (7) at (5,0) {};  \filldraw (5,0) node[align=left,above] {$f$};
					\node[state] (8) at (2.5,0.5) {};  \filldraw (2.5,0.5) node[align=right, right] {$x$};
					\draw
					(3) edge (4)
					(4) edge (5)
					(5) edge (6)
					(6) edge (7)
					(2) edge (3)
					(4) edge (8)
					(5) edge (8)
					(9) edge (4)
					(10) edge (5);
				\end{tikzpicture}
				\caption{\small Graph $I(\Gamma)$}
			\end{center}
		\end{figure}
	\end{example}
	
	\begin{lemma}\label{lem:ECore and I(Gmma)}
		Let $\mc{G}=\mc{G}(\Gamma, \gvv)$. Then we have that $\ECore(\Gamma)$ is finite,
		$$
		\Core(\Gamma,\gvv) < \ECore(\Gamma,\gvv),
		$$
		$$\minCore(\ECore(\Gamma),\gvv)=\minCore(\Gamma,\gvv)$$
		$$\Core(\ECore(\Gamma),\gvv)=\Core(\Gamma,\gvv).$$
	\end{lemma}
	\begin{proof}
		
		First observe that by the definition of the set of involutions $C_2$ in Definition \ref{defn:D-completion}, involutions in $C_2$ are precisely conjugates of singular vertex groups of order 2 or belong to singular vertex groups elementarily equivalent to a infinite dihedral group. Therefore, the quotient of $C_2$ by the action by conjugation is finite.
		
		We show that the number of different links of vertices in $\Lambda$ is finite. This then would imply that $I(\Gamma,\gvv)$ is finite.
		
		From the construction, we have that $\Gamma <\Lambda$. As $\Gamma$ is finite, the number of different proper subgraphs is also finite. For each subgraph $\Delta< \Gamma$, let $V_{\Delta} =\{ v\in \Lambda \mid \bar E_v \in \mathcal G(\Delta)\}$. Notice that since edges correspond to commutation, we have that $\Delta \subset \link{v}$ if and only if $V_{\Delta} \subset \link{v}$. Therefore, for all $v\in \Lambda$, let $\Delta$ be the maximal subgraph of $\Gamma$ such that $\Delta \subset \link{v}$. Then $\link{v} = V_{\Delta}$. Therefore, the number of different links is bounded by the number of proper subgraphs of $\Gamma$ and so in particular it is finite.
		
		From Remark \ref{rem:mincore subgroup of ecore}, we have that $\minCore(\Gamma,\gvv) < \ECore(\Gamma,\gvv)$.
		
		Let $v, v'\in \Core(\Gamma,\gvv) \setminus \minCore(\Gamma,\gvv)$ such that $\link{v}=\link{v'}$. We claim that then for all witness $W$ of $v$ in $\Gamma$, there exists $w\in W$ such that $\link{v}=\link{w}$ (and so in particular, $w$ is weak). Indeed, otherwise, all vertices $v''$ such that $\link{v''}=\link{v}$ (and so in particular $v'$) would be redundant as $v''\notin W$ and $\cap_{w\in W} \st(w)=\link{v}=\link{v''}$. Furthermore, there cannot be three weak vertices with the same link, since a pair of them would witness the fact that the third is redundant.
		
		Therefore, we have that for the class $[v]$ as above, from the definition of $\ECore(\Gamma,\gvv)$ there are two vertices $v_1$ and $v_2$ corresponding to $[v]$ and so $\Core(\Gamma,\gvv) < \ECore\Gamma,\gvv)$.
		
		Finally, we are left to show that $\minCore(\ECore(\Gamma,\gvv))=\minCore(\Gamma,\gvv)$ and $\Core(\ECore(\Gamma,\gvv))=\Core(\Gamma,\gvv)$.
		
		We claim that any vertex $v\in \Gamma$ is weak and redundant in $\Gamma$ if and only if it is weak and redundant respectively in $\ECore(\Gamma,\gvv)$. Indeed, given a subgraph $\Delta$ of $\Gamma$ let $ext(\Delta)$ be the set of vertices of $\ECore(\Gamma,\gvv)$ that are either in $\Delta$ or correspond to a cyclic group $\G(\Delta)$. Then $\lk_{\ECore(\Gamma,\gvv)}(v)=ext(\lk_{\Gamma}(v))$ for any such $v$ and the same is true if we replace $\lk$ by star. It follows that if $W\subset V(\Gamma)$ witnesses weakness or redundancy of $v$ in $\Gamma$ it will also witness weakness or redundancy of $v$ in $\ECore(\Gamma,\gvv)$, as then
		\begin{align*}
			\bigcap_{w\in W}\st_{\ECore(\Gamma,\gvv)}(w)=\bigcap_{w\in W}\ext(\st_{\Gamma}(w))=\\=\ext(\bigcap_{w\in W}\st_{\Gamma}(w))
			=\ext(\lk_{\Gamma}(v))=\lk_{\ECore(\Gamma,\gvv)}(v).
		\end{align*}
		On the other hand, assume that a vertex $v\in \Gamma$ is weak in $\ECore(\Gamma)$, i.e.
		$$
		\bigcap_{w\in W}\st_{\ECore(\Gamma,\gvv)}(w)=\lk_{\ECore(\Gamma,\gvv)}(v)
		$$
		for some set of vertices $W$ of $\ECore(\Gamma)$ and some $v\in \Gamma$. Let $W'\subseteq V(\Gamma)$ be the union of the vertices already in $W$ or belonging to some subgraph $\supp(w)$ corresponding to a vertex of $w\in W$ not in $\Gamma$.
		
		\newcommand{\ec}[0]{\ECore(\Gamma,\gvv)}
		Notice that for $w\in W\setminus V(\Gamma)$ we have $\bigcap_{u\in \supp(w)}\st_{\Gamma}(u)=\lk_{\Gamma}(\supp(w))$, since the graph spanned by $\supp(w)$ is idecomposable and hence
		$$
		\bigcap_{u\in \supp(w)}\st_{\Gamma}(u)=\lk_{\ec}(w)\cap\Gamma=\st_{\ec}(w)\cap\Gamma.
		$$
		It follows that
		\begin{gather}\notag
			\begin{split}
				\lk_{\Gamma}(v)=&\lk_{\ec}(v)\cap V(\Gamma)=\bigcap_{w\in W}(\st_{\ec}(w)\cap V(\Gamma))=\\
				=&(\bigcap_{w\in W\cap V(\Gamma)}\st_{\Gamma}(w))\cap (\bigcap_{u\in \supp(w)}\st_{\Gamma}(u))=\bigcap_{w\in W'}\st_{\ec}(w)
			\end{split}
		\end{gather}
		
		and so $v$ is weak in $\ECore(\Gamma, \gvv)$. This settles the claim.
		
		It follows that
		$$
		\minCore(\Gamma,\gvv)=\minCore(\ECore(\Gamma,\gvv))\cap \Gamma
		$$
		and
		$$
		\Core(\Gamma,\gvv)=\Core(\ECore(\Gamma,\gvv))\cap \Gamma.
		$$
		
		Furthermore, every vertex of $\ECore(\Gamma,\gvv)$ either corresponds to a vertex of $\Gamma$ or to a non-singular cyclic subgroup. As we have seen above, the link of the latter vertices is the intersection of the links of vertices in some subgraph $\Delta$ of $\Gamma$, hence, by definition any such vertex is weak and redundant, hence does not belong to the minimal core $\minCore(\ECore(\Gamma,\gvv))$ or $\Core(\Gamma,\gvv)$ and the statement follows.
	\end{proof}
	
	\section{Defining relative maximal sets of the form $E_{v}\times E_{v}^{\perp}$} \label{sec:8}
	
	Let $\mc{G}=\mc{G}(\Gamma, \gvv)$ be a graph product and suppose that $\Gamma$ satisfies property $AP_{n}$ and is almost positive reduced (resp.\ positive reduced) and for all $v\in \Gamma$, $G_v$ satisfies a simple non-generic almost positive sentence (resp.\ nontrivial positive sentence).
	
	In this section we describe a family of definable sets where each member is of the form $E_v \times E_v^\perp$ for (almost all) $v\in \Gamma^d$.
	
	In the next section, we use these definable sets to show that $\Gamma^d$ and the corresponding subgroups $E_v$ are uniformly interpretable in $\mc{G}$, and that so is the graph $\Core(\Gamma)$.
	
	Next we provide an overview of the arguments developed in this section.

	\paragraph{Strategy.}
	Let us discuss the strategy in the simplest case, when the vertex groups of the graph product satisfy a nontrivial positive sentence and there is no 2-torsion in the vertex groups. We will later indicate the technicalities involved when considering more general sentences and in the presence of 2-torsion elements.
	
	Firstly, as we have proven in Theorem \ref{lem:extensionpreservation}, the class of groups with nontrivial positive theory is closed under taking extensions. This allows us to reduce to the case when all the vertex groups satisfy a common nontrivial positive sentence $\phi$ which is also satisfied by all metabelian groups.
	
	By our previous results, one can assume without loss of generality that this nontrivial positive sentence $\phi$ satisfied by all the vertex groups is of the form
	$$
	\phi \equiv \forall x\exists y\,\,\theta(x,y,z)=1,
	$$
	where $\theta$ is a system of equations.
	
	In Section \ref{sec: from non-compatible to small cancellation} we proved that, roughly speaking, (inside the class of graphs with property $AP_n$) one can uniformly describe a set of words $\mathcal W$ such that, if a pair of elements $a,b$ is so that $\subg{a,b}$ is irreducible in some  $H \in \mc{F}(\subg{a,b})$, then some of the words from $\mc{W}$ evaluated on $a,b$ are $N$-small cancellation in $H$.
	
	Using this set of words, we consider an existential positive formula $\sigma(z,v)$ obtained from $\phi$ by replacing the universal quantifier by the words in $\mathcal W$:
	$$
	\sigma(z,v)\equiv
	\bigwedge_{w\in \mathcal W}\exists y \,\,\theta (w(z,v),y,z).
	$$
	The key observation is that given an element $a$ in a graph product $\mc G$, the set $Y_a = \{h\in \mc{G}\mid \mc{G} \models \sigma(a, h)\}$ consists of elements $h$ such that $\langle a, h\rangle$ is small in each  $H\in \mc{F}(\langle a, h\rangle)$. Indeed, if  $\langle a, h\rangle$  is irreducible in some  $H\in \mc{F}(\langle a, h\rangle)$, then from Corollary \ref{l: small cancellation multi}, we deduce that some of the tuples in $\mathcal W$ are small-cancellation when evaluated on $a,h$ and then, using Section \ref{sec:small cancellation and formal solutions}, we deduce the existence of formal solutions for the positive sentence $\phi$, contradicting the assumption that the positive sentence is nontrivial (i.e. it does not admit formal solutions).
	
	Furthermore, if we consider maximal (by inclusion) proper sets among  $\{Y_a \mid a\in \mc{G}\}$, one can show that such sets can be defined in $\mc{G}$ via first order formulas (with some parameters). For such a $Y_a$ there exists $a'$ such that $Y_a=Y_{a'}$ and $|\ff(a')|=1$. If $|\ff(a)|=1$ and $Y_a$ is maximal, then $Y_a$ is of the form $E_v \times E_v^\perp$, where $v\in \Gamma^c$ and $a\in E_v$. In particular, in this case $Y_a$ is a subgroup. Moreover, if the graph $\Gamma$ is positive reduced and $v \in \Gamma^c$ is maximal, then the set $E_v\times E_v^\perp$ is precisely a maximal $Y_a$ for any $a\in E_v$.
	
	Denote the set of maximal sets among $\{Y_a \mid a\in \mc{G}\}$ by $M_1$. Once we have shown that the  sets in  $\mathcal M_1$ are uniformly definable in $\mc{G}$ (with parameters), we can consider the  sets from $\{Y_a \mid a\in \mc{G}\}$ which are maximal relative to $\mathcal M_1$, that is the set $\mathcal M_2$ of  sets from $Y_a$  that are properly contained in a  set from $\mathcal M_1$, and are maximal with this property. Again, we see that the sets from $\mathcal M_2$ are uniformly definable in $\mc{G}$ (with parameters). Proceeding recursively and repeating the above strategy, we obtain sets $\mathcal M_i$, $i\geq 1$. We see that if $\Gamma^e$ satisfies the property $AP_n$, then the union of the sets $\mathcal M=\cup_{i=1}^{n} \mathcal M_i$ consists on all the sets of the form $E_v\times E_v^\perp$ for all $v\in \Gamma^e$, and some sets of the form $K\times K^\perp$ where $K$ is a maximal cyclic subgroup.
	
	\medskip
	
	Let us now consider the case when the vertex groups satisfy a simple non-generic almost positive sentence, but there are no elements of order $2$ in the graph product.
	
	The class of groups which satisfy a non-generic almost positive sentence is also closed under extensions. However, this is not true if we insist on the sentence to be simple, that is, with only one inequality. This is the reason why we cannot assume that all vertex groups satisfy a common simple non-generic almost positive sentence (as we did in the previous case). We thus consider definable sets for a finite set of simple non-generic almost positive sentences $\phi_i$, $i=1, \dots, n$ such that each vertex group satisfies at least one of them.
	
	By Theorem \ref{t: quantifier reduction}, we can assume that each vertex group satisfies a simple non-generic almost positive sentence of the form
	$$
	\exists z \Theta(z)\ne 1 \, \forall x\exists y\,\,\theta(x,y,z)=1.
	$$
	where $\theta$ is a system of equations. As in the positive case,  we consider the existential positive formula
	$$
	\sigma(z,v)\equiv
	\bigwedge_{w\in \mathcal W}\exists y \,\,\theta (w(z,v),y,z).
	$$
	where $\mc{W}$ is the set of  of words given  in Corollary \ref{l: small cancellation multi}.
	Given $a\in \mc G$, the set $Y_a= \{h\in \mc{G}\mid \mc{G}\models \sigma(a, h)\}$   consists of elements $h$ such that for each  $H\in \ff(a,h)$  such that  $\pi^H(\Theta(a))\neq 1$, we have that $\langle a,h\rangle$ is small in $H$. Notice that in the nontrivial positive case, we were able to assure that the action  is small in all $H\in \mc{F}(\subg{a, h})$, while in this case, this can be assured only for those $H$ where $\Theta(a)$ projects onto a nontrivial element.

	Again, let $\mathcal M_1$ be the set of maximal  (by inclusion)  sets among $Y_a$.  As before, one can show that the sets from $\mathcal M_1$ can be uniformly defined in $\mc{G}$ using parameters $a\in \mc G$ such that
	\begin{itemize}
		\item  $|\ff(a)|=1$; and
		\item $\Theta(a)\neq 1$ in the unique component in $\ff(a)$.
	\end{itemize}
	Under these conditions,  $Y_a$ is again of the form $E_v \times E_v^\perp$, where $v\in \Gamma^c$ and some $a\in E_v$. In particular, such $Y_a$ is a subgroup.
	
	At this point one can already see the technicalities that arise when allowing more than one inequality, that is instead of considering groups which satisfy a simple non-generic almost positive theory, considering groups that satisfy a non-generic almost positive theory. Suppose for simplicity, that $\Theta(z)\neq 1$ consists of two inequalities, say $\Theta_1(z)\ne 1 \cap \Theta_2(z)\ne 1$. If we consider a witness $a$ of the system of inequalities, that is $\Theta_i(a)\ne 1$ for both $i=1,2$,  it could happen that $\ff(a)=\{H_1,H_2\}$ and the projection of $\Theta_1(a)$ in $H_1$ is nontrivial while the projection of $\Theta_2(a)$ is trivial, and vice-versa, the projection on $H_2$ of $\Theta_1(a)$ is trivial and the projection of $\Theta_2(a)$ on $H_2$ is nontrivial.  In this case, the set $Y_a=\{h\in \mc{G}\mid \mc{G} \models \sigma(a,h)\}$   consist of elements $h$ such that $\subg{a,h}$ is small in either $H_1$ or $H_2$. In particular, if $H_1$ and $H_2$ are singular, then $(H_1\times H_1^\perp) \cup (H_2 \times H_2^\perp) \subseteq Y_a$. However, $Y_a$ is not necessarily a subgroup: if there exist $g_i \in H_i^\perp$, $i=1,2$ such that $[g_1, g_2]\ne 1$, then $g_1g_2\notin Y_a$. In this case, the  sets $Y_a$ which are maximal may not include the sets of the form $E_v\times E_v^\perp$ for maximal vertices $v \in \Gamma^e$. 
	
	Similarly as in the case of nontrivial positive theory, if the graph $\Gamma$ is almost positive reduced and $v \in
	\Gamma^c$ is maximal (by  inclusion of stars), then the set $E_v\times E_v^\perp$ is a maximal definable set $Y_a$. However, in the case of nontrivial positive theory, any parameter $a\in E_v$ defines the set $E_v\times E_v^\perp$ while in the case of the non-generic almost positive theory, we know that if $E_v$ is a vertex that satisfies the sentence $\phi_i$ for some $i\in \{1, \dots, n\}$, then there exists some $a\in E_v$ which satisfies the inequality $\Theta(z)\neq 1$, and so $Y_a$ is of the form $E_v\times E_v^\perp$, but in general, for an arbitrary $a\in E_v$ we only  have that $Y_a < E_v\times E_v^{\perp}$.
	
	\bigskip
	
	Let us now examine the consequences of allowing $2$-torsion (i.e.\ the existence of elements of order $2$) in the vertex groups.
	
	The main difference is that, without 2-torsion, one parameter $a$ is enough to define the sets $Y_a$, while under the presence of 2-torsion, one needs to consider a tuple of parameters in order to define sets of the form $E_v\times E_v^\perp$. Indeed, if one tries to reproduce the argument described in the case without 2-torsion, even if one restricts to parameters $a$ with one factor, the set $Y_a$ would not necessarily be of the form $E_v\times E_v^\perp$ and in fact, $Y_a$ may not even be a subgroup. Let us consider a particular case to see this: assume that $a,h_1$ and $h_2$ are pairwise non-commuting  involutions (i.e.\ elements of order $2$), each one belonging to a vertex group, such that $\subg{a,h_1,h_2}$ is not singular, then $h_1, h_2\in Y_a$ because $\langle a, h_i\rangle$ is the infinite dihedral group for $i=1,2$, however $h_1h_2\notin Y_a$ since $\langle h_1h_2, a\rangle$ is irreducible in the unique factor $H\in \ff{a, h_1 h_2}$. In fact, one can show that if $a \in G_v$ is an involution, then $Y_a$ is the (possibly infinite) union of the subgroup $G_v\times G_v^\perp$ and the subgroups of the form $K \times K^\perp$ where $K$ is an infinite
	non-singular dihedral group containing $a$.

	In order to address the fact that $Y_a$ is not a subgroup of the form $E_v\times E_v^\perp$, we consider a second parameter $b\in Y_a$ and, roughly speaking, we consider the definable set $Z_{a,b}$ consisting of elements $h$ that belong to the  definable sets of the form $Y_{w(a,b)}$, where $w\in \mathcal W$ are  the words from Corollary \ref{l: small cancellation multi} (which are used to obtain small-cancellation tuples). As we see in Lemma  \ref{l: simple parameter case},  the words $w(a,b)$ assure that, if the subgroup $\langle a,b\rangle$ is not of exponent 2, one of the words $w(a,b)$ does not have order 2 and $Y_{a,b}$ is then a subgroup of the form $E_v\times E_v^\perp$ for some $v\in \Gamma^d$.
	
	More precisely, we show in Lemma \ref{l: first sieve lemma} that for any any  vertex $v\in \Gamma^d$  which is $\prec^d$-maximal and the associated subgroup $E_v$ is not abelian and satisfies $\phi_i$ for some $i=1, \dots, n$, the subgroup $E_v\times E_v^\perp$ is a definable set of type $Y_{a,b}$, where $a,b\in E_v$ and $[a,b]\ne 1$. The case when the vertex $v\in \Gamma^d$ is $\prec^d$-maximal and $E_v$ is abelian is addressed in Lemma \ref{l: second sieve lemma}, where we consider the set $\mathcal X$ of centralisers of elements $C(g)$ which are maximal among $\{Y_{a,b}\mid a,b\in \mc{G}\}\cup \{C(h)\mid h\in \mc{G}\}$. We show that in this case, for $v\in \Gamma^d$, such that $v$ is $\prec^d$-maximal and $E_v$ abelian, the subgroup $E_v\times E_v^\perp$ belongs to $\mathcal X$ and it is definable,  since it coincides with the centralizer $C(a)$ for any $a\in E_v$. As a result we obtain  a family of definable sets of the form $E_v\times E_v^\perp$ where $v\in \Gamma^d$ is $\prec^d$-maximal. This family contains all subgroups of the form $E_v\times E_v^\perp$ where $v$ is  a $\prec^d$-maximal vertex.
	
	We would like to proceed as in the case of nontrivial positive sentences and prove that there is a family of definable sets of the form $E_v\times E_v^\perp$ for all $v\in \Gamma^d$ which are only properly contained in a $\prec^d$-maximal vertex group, that is, vertices of second $\prec^d$-level.
	
	As we mentioned above, if $E_v$ satisfies a nontrivial positive sentence, then any choice of parameters in $E_v$ defines $E_v\times E_v^\perp$. However this is not the case for almost positive sentences. For instance, the formula may just define the centraliser of an element. If $E_v$ has nontrivial center and we take an element of the center as a parameter, then we define the whole $E_v\times E_v^\perp$. However, if we take  a non-central element as a parameter, then we define its centraliser which is properly contained in $E_v\times E_v^\perp$. This adds definable sets which are not of the form $E_v\times E_v^\perp$ and, a priori, it could produce infinite chains of properly contained definable sets inside $E_v\times E_v^\perp$.
	
	To avoid this, one would like to assure that the set of parameters in the second round are not taken inside the sets $E_v$ for $\prec^d$-maximal $v\in \Gamma^d$. However at this point, only the sets of the form $E_v\times E_v^\perp$ for $\prec$-maximal $v\in \Gamma^d$ are definable. In Section \ref{s: interpretation_of_base_structure_(first)}, we deal with this and show that given the family of sets of the form $E_v\times E_v^\perp$ for $\prec^d$-maximal $v\in \Gamma^d$, there is a family of definable sets of the form $E_v$ for $\prec^d$-maximal $v\in \Gamma^d$. We denote this family of definable sets $\mathcal A$. We are now ready to proceed and consider definable sets which are relatively maximal and whose defining parameters are not in any $E_v\in \mathcal A$. This will allow us to show that these sets are again of the form $E_v\times E_v^\perp$ for those $v\in \Gamma^d$ which are second to maximal in $\prec^d$. However, we have lost some vertices on the way: suppose that $E_v=G_u \ast G_{u'}$ is dihedral, where $v$ is $\prec^d$-maximal and $u,u'\in \Gamma^e \subset \Gamma^d$. Assume further that $\link{u} < \link{v}$ in $\Gamma^d$. Then we have that $G_u \times G_u^\perp \subsetneq E_v\times E_v^\perp$ (so $u$ is a vertex of the second level), but in order to define the set $G_u\times G_u^\perp$ we should take parameters in $G_u$, but as $G_u < E_v \in \mathcal A$, we do not allow taking parameters in $E_v$ (to avoid having definable sets which are not of the form we want). In any case, we proceed without these vertices and we show in Section \ref{s: interpretation_of_base_structure_(first)} how to recover them.
	
	So the main result, Theorem \ref{p: main definability proposition}, is stated in a way that can be used recursively to define the vertices on all $\prec^d$-levels. One should think of the family of definable sets $\mathcal A$ in the statement as the set of vertices $E_v$, $v\in \Gamma^d$ of $\prec^d$-level $r-1$. Given this set, we can define relative maximality and avoid some parameters to define the sets $E_v\times E_v^\perp$ of $\prec^d$-level $r$.
	
	\bigskip
	
	We now proceed to formalizing the strategy described above. In the following lemma we introduce the existential positive formula $\sigma_r(z,v)$ obtained from $\phi$ by replacing the universal quantifier by the words in $\mathcal W$ mentioned previously.  We show that if a graph product in our class witnesses the formula $\sigma_r(z,u)$ in $(c,d)$, then either $c$ does not satisfy the inequality in $\phi$, i.e.\ $\Theta(c)=1$, or the subgroup generated by $c$ and $d$ has a small action on each $H \in \mc{F}(\langle c, d\rangle)$, where $\Theta(\pi^H(c))\neq 1$.

	\begin{lemma}
		\label{l: basic generic almost positive}Let $z$ be a tuple of variables and assume we are given a simple non-generic almost positive sentence
		\begin{align*}
			\phi\equiv\exists z\, \Theta(z)\neq 1\wedge \psi(z)
		\end{align*}
		with $\psi(z)$ positive. Then for any $r>0$ there is a positive existential formula $\sigma_{r}(z,u)$ with $|u|=r$ such that
		any group $G$ satisfies $\forall z (\psi(z)\rightarrow \forall v \,\,\sigma_r(z,v))$ {\rm(}that is for any group $G$ and any tuple $c$, $\forall v \sigma_r(c,v)$ is true in $G$ if $ \psi(c)$ is true in $G${\rm)},   and
		for any graph product $\mathcal{G}=\mathcal{G}(\Gamma,\{G_{v}\}_{v\in \Gamma})$  where $\Gamma$ has property $AP_{n}$,
		the following holds.

		Let $c\in\G^{|z|}, d\in\G^{r}$, and suppose that $\G\models\sigma_r(c,d)$. Then  for all $H\in\ff(c,d)$, either $\pi^{H}(\Theta(c))=1$ or $\pi^{H}(\subg{c,d})$  is not irreducible in $H$.
		
		In the particular case where $\phi$ is a nontrivial positive formula we obtain that
		for any $d\in \mc{G}^r$, if $\G\models\sigma_{r}(d)$ {\rm(}in this case $z$ is an empty tuple of variables{\rm)}  then for all $H\in \ff(d)$, $\subg{d}$ is not irreducible in $H$.
	\end{lemma}
	\begin{proof}
		By Theorem \ref{t: quantifier reduction} we can assume that the positive formula $\psi(z)$ is of the form:
		\begin{align*}
			\forall x\exists y\,\,\theta(x,y,z)=1
		\end{align*}
		where $\theta(x,y,z)=1$ is a system of equations.
		
		Let $N>0$ be the integer provided by applying Corollary \ref{c: formal solutions forall exists} to the formula $\psi(z)$ and let $\mathcal W$ be the set of tuples of words provided by Corollary \ref{l: small cancellation multi} where the number of variables $k$ is $|z|+r$ and the size $m$ of each tuple is $|y|$.
		
		Let us see that the formula
		\begin{align*}
			\sigma_{r}(z,v)\equiv
			\bigwedge_{w\in \mathcal W}\exists y \,\,\theta (w(z,v),y,z)=1
		\end{align*}
		works as intended. It is clear that if $G$ is a group, $c\in G^{|z|}$, and $G\models \forall v\ \sigma_r(c, v)$, then $\psi(c)$ is true in $G$ (this is because the set of terms  $\{\theta(w(c,d), y, c)\mid d\in G^{|v|}\}$ is a subset of the set of terms $\{\theta(d, y, c)\mid d\in G^{|x|}\}$).
		
		Fix now $c\in\G^{|z|},           d\in\G^{r}$ and $H\in \ff(c,d)$, and suppose $\mc{G}\models \sigma_r(c,d)$. We assume that $\subg{c,d}$ is irreducible in $H$, and we will show that in this case $\Theta(\pi^H(c))=1$.

		From Corollary \ref{l: small cancellation multi} we have that, for some $w\in\mc{W}$, $\pi^H(w(c,d))$ is an $N$-small cancellation $m$-tuple with respect to the action of $H$ on any of the associated vertex trees of $H$.
		
		The first part of the conclusion of Corollary \ref{c: formal solutions forall exists} provides a set of Diophantine conditions $\mc{D}$ and a formal solution $\alpha_D$ for each $D\in \mc{D}$. Since $\phi$ is non-generic, we have that $\Theta(z)$ belongs to the normal closure of $\Pi(v,z)$ in $\mbb{F}(v,z)$ for all Diophantine condition $\exists v \ \Pi(v,z)=1$ in $\mc{D}$.

		Then there exists  $b\in\G^{|y|}$  such that $\theta(w(c,d),b,c)=1$ for all $w\in\mathcal{W}$. The second part of the conclusion of Corollary \ref{c: formal solutions forall exists} now ensures that there exists a Diophantine condition $\exists v \ \Pi(v, z)=1$ in $\mc{D}$ and $d\in H^{|v|}$ such that $\Pi(d, \pi^H(c))=1$. Since $\Theta(z)$ belongs to the normal closure of $\Pi(v,z)$ in $\mbb{F}(v,z)$, we obtain $\Theta(\pi^H(c))=1$, as required.
		
		In the case of a nontrivial positive formula $\phi$ one may take $z=\emptyset$  and the result follows from a simplified version of the same argument. In this case the existence of any formal solution for $\phi$ suffices to prove the result.
	\end{proof}
	
	In the next theorem we prove that the converse of the previous Lemma \ref{l: basic generic almost positive} holds in case $\phi$ is a nontrivial positive sentence. Namely, we prove that if $\subg{c,d}$ is small  in all  $H\in \ff(c, d)$, then $\sigma_r(c,d)$ is true in the graph product. Therefore, in the nontrivial positive case, the set defined by $\sigma_r(x,y)$ is precisely the set of tuples $(c,d)$ such that $\subg{c,d}$ is small in each  $H\in \ff(c, d)$ (recall that by convention, we view positive sentences as almost positives ones by add the inequality $\exists z \, z\ne 1$, see Remark \ref{r: positive_is_generic}).
	
	\begin{thm}\label{thm:definitionComp}
		Fix some positive sentence $\phi$ not admitting a formal solution, that is $\F\models\neg\phi$, where $\F$ is any nonabelian free group.
		
		Let $\mathcal{G}=\GP$ be a graph product where $\Gamma$ has property $AP_{n}$, and for all $v\in \Gamma$ the group $G_{v}$ satisfies $\phi$.
		
		Then there is an existential positive formula on two variables $\phi_{n}(x,y)$ such that a pair of elements $(g,h)\in\mathcal{G}^{2}$ satisfies $\phi_{n}$ if and only if $\subg{g,h}$ is small  in $H$, for all $H\in \ff(c, d)$.
		
	\end{thm}
	
	\begin{proof}
		By Lemma \ref{t: quantifier reduction}, we can assume that $\phi$ is of the form:
		\begin{align*}
			\forall x\exists y \,\,\Sigma(x,y)=1.
		\end{align*}
		Furthermore, by Remark \ref{lem:extensionpreservation}, we may assume $\phi$ is satisfied by any metabelian group as well.
		Consider the formula $\sigma_{2}(x,y)$ provided by Lemma \ref{l: basic generic almost positive} (for the value of $n$ of condition $AP_n$). We claim that
		$\G\models\sigma_{2}(g,h)$ for $h,g\in\G$ if and only if $\subg{g,h}$ is small in $K$ for all $K \in \ff( c, d)$.
		
		If $\subg{g,h}$ is irreducible  in $K$ for some $K\in \ff(g,h)$, then $\G\models\neg\sigma_{2}(g,h)$ by Lemma \ref{l: basic generic almost positive}.
		
		On the other hand we know that $\phi\rightarrow\forall x\forall y\sigma_{2}(x,y)$ is valid in any group due to Lemma \ref{l: basic generic almost positive}.
		If $\subg{g,h}$ is small in $H$ for all  $H \in \ff( g, h)$, then for any $H\in\ff(g,h)$ we have the following alternative: either $H$ is singular, or $H$ is non-singular and $\pi^H(\subg{g,h})$ is cyclic or dihedral.
		
		It follows that $\subg{g,h}$ is a subgroup of a group $M$ which is the direct sum of finitely many  conjugates of vertex groups $G_{v}$ and a  group $K$ which is the direct sum of finitely many dihedral and cyclic groups. By assumption, $G_v$ and metabelian groups satisfy $\phi$, and so does their direct product. Therefore, $M\models\phi$ and hence
		$M\models\forall x\forall y\,\sigma_{2}(x,y)$ and thus $M\models\sigma_{2}(g,h)$. Since $\sigma_{2}(x,y)$ is existential we conclude that $\G\models\sigma_{2}(g,h)$.
		
		We thus have that $\subg{g,h}$ cannot be irreducible in $H$ for any $H\in \ff(g,h)$, hence by definition it is small.
	\end{proof}
	
	We are now ready to state the main result of this section. As we discussed at the beginning of the section, this theorem is intended to be used recursively to define families of sets that contain all subgroups of the form $E_v\times E_v^\perp$ for $v \in \Gamma^d$ of certain height in the preorder $\prec^d$. One can think that in the first round, the set $\mathcal A$ is empty and the set $\mathcal Z^j$ contains $E_v\times E_v^\perp$ for all maximal vertices $v\in \Gamma^d$ that satisfy a given non-generic almost positive sentence $\phi^j$. In Section \ref{s: interpretation_of_base_structure_(first)}, given this family of definable sets, we show that the corresponding groups $E_v$ are also definable (provided the vertex $v$ is non-weak), call this set $\mathcal A_1$. At step $k$, we will have a set $\mathcal A$ of definable sets (which are direct sums of) $E_v$ for $v\in \Gamma^d$ of height at most $k-1$ in the preorder $\prec^d$. In this case, the next theorem applied on $\mathcal A$ provides again a family of definable sets that contains $E_v\times E_v^\perp$ for all vertex $v\in \Gamma^d$ with $v$ of height $k$ in the preorder $\prec^d$, that is vertices that are maximal relative to the family $\mathcal A$.
	
	\newcommand{\Y}[0]{\mathcal{Y}}
	\newcommand{\XX}[0]{\mathcal{X}}
	\begin{thm}
		\label{p: main definability proposition}
		Let $n>0$, let $\Phi$ be a finite collection of simple non-generic almost positive formulas, and let $\textbf C=\G(\mathbf{K}_{n},\mathbf{C}_{\Phi})$.
		
		Suppose we are given a finite collection of formulas $\psi^{i}(x,y)$, $1\leq i\leq s$, with $|x|=1$, such that for  each $\G\in\textbf C$ the set \begin{equation}\label{e: def_mcA}
			\mathcal{A}=\left\{\{g\in \G\mid \G\models\psi^i(g,h)\}\mid h\in \G^{|y|}; i= 1, \dots, s\right\}
		\end{equation}
		is a collection of subgroups of $\G$ with the following properties:
		\begin{enumerate}
			\item  \label{A-property conjugation} $\mathcal{A}$ is closed under conjugation in $\G$: if $H\in \G$ and $g\in \G$, then $H^g\in \mc{A}$.
			\item \label{A-property factor closure} If $A\in \mc{A}$ and $ K \leq A$ is a subgroup, then $H\in\mathcal{A}$  for any $H\in\ff(K)$.
			\item \label{A-property vertex closure} If $A\in \mc{A}$, $v\in \Gamma^d$, and $E_{v}\cap A\neq \{1\}$, then one of the following holds:
			\begin{itemize}
				\item $|E_v|=2$, there exists $A'\in \mc{A}$ such that $ E_{v'}\subseteq A'$, $E_{v'}$ is dihedral, and   $E_v \leq E_{v'}$;
				\item $E_{v}\subseteq A'$ for some $A'\in\mathcal{A}$.
			\end{itemize}
		\end{enumerate}
		Then there exists an integer $r$ and  $|\Phi|$ definable in $\mc{G}$  sets $U^{j}\subseteq\G^{r}$, $1\leq j\leq |\Phi|$, so that for each $1\leq j \leq |\Phi|$ there is a collection of sets $\mc{Z}^j = \{Z_{(a,b)}^j \mid (a,b)\in U^j\}$ with the following properties.
		
		Let  $\mathcal{G}=\GP\in \G(\mathbf{K}_{n},\mathbf{C}_{\Phi})$. Then for any $1\leq j\leq |\Phi|$  and any $Z \in \mathcal{Z}^{j}$, there is $(a,b) \in U^j$ such that\begin{itemize}
			\item  $(a,b) \not\subseteq A$ for all $A\in \mathcal A$,
			\item $Z^{j}_{a,b}=E_{v}\times E_{v}^{\perp}$ for some $v\in \Gamma^d$, and with $(a,b)\subseteq E_{v}$.\footnote{In general, for a fixed $j$, the sets $Z_{a,b}^j$ are not  in one-to-one correspondence with the tuples  $(a,b) \in U^j$.} In this case, $(a,b)\subseteq E_{v}$ is of the {\rm(}possibly non-unique{\rm)} form:
			\begin{itemize}
				\item $b=1$ and $a=(g,1, 1, \dots, 1)$ with $g\in Z(E_{v})\setminus\{1\}$; or
				\item $\subg{a,b}$ is not abelian.
			\end{itemize}
		\end{itemize}
		Moreover, for any $1\leq j \leq |\Phi|$,  any $(a,b)\in U^j$, and any $v\in \Gamma^{d}$, we have that $Z_{a,b}^{j}=E_{v}\times E_{v}^{\perp}$ if and only if  $v$ satisfies the following two conditions:
		\begin{itemize}
			\item {\rm(}relative maximality with respect to $\mathcal A${\rm)} for all $v'\in \Gamma^d$ such that $v \preceq^d v'$, we have that $E_{v'}\subseteq A$ for some $A\in \mathcal A$;
			\item {\rm(}order $2$ and not dihedral dominated{\rm)} if $|E_v|=2$, then there is no $v'\in \Gamma^d$ such that $v\preceq^d v'$, $E_v\leq E_{v'}$, $E_{v'}$ is dihedral, and $E_{v'}\subseteq A$ for some $A\in \mathcal A$.
		\end{itemize}
	\end{thm}
	
	\begin{proof}
		By \cite[Theorem 6.5]{casals2019positive} we may assume that $\Phi=\{\phi^{i}\}_{i=0}^{|\Phi|}$, is a set of simple non-generic almost positive formulas of the form
		$$
		\phi^{i}\equiv\Theta(z^{i})\neq 1\wedge\psi^{i}(z^{i}),
		$$
		where $\Theta(z^{i})\neq 1$ consists of a single inequality;
		$\psi^{i}(z^{i})\equiv \forall x^i\exists y^i \,\,\Sigma^i(x^i,y^i, z^i)=1$ is a $\forall\exists$-positive formula,
		for $ 1\leq i\leq |\Phi|$, and
		$\phi^{0}$ is a nontrivial positive formula which we may and do assume to be satisfied by any metabelian group.
		
		By adding dummy variables if necessary, we may assume that $z^{i}=z$ for all $1\leq i\leq |\Phi|$.
		
		For each $0\leq i\leq |\Phi|$, we let
		$$
		\sigma^{i}_{2}(z,v)\equiv
		\bigwedge_{w\in \mathcal W}\exists y \,\,\Sigma^i (w(z,v),y,z) \quad \hbox{(where $|v|=2$)}
		$$
		be the positive existential formula on $|z|+2$ free variables provided by Lemma \ref{l: basic generic almost positive} given the sentence $\phi^i$   (for  same value of $n$ given in the statement of the present theorem).
		
		For each $1\leq i\leq |\Phi|$ we define the following formula on the tuple of variables $z$ and the single variables $w$ and $u$:
		$$
		\tilde{\sigma}^{i}(z,w,u)\,\,\equiv\,\,\exists u_{1}\exists u_{0}\,\,\left[ u=u_{1}u_{0}\wedge \sigma^{i}_{2}(z,w,u_0)\wedge[w,u_{1}]=1\wedge\bigwedge_{l=1}^{|z|}[z_{i},u_{1}]=1 \right],
		$$
		where $u_0, u_1$ are single variables, and we  view $(w,u_0)$ as the variable $v=(v_{1},v_{2})$.
		
		Let also $\xi(x)$ be the formula on one variable $x$ expressing the fact that there exists elements $y_1, \dots, y_m$ such that $x=y_{1}y_{2}\dots y_{m}$, where $m$ is the bound of the size of the maximal clique of $\Gamma$,  $[y_{j},y_{\ell}]=1$ for all $1\leq j\neq \ell \leq m$, and for all $1\leq j \leq m$ there exists $A_j\in \mc{A}$ such that $y_{j}\in A_{j}$.  Informally, the formula $\xi(x)$ is intended to express the statement ``trivial modulo $\mathcal A$''. Note that such a formula can be written in first-order logic because of the definability and properties of the set $\mc{A}$ defined in Equation \eqref{e: def_mcA}.
		
		Finally, let:
		$$
		\tau^{j}(z,w,u)\,\,\equiv\,\neg\xi([\Theta(z),w])\wedge \tilde{\sigma}^{j}(z,w,u)
		$$
		and
		$$
		\tau^{0}(z,w,u)\equiv \bigvee_{i=1}^{|z|}\neg\xi(z_{i})\wedge \sigma_2^{0}(z,w,u).
		$$
		For the record, we highlight the following simple observation, which is a consequence of the second  closure property of $\mathcal A$ (as listed in the statement of Theorem \ref{p: main definability proposition}):
		\begin{obs}
			\label{o: meaning of xi} Given $\G\in\textbf C$ and $g\in\G$ we have $\G\models\xi(g)$ if and only if for all $\hh\in \ff(g)$, $\pi^{\hh}(g) \in A$ for some $A\in\mathcal{A}$.
		\end{obs}
		
		Recall that from Lemma \ref{l: basic generic almost positive} we have that $\forall z\,(\psi^{j}(z)\rightarrow\forall u\,\sigma^{j}_{1}(z,u))$ holds in the theory of groups.
		
		\begin{definition}[Definable set $Y_{a,b}^j$]\label{defn: Y}
			For $a\in\G^{|z|}$ and $b\in \G$ and $0\leq j\leq |\Phi|$, we define
			$$
			Y^{j}_{a,b}=\tau^{j}(a,b,\G) = \{g\in \G\mid \G \models \tau^j(a,b,g)\} \subseteq \G.
			$$
			For $0\leq j\leq |\Phi|$ let $V_{j}$ be the collection of vertices $v\in \Gamma^e$ such that $H_v\models\phi^{j}$. Finally, define
			$$
			\mathcal{Y} =\left\{Y^{j}_{a,b}\mid 0\leq j\leq |\Phi|,\  a\in\G^{|z|},b\in\G \right\}.
			$$
		\end{definition}
		
		In the next lemma we describe some properties of the definable sets $Y_{a,b}^j$ and we show that, in some cases, they are subgroups.
		
		\begin{lemma}
			\label{l: simple parameter case}
			Let $(a,b)\in \G^{|z|+1}$ and $0\leq j \leq |\Phi|$   be such that  $\Theta^{j}(a)\neq 1$ if $j\neq 0$. Then the following hold:
			\begin{enumerate}
				\item \label{assertion 0} Suppose $Y^{j}_{a,b}\neq\emptyset$. Then if $j=0$, we have that $\subg{a,b}$ is not irreducible in any $H\in \ff(a,b)$.
				
				If $j\neq 0$, we have that $\subg{a,g}$ is not irreducible in any $H\in \ff(a,b)$ such that  $\pi^H(\Theta^{j}(a))\neq 1$.
				\item \label{assertion 1} If  $\subg{a,b} \leq E_v$ for some $v\in \Gamma^d$, then $Y^{j}_{a,b}\subseteq E_{v}\times E_{v}^{\perp}$.
				
				\item \label{assertion 2} 
				Assume that $(a,b)\in H_{v}^{|z|+1}$ for some $v\in V_{j}$ {\rm(}see {\rm Definition \ref{defn: Y})}. Then, if $H_{v}\models \forall u\,\tau^{j}(a,b,u)$, we have that $H_{v}\times H_{v}^{\perp}\subseteq Y^{j}_{a,b}$.
				
				In particular, if $j=0$ or if $j\geq 1$ and $H_{v}\models\Theta(a)\neq 1\wedge\psi(a)$, we have that $H_{v}\models \forall u\,\tau(a,b,u)$ and so the latter statement holds in this case.
			\end{enumerate}
		\end{lemma}
		
		\begin{subproof}
			Assertions \ref{assertion 1} and \ref{assertion 0} follow from the definition of $\tau^{j}(z,z',u)$ and the properties of $\sigma^{j}_{r}$ given by Lemma \ref{l: basic generic almost positive}.
			
			Let us now check Assertion \ref{assertion 2}. Since $\tilde{\sigma}^{j}(z,z',u)$ is an existential formula we have that $H_{v} \models\tilde{\sigma}^{j}(a,b,k)$ implies $\G\models\tilde{\sigma}^{j}(a,b,h)$ for any $h\in H_{v}$. Since $H_{v}$ is a retract of $\G$ and $\tilde{\sigma}^{j}$ is positive, the opposite implication is valid as well: $H_{v}\models\tilde{\sigma}^{j}(a,b,h)$ only if  $\G\models\tilde{\sigma}^{j}(a,b,h)$. To conclude, it suffices to notice that by definition $\tilde{\sigma}^{j}(a,b,h)$ implies $\tilde{\sigma}^{j}(a,b,hh')$ for any $h'\in C(a,b)$, so that $(Y^{j}_{a,b}\cap H_{v})\times H_{v}^{\perp}\subseteq Y^{j}_{a,b}$ if $j\geq 1$, while the property follows from Theorem \ref{thm:definitionComp} in the case $j=0$.
		\end{subproof}
		
		\begin{obs}
			\label{o: passing to components} Let $a\in\G^{m}$, $d\in\G$ and let $H\in\ff(a)$ be such that $\pi^{H}(\subg{a})$ is not of order $2$ and is not trivial, and $\subg{a,d}$ is not irreducible in the unique $K\in\ff(a,d)$ containing $H$. Then $H=K$.
		\end{obs}

		\begin{subproof}
			Since the action of $\pi^K(\langle a,d\rangle)$ in $K$ is not irreducible, this subgroup is either singular, cyclic or dihedral. In the first two alternatives, nontriviality of $\pi^H(\subg{a})$  alone implies $H=K$. In the third alternative, $H\neq K$ can only take place if $\pi^{K}(\langle a \rangle)$ has order two, which is ruled out by our assumption on $\pi^H(a)$.
		\end{subproof}
		
		The next lemma describes when the parameters define maximal definable sets. Without 2-torsion, it would state that one can assume the parameters to have one factor.
		
		\begin{lemma}
			\label{l: reducing to one block}Fix $0\leq j\leq |\Phi|$ and $(a,b)\in\G^{|z|+1}$, and let $H\in\ff(a,b)$ be such that $\pi^H(\subg{a,b})$ is non-abelian, and in case $j\geq 1$ we have $\pi^H(\Theta(a)^{j})\neq 1$.
			Then $Y^{j}_{a,b}\subseteq Y^{j}_{\pi^H(a),\pi^H(b)}$.
			
		\end{lemma}
		\begin{subproof}
			Let $d\in Y_{a,b}^j$. From the definition of $Y_{a,b}^j$, see Definition \ref{defn: Y}, we have that $d=d_{0}d_{1}$ where
			$\sigma_{2}^{j}(a,b,d_{0})$ holds in $\G$ and $[d_{1},a_{\ell}]=1$ for any $1\leq \ell \leq |z|$ and $[d_1,b]=1$. Recall that 
			$$
			\sigma^{j}_{2}(z,v)\equiv
			\bigwedge_{w\in \mathcal W}\exists y \,\,\Sigma^j (w(z,v),y,z)=1.
			$$
			
			Let $\ff'=\ff(a,b,d)$ and let $K$ be the unique $K\in\ff'$ that contains $H$.
			
			By assumption $\pi^H(\subg{a,b})$ is non-abelian, and in particular it is not of order 2. Furthermore, since the inequality $\pi^H(\Theta^{j}(a))\neq 1$ is satisfied, it follow from Lemma \ref{l: simple parameter case} that $\subg{a,b,d_0}$ is not irreducible in $K$. Hence, we can apply Observation \ref{o: passing to components} to $(a,b),d_{0}$ and we get that $\pi^H(\subg{a,b}) = \pi^K(\subg{a,b})$ and so $\pi^H(a)=\pi^{K}(a)$ and $\pi^H(b)=\pi^K(b)$.
			We set $a^0=\pi^K(a)$ and $b^0=\pi^K(b)$. Let $c^0=\pi^K(d)$ and $$c^{1}= \pi^{\left(\prod_{M\in \ff'\setminus K}M\right)}(d),$$ so that $d_{0}=c^{0}c^{1}$. Furthermore, we have that $[c^{1},a^0_{\ell}]=1$ for all $1\leq \ell \leq |a|$ as well as $[c^{1},b^{0}]=1$. Since $\sigma_2^{j}(a,b,d_0)$ holds, there exists $e\in\G^{|t|}$ such that $\Sigma^{j}(w(a,b,d),e,a)=1$. Therefore we have
			\begin{gather}\notag
				\begin{split}
					1=\pi^K(\Sigma^{j}(w(a,b,d),e,a))=&\Sigma^{j}(w(\pi^K(a),\pi^K(b),\pi^K(d)),\pi^K(e),\pi^K(a))=\\
					=&\Sigma^{j}(w(a^0,b^0,c^0),\pi^K(e),a^0)
				\end{split}
			\end{gather}
			and so $\sigma_2^{j}(a^0,b^0,c^0)$ holds in $\mathcal G$. Then taking $c_{0}$ as $u_{0}$ and $c_{1}d_{1}$ as $u_{1}$, they witness the fact that $d\in Y_{a^{0},b^{0}}$.
		\end{subproof}
		
		In the next lemma, we do the first sieve and show that we can define the sets $E_v\times E_v^\perp$ which are maximal relative to $\mathcal A$ and are not abelian. We deal with the abelian case in a subsequent lemma.

		\begin{lemma}
			\label{l: first sieve lemma}
			Let $\mathcal{Y}_{\max}(\mathcal A)$ be the family of maximal non-empty members of $\Y$ such that $(a,b) \in \G^{|z|+1}$ and $(a,b) \not\subset A$ for all $A\in \mathcal A$. Then any $Y\in \mathcal Y_{\max}(\mathcal A)$ is of the form $Y=Z^{j}_{a,b}$ and $Z_{a,b}^{j}=E_{v}\times E_{v}^{\perp}$ with $(a,b)\subseteq E_{v}$, $v\in \Gamma^d$ and $E_v$ is either singular or dihedral.
			
			Moreover, $Z_{a,b}^{j}=E_{v}\times E_{v}^{\perp} \in \mathcal Y_{\max}(\mathcal A)$ if and only if $v\in \Gamma^{d}$ satisfies that for all $v'\in \Gamma^d$ such that $v \prec^d v'$, then $E_{v'}\subset A$ for some $A\in \mathcal A$ and $E_v$ has trivial center if $j\ge 1$ and it is non-abelian if $j=0$ in which case the defining parameters $(a,b)$  can be chosen so that $\subg{a,b} \subset E_v$ is non-abelian.
		\end{lemma}

		\begin{subproof}
			We start by noticing that whenever $Y^{j}_{a,b}\neq\emptyset$, by Lemma \ref{l: simple parameter case}, there cannot be any $\hh\in\ff(a,b)$ such that $\Theta^{j}(a)^{\hh}\neq 1$ and the action of $\subg{a,b}$ on $\hh$ is irreducible. Also by part \ref{assertion 1} of Lemma \ref{l: simple parameter case}, if $|\ff(a,b)|=1$, then $\subg{a,b}\leq E_{v}$ for some $v\in \Gamma^d$ and $Y_{a,b}\subseteq E_{v}\times E_{v}^{\perp}$.
			
			Assume that $v\in V_{j}$ for $j\geq 1$ and $H_v \cap A = 1$ for all $A\in \mathcal A$ and for all $v'\in \Gamma^d$ such that $v \prec^d v'$, we have that $E_{v'} \subset A$ for some $A\in \mathcal A$. Then there is $a\in H_v^{|z|}$ such that $H_{v}\models \Theta^{j}(a)\neq 1\wedge\psi^{j}(a)$. If $H_v$ has trivial center, then as $\Theta^j(a)\ne 1$, there must exist $b\in H_{v}$ such that $[b,\Theta^j(a)]\neq 1$. In particular, as $H_v \notin \mathcal A$, we have that $\neg\xi(a)$ and $\neg\xi([\Theta(a),b])$. The observation above and part (\ref{assertion 2}) of Lemma \ref{l: simple parameter case} then implies that $Y_{a,b}^j=H_{v}\times H_{v}^{\perp}$.
			If $v\in V_{0}$ and satisfies the maximality condition, then as soon as we can find $a\in H_{v}^{|z|},b\in H_{v}$ where $b$ does not commute with $a$
			it will follow that $H_{v}\times H_{v}^{\perp}=Y_{a,b}$ in exactly the same way.
			
			Assume now that we are given $Y=Y^{j}_{a,b}\in\Y_{\max}$ (we may assume $Y^{j}_{a,b}\neq\emptyset$).
			By Observation \ref{o: meaning of xi}, if $j\geq 1$ there is some component of the element $[\Theta^{j}(a),b]$ which lies outside every $A\in\mathcal{A}$. In particular $[\Theta^{j}(a),b]\ne 1$ and so $\subg{a,b}$ is non-abelian. Hence, there is a non-abelian component $\hh$ of $\subg{a,b}$ such that $\hh \not\in \mathcal A$ and $\Theta^{j}(a^{\hh})\neq 1$. If $j=0$ another application of the Observation yields some component $\hh\in\ff(a,b)$ not contained in
			$A$ for any $A\in\mathcal{A}$.
			
			By Lemma \ref{l: reducing to one block}, we have that  $Y_{a,b} \subset Y_{a^{\hh},b^{\hh}}$. Therefore, as $(a^\hh, b^\hh) \subset \hh$ and $(a^\hh, b^\hh) \not\subset \mathcal A$, from maximality we can assume that $|\ff(a,b)|=1$ and $\ff(a,b)$ is not in $\mathcal A$.
			
			In fact, we know that $\subg{a^{\hh},b^{\hh}}$ is contained in some $E_v$ for $v\in \Gamma^d$ and $E_v$ is either singular or dihedral, as by assumption $\subg{a,b}<E_v$ is non-abelian. We know that $Y_{a,b}^{j}\subseteq E_v\times E_{v}^{\perp}$. Since in the singular case, we assume that $G_{v}$ has trivial center, it follows from Lemma \ref{l: simple parameter case} that in fact $E_v\times E_v^{\perp}\in\mathcal{Y}$, so any $Y\in\mathcal{Y}_{\max}$ is of this form.
			
			To conclude, assume that $v\in \Gamma^d$ satisfies that for all $v'\in \Gamma^d$ such that $v\prec^d v'$, then $E_{v'}\subset A$ and $E_v$ has trivial center. As $E_v$ is non-abelian and so in particular $|E_v|\ne 2$, by the properties of the family $\mathcal A$, we have that if $E_v \cap A \ne 1$, then $E_v \in \mathcal A$. As by assumption $E_v\not\in \mathcal A$, as we argued above, we have that $E_v\times E_v^\perp=Y_{a,b}$ where $\subg{a,b} < E_v$ and $\subg{a,b}\not\subset A$ for all $A\in \mathcal A$, therefore the defining parameters of $Y_{a,b}$ satisfy the conditions of $\mathcal Y_{\max}$. The maximality of $Y_{a,b}$ follows from the definition of $\prec^d$ and the condition that for all $v'\in \Gamma^d$ such that $v\prec^d v'$, then $E_{v'}\subset A$. Therefore $E_v\times E_v^\perp \in \mathcal Y_{\max}$.
		\end{subproof}
		
		In the next lemma we show that for the abelian vertices which are maximal relative to $\mathcal A$, the subgroup $E_v\times E_v^\perp$ is also definable unless $E_v$ is of order 2 and it is lost by the restriction on the choice of parameters, that is $E_v < E_{v'}$ where $E_{v'}$ is dihedral and $E_{v'}\subset A$ for some $A\in \mathcal A$.
		\begin{lemma}
			\label{l: second sieve lemma} Let $\mathcal{X}(\mathcal A)=\{C(g)\,|\,\G\models \neg\xi(g)\}$ and let $\mathcal{X}_{\max}(\mathcal A)$ be the collection of all
			members of $\XX(\mathcal A)$ that are not strictly contained in any member of $\mathcal{Y}_{\max}(\mathcal A)\cup\mathcal{X}(\mathcal A)$. Then any $X\in\XX(\mathcal A)$ is of the form $E_v \times E_{v}^{\perp}$ for some $v\in \Gamma^d$.
			
			Moreover, $E_v \times E_{v}^{\perp}\in \XX(\mathcal A)$ if and only if  $v\in \Gamma^{d}$ satisfies the following two conditions:
			\begin{itemize}
				\item {\rm(}relative maximality with respect to $\mathcal A${\rm)} for all $v'\in \Gamma^d$ such that $v \prec^d v'$, then $E_{v'}\subset A$ for some $A\in \mathcal A$;
				\item $E_v$ has nontrivial center if $j\ge 1$ or it is abelian if $j=0$ and if $|E_v|=2$, then there is no $v'\in \Gamma^d$ such that $v \prec^d v'$, $E_v < E_{v'}$ and $E_{v'}\subset A$ for some $A\in \mathcal A$.
			\end{itemize}
			
			In this case, we can view $X = Z_{a,b}$ where $b=1$ and $a=(g,\bar 1)\in Z(E_v)\setminus \{1\}$.
		\end{lemma}
		\begin{subproof}
			Clearly for any $X=C(g)\in\XX_{\max}(\mathcal A)$ we may assume that $|\ff(g)|=1$ and does not belong to any $A\in\mathcal{A}$. If $g$ is non-singular, then by Lemma \ref{l: description of centralizers} and the definition of the $d$-completion, we have that $C(g)=E_v\times E_v^\perp$ for some $v\in \Gamma^d$.
			
			It is also easy to see that if $G_{v}$ has nontrivial center, then $G_{v}\times G_{v}^{\perp}=C(g)\in\XX(\mathcal A)$ for $g\in Z(G_{v})\setminus\{1\}$. Furthermore, if in the latter case $v$ satisfies the two conditions of the statement, then the relative maximality assures that $X=G_{v}\times G_{v}^{\perp}$ is not properly contained in any $X'\in\XX(\mathcal A)$ and, as $\mathcal A$ is closed under singular vertices (which do not belong to dihedral groups, i.e. singular vertices of order 2 dominated by a dihedral group), the second condition assures that we can take the parameter $g\in G_v$ and $g\notin A$ for all $A\in \mathcal A$ and so $C(g)=G_v \times G_v^\perp \in \XX_{\max}(\mathcal A)$.
			
			To finish it suffices to show that if $g\in G_{v}$ is not central in $G_{v}$, then there exists $Z\in\XX(\mathcal A)\cup\Y(\mathcal A)$ strictly containing $C(g)$. If $Z(G_{v})\neq \{1\}$, then clearly $C(h)\supsetneq C(g)$ for $h\in Z(G_{v})\setminus\{1\}$ and $C(h)\in \XX(\mathcal A)$. Otherwise, $G_{v}\times G_{v}^{\perp}\in\Y$ by Lemma \ref{l: first sieve lemma} and this set properly contains $C(g)$.
		\end{subproof}
		
		The proof of the proposition ends with the observation that the family $\mathcal{Z}(\mathcal A)=\XX_{\max}(\mathcal A)\cup\Y_{\max}(\mathcal A)$ satisfies the required conditions by virtue of Lemmas \ref{l: first sieve lemma} and \ref{l: second sieve lemma}.
	\end{proof}

	\begin{notation}
		We denote by $U^j(\mathcal A) \subset \G^{|z|+1}$, the definable set of parameters $(a,b)\in  \G^{|z|+1}$ such that $Y=Z_{a,b}^j\in \mathcal Z(\mathcal A)$. When $\mathcal A$ is clear from the context, we sometimes omit it and just denote the definable set of parameters by $U^j$.
	\end{notation}

	\begin{remark}
		Notice that if one considers free products of groups in the class $\G(\mathbf{K}_{n},\mathbf{C}_{\Phi})$ without two torsion, as by definition all singular vertex groups are maximal and $E_v^\perp=1$ for all $v\in \Gamma^d$, we deduce as a corollary of \rm{Theorem \ref{p: main definability proposition}}, the definability of the graph $\Gamma^c=\Gamma^d$ (and almost immediately the interpretability of the core).
	\end{remark}

	\section{Interpretability of the underlying group-labelled graph}\label{s: interpretation_of_base_structure_(first)}

	\subsection{Decomposition and main result}\label{s: proof_of_interpretability}
	
	\newcommand{\ZZ}[1]{Z_{#1}}
	\newcommand{\Zpar}[1]{K_{#1}}
	\newcommand{\Zort}[1]{K_{#1}^{\perp}}

	The main goal of this section is to show that the (extended core of the) base structure is uniformly interpretable in an adequately large class $\textbf C$ of graph products.
	
	The proof has two main steps. The first step is to show that the group-labelled graph $\Gamma^d$ is uniformly interpretable in the
	class. This proof is inductive in nature: at the first step, one considers the definable sets provided by Theorem \ref{p: main definability proposition} (in the case $\mathcal A =\emptyset$), which are of the from $E_v\times E_v^\perp$ for $\prec^d$-maximal vertices $v$ of $\Gamma^d$, and then Theorem \ref{thm: main isolation theorem} isolates the vertex groups (modulo some central involutions). We set this family to be $\mathcal A_1$. In the next step, we consider again the definable sets provided by Theorem \ref{p: main definability proposition} but relative to the set $\mathcal A_1$. In this case, we obtain definable sets of the form $E_v\times E_v^\perp$ but now for vertices $v$ of $\prec^d$-level 2, that is, maximal relative to $\mathcal A_1$. In this case, given this new family of definable sets, Theorem \ref{thm: main isolation theorem} shows (roughly speaking) that $E_v \times\prod_{v \prec^d w} E_w$ is definable (modulo some central involutions). Proceeding this way level by level, we show that there is a family of definable sets of the form $E_v \times \prod_{v \prec^d w} E_w$. We finally use this family to interpret the graph $\Gamma^d$ and the associated vertex groups. For this step, as one can see in the strategy, Theorem \ref{thm: main isolation theorem} is the main technical result that encapsulates the induction step in the procedure, that is it takes some definable sets of the form $E_v\times E_v^\perp$ and isolates $E_v$ (relative to vertices that dominate $v$ in the graph $\Gamma^d$ and some central involutions).
	
	The second step is to show that (the extended core of) $\Gamma$ and the associated vertex groups are uniformly interpretable in the class $\textbf C$. In this case, the main difficulty is to distinguish using first-order sentences the singular infinite cyclic and dihedral groups from the non-singular ones. We show that we can always distinguish between singular and non-singular dihedral groups but in the case of infinite cyclic groups, we were not able to (and conjecturally one cannot) distinguish the infinite cyclic subgroup corresponding to singular weak vertex groups from the non-singular cyclic subgroups. For this reason, one can only interpret the core of the graph $\Gamma$ and not recover the whole graph $\Gamma$. 
	
	\bigskip
	
	We will use the set of sets $\mc{Z} =\{Z_{a,b}^j\mid (a,b)\in U^j, j=1,\dots, |\Phi|\}$ from Theorem \ref{p: main definability proposition}. By Theorem  \ref{p: main definability proposition} we have that each $Z_{a,b}\in \mc{Z}$ has the form $E_v \times E_v^\perp$ for some $v\in \Gamma^d$.
	
	In the next definition, we describe a decomposition of a set of the form $Z=E_v \times E_v^\perp$ in to three pieces: the subgroup $H$ with the property that all the parameters in $H^{|z|+1} \cap U^j$ define the same set $Z$, the subgroup $U$ such that all the parameters in $U^{|z|+1} \cap U^j$ define sets that strictly contain $Z$, and the subgroup $L$ such that the parameters in $L^{|z|+1} \cap U^j$ define sets which do not contain $Z$, that is, they are either strictly smaller than $Z$ or do not contain $Z$ and contain elements outside $Z$.
	
	\begin{definition} [Level decomposition of sets]\label{defn:decomposition definable}
		Let $Z= E_v\times E_v^\perp$.
		\begin{itemize}
			\item Let $\Delta_H < \Gamma^d$ be the induced subgraph with vertex set
			$$
			\{ v'\in \Gamma^d \mid  Z=E_{v'}\times E_{v'}^\perp\}=\{v'\in \Gamma^d\mid v =^d v'\}.
			$$
			Notice that $\Delta_H$ is a clique. Indeed, for all $v,v'\in \Delta_H$, $v\ne v'$ we have by definition that $E_v\times E_v^\perp=E_{v'} \times E_{v'}^\perp$. It follows that $E_v < E_{v'}^\perp$ and so $(v,v')\in E(\Gamma^d)$.
			
			Let $H=\prod\limits_{v'\in \Delta_H} E_{v'}$. It follows from the definition that $Z = H \times H^\perp$ and $H^\perp$ is a subgraph product of $\G$.
			
			\item Let $\Delta_{U} < \Gamma^d$ be the induced subgraph with vertex set
			$$
			\{ v'\in \Gamma^d \mid  E_{v'} \leq Z \lneq E_{v'}\times E_{v'}^\perp\}=\{v'\in \Gamma^d\mid  v\precneq ^d v'\}.
			$$
			Notice that $\Delta_{U}$ is a clique.
			
			Let $U = \prod\limits_{v'\in \Delta_{U}} E_{v'}$. By definition, $U < H^\perp$ and so $U$ is a subgraph product of $\G$.
			
			\item Let $\Delta_L < \Gamma^d$  be the induced subgraph with vertex set
			$$
			\{ v'\in \Gamma^d \mid  Z \nleq E_{v'}\times E_{v'}^\perp\}=\{v'\in \Gamma^d\mid  v \npreceq ^d v '\}.
			$$
			Let $L = \langle E_{v'} \mid v'\in \Delta_L\rangle$. By definition, $L < H^\perp$ and so $L$ is in fact a subgraph product of $\G$.
		\end{itemize}
		
		We have that $Z=H \times H^\perp$ and $H^\perp = U\times L$ and so $Z$ admits the decomposition $Z=H \times U\times L$ which we call the \emph{level decomposition of} $Z$. Sometimes if we want to stress that the factors are from the decomposition of $Z$ we write $H(Z), U(Z)$ and $L(Z)$.
	\end{definition}
	
	\begin{obs} \label{Obs: K structure}
		Assume that $\Gamma$ is reduced. Let $v\in \Gamma^d$ and $Z=E_v\times E_v^\perp$. Then we have the following properties.
		
		\begin{itemize}
			\item At most one $E_v < H$  is either singular, non-singular cyclic, or dihedral  not of the form $G_{v}*G_{w}$ for some $v,w\in \Gamma^e$. Indeed, suppose that $E_v$ is as above. By definition of $H$, we have that $Z=E_v\times E_v^\perp$. Then for any other $E_w < H$, we have that $E_w < E_v^\perp$ which is a subgraph product and so by the second item in {\rm Remark \ref{rem:properties Gammad}}, we have that $E_w \in \Gamma^e$ or $E_w =G_w \ast G_{w'}$ for $w,w'\in \Gamma^e$.
			
			Furthermore, if there are two singular factors in $H$, say $G_v < H$ and $G_{v'}< H$ for $v,v'\in \Gamma^e$, from the definition of $H$ we have that $G_v\times G_v^\perp = G_{v'} \times G_{v'}^\perp$ and so $\st(v)=\st(v')$ contradicting that $\Gamma$ is reduced.
			
			\item If some $E_v< H$ is non-singular cyclic or it is dihedral and not of the form $G_{w}* G_{w'}$ for $w,w'\in
			\Gamma^e$, then $\hh^{\perp}\not\subseteq Z$ for any $\hh\in\ff(L)$. Indeed, let $\G(\Delta_v)^h$ be the subgraph product of $\G$ that contains $E_v$. Notice that $E_v \lneq \G(\Delta_v)^h$. Then $\G(\Delta_v)^h < \mathcal H^\perp$ for all $\hh\in\ff(L)$ but $\G(\Delta_v)^h \not\subset E_v \times E_v^\perp$.
			
			\item if $\hh \in \ff(L)$, then $\hh$ has an irreducible action on the corresponding tree. Indeed, by the definition of $L$, no factor $\hh$ can be singular or dihedral as otherwise, if $\hh=E_v$ for some $v\in \Gamma^d$, then $Z < E_v\times E_v^\perp$ contradicting the definition of $L$.
		\end{itemize}
	\end{obs}
	
	Given a set $Z=E_v \times E_v^\perp$, the factor $H(Z)$ is the direct sum of $E_w$ such that $E_v\times E_v^\perp = E_w\times E_w^\perp$. The next theorem states that given a collection of definable pairs of subgroups $(U(Z),Z)$, we can also define the collection $E_w\times U(Z) (\times CI)$, where the central involutions $CI$ of $Z$ only appear when $E_w$ is a dihedral group. Roughly speaking, the theorem tells us that we can isolate the subgroups $E_v$ modulo the definable set $U$ which ``dominates" $E_v$ (that is, the parameters in $U^{|z|+1} \cap U^j$ define sets $Z'$ that properly contain $E_v\times E_v^\perp$).
	
	\begin{thm}
		\label{thm: main isolation theorem} Let $\textbf C\subseteq\mc{G}(\textbf{K}_N, \textbf{C}_{\Phi})$ {\rm(}see {\rm Definition \ref{d: K_N_and_C})} where $N>0$. Assume one of the following:
		\begin{itemize}
			\item $\Phi$ consists only of nontrivial positive sentences and $\Gamma$ is positive reduced for any $\GP$ in $\textbf C$.
			\item $\Phi$ consists of simple non-generic almost positive sentences and $\Gamma$ is almost positive reduced for any $\GP$ in $\textbf C$.
		\end{itemize}
		
		Assume we are given a collection of pairs of sets $(U,Z)$ uniformly definable in $\textbf C$, where $Z = E_v\times E_v^\perp$ for some $v\in \Gamma^d$ and $U=U(Z)$ is the factor in the level decomposition of $Z= H \times U\times L$, see {\rm Definition \ref{defn:decomposition definable}}.
		
		Then the family consisting of the following sets is uniformly definable in $\textbf C$:
		\begin{itemize}
			\item $U(Z)\times E_v$, where $v\in \Gamma^c$ and $Z=E_v\times E_v^\perp$ and
			\item $(U(Z)\times E_v) \cdot CI(Z)$, where $v\in \Gamma^d \setminus \Gamma^c$, $Z = E_v\times E_v^\perp$ and $CI(Z)$ is the group generated by all the central involutions of $Z$.
		\end{itemize}
		
		What is more, it splits into two uniformly definable families: one which contains all sets of the first type and another one which contains all sets of the second type for which $CI(Z)$ is not contained in $U(Z)$ {\rm(}not necessarily only those{\rm)}.
	\end{thm}
	
	\begin{obs}\label{obs:involutions}
		The subgroup of central involutions $CI$ of $Z$ is precisely the central involutions of the {\rm(}unique{\rm)} singular vertex in $H$ {\rm(}if exists{\rm)} and the central involutions of $U(Z)$. Hence, have that $CI = \prod\limits_{\substack{w\in\Gamma^{e}\\ v\preceq^{d} w}}CI(E_{w})$.
	\end{obs}
	
	\subsection{Isolating vertex groups: proof of Theorem \ref{thm: main isolation theorem}}
	\label{ssct: isolating}
	
	Given a definable collection of  pair $(U(Z), Z)$, our goal is to show that the collection $E_w\times U(Z)$ is also definable for all $E_w$ such that $E_w\times E_w^\perp = Z$. We prove this in several lemmas that progressively reduce the $Z$ to smaller subgroups until we will reach the desired  subgroup $E_w \times U(Z)$.
	
	In the next lemma, we show that in the nontrivial positive case, we can remove the subgroup $L(Z)$, that is, if the collection of pairs $(U(Z), Z)$ is definable, so is the collection $H(Z) \times U(Z)$. The difference between the nontrivial positive and the simple non-generic almost positive case is due to the fact that, in the nontrivial positive case, all parameters in the vertex group define the same definable set while this fact does not hold in the simple non-generic almost positive case. The main observation is that in the subgraph product $L(Z)$ there are no central vertices and so for each $E_u < L$, there exists $E_{u'}<L$ such that $E_{u'} \not< E_u \times E_u^\perp$. So if one considers the intersection of all the definable sets with parameters in $Z$, one defines the collection of subgroups $H(Z) \times U(Z)$. If we use this argument in the simple non-generic almost positive case, since parameters in a vertex groups $E_w$ could define a subgroup properly contained in $E_w\times E_w^\perp$ we would only obtain a proper subgroup of $H(Z) \times U(Z)$.
	
	\begin{lemma}
		\label{l:definability positive}
		In the setting of {\rm Theorem \ref{thm: main isolation theorem}}, assume that $\Phi$ consists only of nontrivial positive sentences and $\Gamma$ is positive reduced for any $\G(\Gamma,\gvv)$ in $\textbf C$.
		Then there is a a formula with an associated subdivision of its free variables which in each $\G\in\textbf C$ defines the family of sets consisting of $U(Z)\times H(Z)$ for any $Z\in\mathcal{Z}$.
	\end{lemma}
	\begin{proof}
		Define $P(Z_{a,b})$ to be the
		$$
		P(Z_{a,b})=\bigcap\limits_{a',b'\subset Z_{a,b}} Z_{a',b'}.
		$$
		
		Let $Z\in \mathcal Z$, then from the level decomposition of $Z$, see Definition \ref{defn:decomposition definable}, we have that $Z= H\times U \times L$. From Observation \ref{Obs: K structure}, we have that $H$ has at most one component $E_v$ which is either singular, infinite cyclic or dihedral not of the form $G_v\ast G_w$; the other components are dihedral groups of the form $G_v\ast G_w$. We claim that $P(Z_{a,b})= H \times U$.
		
		Let us first show that $H\times U < Z_{a',b'}$ for all $a', b' \subset Z$. From the description of parameters, see Theorem \ref{p: main definability proposition}, we can assume that $|\ff(a',b')|=1$ and so $\langle a',b' \rangle$ is a subgroup of one of $U, H, L$. If $\langle a',b' \rangle \leq U \times H$, then from the properties of $H$ and $U$, see Definition \ref{defn:decomposition definable}, it follows that $Z < Z_{a',b'}$. If $\langle a',b' \rangle \leq L$, then $H \times U = L^\perp < Z_{a',b'}$. Therefore, we have that $H \times U< P(Z_{a,b})$.
		
		The equality is obtained as follows. By the description of parameters, see Theorem \ref{p: main definability proposition}, there are $(a,b) \subset Z$ such that $Z=Z_{a,b}$. If $g \in P(Z)$, as $(a,b)\subset Z$, we have that $g\in Z_{a,b}=Z$. Hence $g=g_0g_1$ for some elements $g_0\in U\times H$ and $g_1 \in L$. Since $U\times H\leq P(Z)$ and $P(Z)$ is a subgroup, we have that $g_1 \in P(Z)$. We want to show that, in fact, $g_1=1$. Indeed, assume towards contradiction that $1\ne g_1 \in L$ and that $g_1^{\hh}\ne 1$ for $\hh \in \ff(L)$. From Observation \ref{Obs: K structure}, we have that $\hh$ cannot be singular or dihedral, i.e. $\hh$ admits an irreducible action on a tree $T_v$. From Corollary \ref{c: stable axis}, there is a hyperbolic element which is not of smaller complexity in $\hh$, it is $(n+3)$-stable (with respect to the action of $\hh$ on $T_v$) and $g_1$ does not belong to the stabiliser $E(h)$ of $h$. In particular, $g_1^{\hh}$ and $h$ have an irreducible action on $T_v$. Then $g_1 \notin Z_{a',b'}$ for $(a',b') \subset Z \cap \subg{h}$ and so $g_1 \notin T(Z)$. Therefore, $P(Z)=H\times U$.
		\end{proof}
	
	We have define $L$ to be the subgroup whose parameters in $L^{|z|+1} \cap U^j$ define sets which do not contain $Z$. We consider $L_0$ to the direct sum of factors of $L$ with the property that all the parameters in $L_0$ define sets which are strictly contained in $Z$; its complements in $L$, called $L_1$ is the direct sum of factors that contain at least one parameter such that the definable set does not contain $Z$ but contains an element from outside of $Z$. In the next lemma we show that given a definable collection of pairs $(U(Z), Z)$, then the collection $H\times U \times L_0$ is also definable.
	
	\begin{lemma}\label{l: removing components}
		For any $Z\in\mathcal{Z}$ consider the decomposition given in {\rm Definition \ref{defn:decomposition definable}} and let $L_{0}(Z)$ be the the subgroup of $L$ generated by the collection $\ff_{0}$ of all those $\hh\in\ff(L)$ for which $C(g)\subseteq Z$ for any $g\in\hh$
		and $T(Z)=U(Z)\times H(Z)\times L_{0}(Z)$. Then the family
		$\{T(Z)\}_{Z\in\mathcal{Z}}$ is uniformly definable in $\textbf C$.
	\end{lemma}
	\begin{proof}
		We will omit the explicit dependency on $Z$ from the notation.
		
		Let $\ff_{1}=\ff(L)\setminus\ff_{0}$ and let $L_{1}=\prod_{\hh\in\ff_{1}}\hh$. Let $A_0$ be the set of $g\in Z\setminus U$ such that
		\begin{enumerate}
			\item \label{condition 2A0} $g^2$ (and so $g$) is not almost central in $Z$ relative to $U$: $\exists h\in Z\,\, (g^2)^hU\ne (g^2)^{\pm 1}U$,
			\item \label{condition 3A0} $g$ has maximal centralizer in $Z$ relative to $U$ and within the class of those elements which are not central in $Z$: $\neg\exists h\in Z$ such that $C(g)U \subsetneq C(h)U\not\supseteq Z$
			\item \label{condition 4A0} either $\exists h'\notin Z$, $[g,h']=1$ or $\exists h\in Z$, $\exists h' \notin Z$, $[g,h]\notin U$ and $[h,h']=1$.
		\end{enumerate}
		
		We start by observing that $A_0 \subset Z(T) \times L_{1}$, where $Z(T)$ is the center of $T$. Take $a\in A_{0}$ and let $a_{0}=\pi_{T}(a)$ and $a_{1}=\pi_{L_{1}}(a)$. We claim that $a_{1}\neq 1$, or equivalently, that $a\nin T$. Suppose towards contradiction that $a\in T$.
		
		We first show that $a$ has a unique non-central component. Let $\ff_{nz}$ be the collection of all $\hh\in\ff(H\times L_{0})$ such that $a^{\hh}$ is not central in $Z$ (equivalently, the corresponding component). By condition \ref{condition 2A0}, we have that $a$ is not almost central so this forces $\ff_{nz}$ to be non-empty. Let us now see that there is precisely one component, i.e. $|\ff_{nz}|=1$. Indeed, Condition \ref{condition 3A0} implies that the centralizer of each $a^{\hh}$ for $\hh\in\ff_{nz}$ coincides with that of $a$, but, if there is more than one component, this can only happens if each such $a^{\hh}$ is central in the corresponding factor, against the definition of $\ff_{nz}$. Therefore, $a$ has a unique non-central component, which it must belong to some $K\in \ff(H)$ or to some $\mathcal H\in \ff(L_0)$.
		
		Let $a=a' c$, where $a'$ is the non-central component of $a$ and $c$ is a central element in $Z$, i.e. $c$ is the product of the central components of $a$. Let us see that $a'\notin E_v < H$.
		
		Assume first that $E_v$ is singular, that is $E_v=G_v$ (if it exists). As $Z=G_v\times G_v^\perp$, we have that $C(a)=C(a') < Z$ and so $a$ does not commute with any $h$ such that $h\notin Z$. Moreover, any $h\in Z$ that does not commute with $a$ modulo $U$ must have some component in $G_{v}$. This implies that $\subg{h}^{\perp}\subseteq G_{v}^{\perp}\subseteq Z$. Since $C(h)$ is the product of a subgroup of $Z$ (generated by singular vertex groups containing the components of $h$ or the cyclic subgroups they generate) and $\subg{h}^{\perp}$ it follows that $C(h)\leq Z$ and thus $a$ cannot fulfill Condition \ref{condition 4A0} of the statement either. Therefore $a' \notin G_v$.
		
		Let us now see that $a' \notin E_v$ for any non-singular cyclic or dihedral $E_v < H$. Indeed, it follows from the Condition \ref{condition 2A0}, as $a^2$ does not belong to $Z(E_v)U$ so $a'$ cannot be an involution, nor cyclic and nor almost central.
		
		Similarly, we have that $a' \notin \mathcal H$, for all $\mathcal H\in \ff(L_0)$. From the definition of $L_0$, we have that $C(a)<Z$. If $h$ does not commute with $a$, $h$ has a component in $\mathcal H$. Arguing as above, we have that $C(h) < C(h^{\mathcal H}) < Z$ and so $a$ does not satisfies Condition \ref{condition 4A0} in the statement.
		
		We have ruled out $a\in T$. It is now easy to see that $a_{0}\in Z(T)$, as otherwise it would violate Condition \ref{condition 3A0} as $C(a)\subsetneq C(a_{1})$ and no centralizer of an element in $L_{1}$ can contain $Z$.
		
		Let
		$$
		X = Z\cap \bigcap\limits_{g\in A_0}C(g).
		$$
		We claim that $X=T$.
		
		As $A_0 \subset Z(T) \times L_1$, it follows that $T\leq C(g)$ for any $g\in A_{0}$, so clearly $T\leq X$.
		From Observation \ref{Obs: K structure} we have that $\hh\in\ff(L_{1})$ admits an irreducible action on the corresponding tree.

		Let $\Fill(\hh)=\{g\in\hh\,|\,\ff(g)=\{\hh\}\}$. Notice that $\Fill(\hh)$ is the set of elements in $\hh$ that are hyperbolic in all the trees associated to $\hh$. It is easy to see that $\Fill(\hh)\subseteq A_{0}$. Pick $g\in \Fill(\hh)$. Conditions \ref{condition 2A0} is clear as $g$ is not almost central in $\hh$. The description of centralizers in graph products (restricted to $\hh$) implies that $g$ satisfies Condition \ref{condition 3A0} as well: if $g_{0}\in Z$ is such that    $C(g)\subseteq C(g_{0})$ then $g_{0}^{\hh}\in C(h)\cap\hh=\subg{g}$ and $C(g)\neq C(h_{0})$ implies that    $g_{0}^{\hh}=1$ and hence $Z\leq C(g_{0})$.
		
		Now, by definition of $L_{1}$ there exists $G_{v}\leq\hh$ such that $ G_{v}^{\perp}\not\leq Z$ and any $h\in G_{v}\setminus\{1\}$ will satisfy $[h,g]\notin U$ and $C(h)\not\subseteq Z$ and thus witnesses that $g$ satisfies Condition \ref{condition 4A0} as well.
		
		Since $\bigcap_{h\in \Fill(\hh)}C(h)=\hh^{\perp}$, it follows that
		$$
		X_{1}=Z\cap \bigcap_{g\in A_{0}}C(g)\subseteq Z\cap\bigcap_{\hh\in\ff(L_{1})}\bigcap_{h\in \Fill(\hh)} C(h)\subseteq Z\cap\bigcap_{\hh\in\ff(L_{1})}\hh^{\perp}=T
		$$
		and we are done.
	\end{proof}
	
	Recall that our goal is to isolate the factors of $H$ (modulo the vertex groups that dominate it), that is to define $K \times U(Z)$ where $K\in \ff(H)$. Our next lemma is directed to isolate the dihedral components that have an involution which commutes with an element outside of $Z$. In the next lemma, we show that this set of involutions, called $I(T)$ is definable (modulo $U$).

	\begin{lemma}\label{lem: I(t)}
		Let $Z \in \mathcal Z$ and let $T$ be the definable set $U\times H \times L_0$ from {\rm Lemma \ref{l: removing components}} in the almost positive case, and the set $U\times H$ from {\rm Lemma \ref{l:definability positive}} in the positive case.
		
		Define $I(T)$ to be the set of involutions with maximal centraliser and commuting with an element outside of $Z$, that is
		$$
		I(T)= \{ g\in T\setminus U\, \mid\, g^{2}=1\wedge \neg\exists h\in T \,\, C(g) \subsetneq  C(h) \ne \mathcal G, \, \exists h\notin Z \, [g,h]=1\}.
		$$
		
		\begin{itemize}
			\item Then $I(T)\subset\bigcup_{K\in\mathcal{K}_{d}}KU$, where $\mathcal K_d$ is the collection of all $K\in \ff(H)$ which are dihedral such that $I(T) \cap K \ne \emptyset$.
			
			\item Let $K\in \ff(H)$ be a dihedral subgroup generated by two involutions $e$ and $e'$. If $e\notin K\cap I(T)$, then $e\in G_v$ for some $v\in \Gamma^e$ and $G_v^\perp < Z$.
		\end{itemize}
	\end{lemma}
	
	\begin{proof}
		
		From the decomposition of the sets $Z$, see Definition \ref{defn:decomposition definable}, we have that if $K\in \ff(H)$, then $Z = K \times K^\perp$. Furthermore, from the definition of $\mathcal K_d$, we have that $I(T)\cap H \subset \mathcal K_d$. From the definition of $L_0$ we have that $C(h)\leq Z$ for all $h\in L_0$. It follows that if $g\in I(T)$, then $g$ does not have any component in $L_0$ and in $K\notin K_d$, that is, $I(T)\subseteq U\times\prod_{K\in\mathcal{K}_{d}}K$. Furthermore, the maximality condition for the centraliser of $g\in I(T)$ implies that $I(T)\subset\bigcup_{K\in\mathcal{K}_{d}}KU$.

		Assume that $K= \langle e\rangle \ast \langle e'\rangle$, $e=e_1 \dots e_r$ and $e'=e_1'\cdots e'_{r'}$ where $e_i, e_j'$ are singular involutions and $\{e_i\}$ (resp. $\{e_j\}$) pairwise commute and for each $e_i$, $i=1, \dots,r$, there $e_j'$, $j=1, \dots, r'$ such that $[e_i,e_j]\ne 1$. Suppose that $r>1$. Then $e_1 \notin Z$, otherwise, as $Z$ is a subgroup we would have that $ee',e_1e' \in Z$ and as $[e_1,e']\ne 1$, we would have that $ee'$ and $e_1e'$ are infinite order elements that do not commute so $e,e',e_1$ belong to the same component of $Z$ which is not dihedral. This contradicts that $K$ is a component of $Z$. Therefore, $e_1 \notin Z$  and $[e,e_1]=1$ so $e\in I(T)$. Therefore, if $e\notin I(T)$, we have that $r=1$ and so $e$ is singular. From the definition of $I(T)$ we have that $G_v^\perp < Z$.
	\end{proof}

	In the next lemma we isolate the dihedral factors of $H$ with an involution in $I(T)$.
	
	\begin{lemma}\label{l: isolating dihedrals}
		For any $Z\in \mathcal{Z}$, let $\mathcal{K}_{d}(Z) \subset \ff(H)$ be as in {\rm Lemma \ref{lem: I(t)}}, and let $CIT(Z)$ be the subgroup of central involutions of $T$. Then the family
		$$
		\bigcup_{Z\in\mathcal{Z}}\{U(Z)\cdot CIT(Z)\cdot K\}_{K\in\mathcal{K}_{d}(Z)}
		$$
		is uniformly definable in $\textbf C$.
	\end{lemma}
	
	Notice that by Observation \ref{obs:involutions} we have that $CIT(Z)$ is precisely $CI(Z)$, that is is the product of central involutions of the unique singular $G_v < H$ if it exists and the central involutions of $U(Z)$.
	\begin{proof}
		
		From Lemma \ref{lem: I(t)}, we have that $I(T)\subset\bigcup_{K\in\mathcal{K}_d}KU$.
		
		For all $g\in I(T)$ let $H(g)=\{ h\in T \mid h^g=h^{-1}\}$ and $D(g)=gH(g)\cup H(g)$.
		Notice that since $g\in KU$ for some $K\in\ff(H)$, for all $h\in H(g)$ we have that $gU,hU$ generate a dihedral in $KU/U\cong K$. Furthermore, from the definition, we have that $D(g) < K\cdot CIT(Z) \cdot U$ and $D(g)/U \simeq K\cdot CIT(Z)$. As $U$ and $D(g)$ are definable definable, so is $D(g) \cdot U = K\cdot CIT(Z) \cdot U$. Notice that $CIT(Z)$ is contained in the direct sum of $U$ with the singular $K\in\mathcal{K}$ (if it exists), since no other factors of $T$ can contain central involutions.
		
		It follows that the collection
		$\{U\cdot K\cdot CIT(Z)\,:\,Z\in\mathcal{Z},K\in\mathcal{K}_d(Z)\}$
		is uniformly definable in $\textbf C$.
	\end{proof}

	In the next lemma we isolate the rest of the factors of $H$.
	\begin{lemma}
		\label{l: full isolation} For any $Z\in\mathcal{Z}$, the family $\{U\cdot K\}_{Z\in\mathcal{Z}, K\in \ff(H)\setminus \mathcal K_d}$ is uniformly definable in $\textbf C$.
	\end{lemma}
	
	\begin{proof}
		For $Z\in\mathcal{Z}$, let $T$ be the definable set $U\times H \times L_0$ from Lemma \ref{l: removing components} in the almost positive case, and the set $U\times H$ from Lemma \ref{l:definability positive} in the positive case.
		
		Define
		$$
		T'(Z)=U(Z)\times\prod\limits_{K\in\ff(H)\setminus \mathcal K_d} \times L_{0}(Z).
		$$
		
		Notice that $T'(Z)$ is precisely $T'(Z)=T(Z)\cap\bigcap\limits_{g\in I(T)}\bigcap\limits_{h\in D(g)} C(h)$, since $D(g)$ are dihedral (and so do not contain any central element).
		
		Let $\mathcal K_e= \ff(H)\setminus \mathcal K_d$.
		By Observation \ref{Obs: K structure}, we have that the number of singular, cyclic and dihedral components which are not of the form $G_v\ast G_{v'}$, for some $v,v'\in \Gamma^e$ in $\ff(H)$ is at most one. Let $K_0$ be this component if it exists. Furthermore, from Lemma \ref{lem: I(t)}, we have that if $K\in \mathcal K_e\setminus \{K_0\}$, then $K$ is dihedral of the form $G_v\ast G_{v'}$ and $G_v^\perp, G_{v'}^\perp < Z$.
		
		We next claim that if the graph $\Gamma$ is positive reduced, then $|\mathcal K_e|=1$. Indeed, suppose towards contradiction that there is $K \ne K'\in \mathcal K_e$. Suppose first that $K'=K_0$ is a non-singular cyclic subgroup or dihedral group generated by singular involutions but not of the form $G_v \ast G_{v'}$, for $v,v'\in \Gamma^e$. Then $\G(\supp(K_0))\not< Z$ but $\G(\supp(K_0)) < K^\perp=G_v\ast G_{v'}$ and so $\G(\supp(K_0))< G_v^\perp$ and by the definition of the set $I(T)$, see Lemma \ref{lem: I(t)} we have that $K\in \mathcal K_d$. Therefore, if $K_0$ is cyclic or dihedral not of the form $G_v\ast G_{v'}$ we have that $\mathcal K_e=\{K_0\}$. Suppose now that there is $K,K'\in \mathcal K_e$, $K$ dihedral and $K'$ dihedral (resp. singular). From Lemma \ref{lem: I(t)}, we have that $K,K'\in \mathcal K_e$ is of the form $K=G_{v_1}^{h_1}\ast G_{v_2}^{h_2}$, $K'=G_{v_3}^{h_3}\ast G_{v_4}^{h_4}$ where $\langle e_i\rangle=G_{v_i}^{h_i}$, $v_i\in \Gamma$, $h_i\in \G$,  and $(G_{v_i}^{h_i})^\perp < Z$, $1\le i\le 4$ (resp. $K'=G_{v_3}^{h_3}$ and $(G_{v_3}^{h_3})^\perp < Z$). By conjugating by $h_1^{-1}$ if necessary, we can assume that $h_1=1$. By the description of centraliser, we have that $h_2\in {(K^\perp)}^\perp$. If $h_2$ cannot be taken to be $1$ (i.e. $h_2\notin C(G_{v_1})\cdot C(G_{v_2})$), then, on the one hand we have that $h_2 \notin G_{v_2}\times G_{v_2}^\perp$ and so $h_2\notin Z$; and on the other hand, as $K'< K^\perp$ and (there is some) \footnote{the point is that the conjugator is not unique. I am thinking of the "shortest"}$h_2\in (K^\perp)^\perp$, we have that $h_2\in (K')^\perp$ and so $h_2\in (G_{v_3}^{h_3})^\perp$. But then $G_{v_3}^\perp \not< Z$-contradicting our assumptions. Therefore, if $h_2$ is nontrivial, in particular if $K$ is generated by two involutions that are conjugate, then $\mathcal K_e=\{K_0\}$. Assume now that $h_2=1$ and so in particular $v_1\ne v_2$. By a symmetric argument we can assume that $K= G_{v_1} \ast G_{v_2}$ and $K'=(G_{v_3}\ast G_{v_4})^{h}$, $v_i\in \Gamma$, $h\in \G$ (resp. $G_{v_3}^{h}$). Similarly, if $h\ne 1$, then we have that $h\in G_{v_1}^\perp$ but $h\notin G_{v_3}^\perp$ contradicting our assumptions. Hence $h=1$ but in this case, we have that $v_i\in \Gamma$ and the subgraph induced by the $v_i$, $i=1, \dots, 4$ contradicts the fact that $\Gamma$ is positive reduced. Therefore, we have shown that $|\mathcal K_e|=1$.
		
		Furthermore, in the nontrivial positive case we have that $L_0=1$. In the almost positive, it follows from Observation \ref{Obs: K structure}, that if $K_0$ is cyclic or dihedral not of the form $G_v\ast G_{v'}$ then $L_0=1$. If $K$ is (up to conjugacy) of the form $G_v\ast G_{v'}^h$, $v\in \Gamma$, $h\in \G$, $h\notin C(G_v)\cdot C(G_{v'})$, then $L_0=1$ as $h\in \mathcal \hh^\perp$ for all $\hh\in \ff(L_0)$ but $h\notin Z$. If $K_0$ is singular or $K$ is dihedral of the form $(G_v \ast G_{v'})^h$, $v,v'\in \Gamma$, $h\in \G$, then if the graph is almost positive reduced, we have that $L_0=1$.
		
		Therefore, there are two cases:
		\begin{enumerate}
			\item \label{second case}  either $\mathcal{K}_{e} = \{K\}$ and $L_0=1$; in this case, $T'(Z)=K \times U(Z)$; or
			
			\item \label{first case} $\mathcal{K}_{sc}=\emptyset$ and so in this case we have that $T'(Z)=U(Z)\times L_{0}(Z)$.
		\end{enumerate}
		
		We claim that there is a formula on the parameters defining $Z$ (which can be proved without underscoring those parameters) that determines which of the two cases above takes place.
		
		Let $Y_{a,b}^{j}$, $0\leq j\leq r$ be the definable families from Section 7, see Definition \ref{defn: Y}.
		
		In case \ref{second case} above, there is $0\leq j\leq r$ and a tuple $(a,b)\subseteq T'(Z)$ such that either  $(a,b)\nsubset U$ and $\Theta^{j}(a)\notin U(Z)$ if $j\geq 1$ and $Z \subseteq Y^{j}_{a,b}$. Indeed, by Lemma \ref{l: simple parameter case}, it suffices to take the $j$ corresponding to the sentence $\phi^{j}$ satisfied by $K\in \mathcal K_e$ in case $K$ is singular or $j=0$ in case $K$ is cyclic or dihedral, as then $Z=K \times K^\perp< Y_{a,b}^j$.
		
		We claim that such $j$ and associated $(a,b)$ cannot exist in Case \ref{first case}. Indeed if $\Theta^{j}(a)^{\hh}\neq 1$ for some $\hh\in\ff(L_{0})$, as $\hh$ is irreducible, there must be some $h\in\hh$ for which $\subg{a^{\hh},h}$ is irreducible.
		Recall that for any tuple $(a,b)$ and any $c\in Y_{a,b}^{j}$ there is no component $\hh\in\ff(a,b,c)$ such that $\subg{a,b,c}^{\hh}$ is irreducible and $\Theta^{j}(a^{\hh})\neq 1$ holds. It follows that $h\nin Y^{j}_{a,b}$ and so $Z\not< Y_{a,b}^j$. Therefore, the family of $Z\in \mathcal Z$ such that $\mathcal K_e$ is non-empty is definable.
	\end{proof}
	
	\begin{proof}[Proof of Theorem \ref{thm: main isolation theorem}]
		
		The proof of Theorem \ref{thm: main isolation theorem} follows from Lemma \ref{l: isolating dihedrals} and Lemma \ref{l: full isolation}.
	\end{proof}
	
	\subsection{Level by level analysis, recognizing vertex groups}\label{sec: levels}
	\newcommand{\tl}[0]{\tilde{\mathcal{L}}^{d}}
	
	The goal of this section is to show that different graphs, the $d$-graph $\Gamma^d$ and the core of $\Gamma$, are uniformly interpretable inside an appropriate class of graph products. In order to achieve this, we use the families of definable sets obtained in Theorem \ref{p: main definability proposition} and in Theorem \ref{thm: main isolation theorem}. The definable sets in Theorem \ref{thm: main isolation theorem} are given in terms of the Decomposition \ref{defn:decomposition definable}.
	
	In the notation of the level decomposition, given some $v\in\Gamma^{d}$ we define
	$$
	U_{v}=\prod\limits_{\substack{w\in\Gamma^{d}\\ v\precneq^{d} w}}E_{w}, \quad
	\bar{U}_{v}=U_{v}\cdot \prod\limits_{\substack{w\in\Gamma^{d}\\ v\simeq^{d} w}}CI(E_{w})
	$$
	where $CI(H)$ denotes the central involutions of a (sub)group $H$.
	
	In this notation, the definable sets provided by Theorem \ref{thm: main isolation theorem} are of the form $E_v \times U_v$ for $v\in \Gamma^d$.
	
	\begin{definition}\label{defn:small d-graph}
		Let $\G=\G(\Gamma, \gvv)$ be a graph product and let $(\Gamma^{d},\prec^{d})$ be as defined in Definition \ref{defn:D-completion} and Definition \ref{d:d-order}. We will denote by $\tilde\Gamma^d$ the subgraph of $\Gamma^d$ spanned by the vertices $E_v$ such that either $|E_v|\ne 2$ or $|E_v|=2$ and there is not $v'\in \Gamma^d$ such that $E_{v'}$ is dihedral, $E_{v}\subseteq E_{v'}$ and $v \prec^d v'$.
		
		We say
		that $v\in \tilde{\Gamma}^{d}$ is of level $r$ if $k=r$ is the maximum length of an ascending chain $v=v_{k}\prec^{d}v_{k-1}\dots\prec^{d} v_{1}$ in $\tilde{\Gamma}^{d}$ we denote by $\tilde{\mathcal{L}}_{r}$ the collection of all $v\in \tilde{\Gamma}^{d}$ of level $r$, by $\tilde{\mathcal{L}}_{\leq r}$ the collection of all vertices of level at most $r$ and so on.
	\end{definition}

	\begin{definition}\label{dfn: class section 9}
		For the remainder of the section, we let $\textbf C\subseteq\mc{G}(\textbf{K}_N, \textbf{C}_{\Phi}^{r})$ {\rm(}see {\rm Definition \ref{d: K_N_and_C})}, where $N,r>0$ and either:
		\begin{itemize}
			\item $\Phi$ consists only of nontrivial positive sentences and $\Gamma$ is positive reduced for any $\G(\Gamma,\gvv)$ in $\textbf C$.
			\item $\Phi$ consists simple non-generic almost positive sentences and $\Gamma$ is almost positive reduced for any $\G(\Gamma,\gvv)$ in $\textbf C$.
		\end{itemize}
	\end{definition}
	
	\begin{thm}
		\label{thm: almost interpretation of d-graph}
		Let $\textbf C$ be as above. Then there is a finite union of definable families whose interpretation in any $\G\in\textbf C$ coincides with any fixed families of triples of subgroups of $\G$ from the two forms below: {\rm(}possibly with repetitions{\rm)}:
		
		$$
		\mathcal{A}=\{(U_{v},E_{v}\times \bar{U}_{v},E_{v}\times E_{v}^{\perp})\}_{v\in \tilde{\Gamma}^{d}}, \quad
		\mathcal{A}'=\{(U_{v},E_{v}\times \bar{U}_{v},E_{v}\times E_{v}^{\perp})\}_{v\in \Gamma^{d}}.
		$$
	\end{thm}
	\begin{proof}
		\newcommand{\A}[0]{\mathcal{A}}
		\newcommand{\B}[0]{\mathcal{B}}
		
		For every $1\leq r\leq n$ (replace $n$ by the exact number of levels of $\Gamma$) let:
		$$
		\A_{r}=\{(U_{v},E_{v}\times \bar{U}_{v},E_{v} \times E_{v}^{\perp})\}_{v\in \tl_{\leq r}}
		$$
		Given $1\leq l\leq 3$ we will denote by $\A_{r}^{l}$ the family of $l$-th coordinates of members of $\A_{r}$.
		
		We show that the family $\mathcal{A}_{r}$ is uniformly definable in $\textbf C$ by induction on $r$.
		At every step we also show that the family
		$$
		\mathcal{B}=\{(U_{v},E_{v}\times \bar{U}_{v},E_{v}\times E_{v}^{\perp})\,|\,v\in\Gamma^{e},|G_{v}|=2,G_{v}\leq E_{w}, w\in\tl\}
		$$ is uniformly definable in $\textbf C$. We prove this in Lemma \ref{lem: definability of Gamma^d} below.
		
		For $r=N$ we have $\A_{r}=\A$; from the definition of the set of vertices of $\tl$ we have that $\A'=\A\cup\B$. Notice that the index at which the sequence $(\mathcal{A}_{r})_{r\geq 1}$ stabilizes might depend on the particular graph product under consideration.
		
		The general strategy for the induction step is as follows. Assuming the property has been shown for $\A_{r-1}$, we apply Theorem \ref{p: main definability proposition} to obtain a family of definable sets that contains the sets of the form $E_v\times E_v^\perp$ for vertices $v\in \tl_{r}$. We then use the tools from Section \ref{ssct: isolating} to obtain subgroups of this collection and filter them in an uniform way, from which the uniform definability of $\A_{r}\setminus\A_{r-1}$ and hence $\A_{r}$ follows.
		
		To get the induction started we set $\A_{0}^{1}=\{\{1\}\}$.
		
		Assume now the result has been proven for $\A_{r-1}$.
		
		We apply Theorem \ref{p: main definability proposition} to the family $\bigcup_{A\in\A_{r-1}^{1}}\ff(A)$, yielding a finite union $\mathcal{Z}_{r}$ of definable families such that any graph product $\mathcal{G}=\mathcal{G}(\Gamma,\{G_{v}\}_{v\in V})\in\textbf C$ interprets $\mathcal{Z}_{r}$ in such a way that any $Z \in \mathcal{Z}_{r}$ is of the form $E_{v}\times E_{v}^{\perp}$ where $v\in \Gamma^d$ and moreover, given $v\in V^{d}$ we have $E_{v}\times E_{v}^{\perp} \in \mathcal Z_{r}$ precisely when the following two conditions are satisfied:
		\begin{itemize}
			\item (relative maximality with respect to $\mathcal A^{1}_{r-1}$) for all $v'\in \Gamma^d$ such that $v \prec^d v'$, then $E_{v'}\subset A$ for some $A\in \mathcal A^{1}_{r-1}$;
			\item (order 2 and not dihedral dominated): if moreover $|E_v|=2$, then there is no $v'\in \Gamma^d$ such that $v\prec^d v'$, $E_v \subset E_{v'}$  ($E_{v'}$ is dihedral) and $E_{v'}\subset A$ for some $A\in \mathcal A^{1}_{r-1}$.
		\end{itemize}
		
		We claim that $Z\in \mathcal Z_{r}$ if and only if $v\in\tl_{r}$.  Indeed, from Theorem \ref{p: main definability proposition} we can assume that $|E_v|>2$. If $v\in \Gamma^d$ has level $r$, then from the induction assumption, it satisfies the relative maximality with respect to $\mathcal A_{r-1}^1$ and so $E_v\times E_v^\perp\in \mathcal Z_{r}$. Conversely, if $E_v\times E_v^\perp \in \mathcal Z_{r}$, it must have level $r$ as otherwise, if the level is less than $r$ by the induction assumption $E_v \subset A$ for some $A\in \mathcal A_{r-1}^1$ and if the level is more than $r$, then $v\prec^d v'$ where the level of $v'$ is $s \ge r$ and by the induction hypothesis $E_{v'}$ is not contained in $A\in \mathcal A_{r-1}^1$ and so it does not satisfy the first condition.
		
		We claim that the family $\mathcal{W}_{r}=\{(U(Z),Z)\}_{Z\in \mathcal Z_{r}}$ is uniformly definable in $\textbf C$. For each $Z\in \mathcal Z_r$, we have that $Z=E_v\times E_v^\perp$ and $$
		U(Z)=U_{v}=\prod\limits_{\substack{w\in\Gamma^{d}\\ v\precneq^{d} w}}E_{w}.
		$$ For all $w\in \Gamma^d$ such that $v\precneq^{d} w$, we have that $(U_w, E_w \times \bar{U}_w, E_w\times E_w^\perp)\in \mathcal A_{r-1}$. Furthermore, we have that $E_w\times \bar{U}_w \leq U(Z)$, since any component of an element in $\bar{U}_{w}$ is contained in some $E_{u}$ with $w\preceq^{d}u$, including in the case of a central involution in $\bar{U}_{w}\setminus U_{w}$. Since the number of summands in $U(Z)$ is bounded by $N$, we can write $U(Z)$ as the collection of all $g$ that can be written as a product of $h_{1},\dots, h_{N}$ where for each $i$  there is $(U_{w_i},E_{w_i}\times \bar U_{w_i}, E_{w_i} \times E_{w_i}^\perp)\in\mathcal{A}_{r-1}$ such that
		$Z\lneq E_{w_i} \times E_{w_i}^\perp$ and $h_{i}\in E_{w_i}\times \bar U_{w_i}$.
		
		We can now apply Theorem \ref{thm: main isolation theorem} to the family  $\mathcal W_r$. Given $Z=E_{v}\times E_{v}^{\perp}\in\mathcal Z_r$, consider the level decomposition of $Z = H\times U \times L$ as in Definition \ref{defn:decomposition definable}. By Theorem \ref{thm: main isolation theorem} we have that the family $\mathcal{D}^2$ of sets consisting of:
		
		\begin{itemize}
			\item $U(Z)\times K$ for any $Z\in\mathcal{Z}$ and $K\in\ff(H)\setminus \mathcal K_d(Z)$
			\item $(U(Z)\times K)\cdot CI(Z)$ for any $Z\in\mathcal{Z}$ and $K\in \mathcal K_d(Z)$, where $CI$ is the subgroup whose nontrivial elements consist of all the central involutions of the unique singular $K_{0}\in\mathcal{K}$
		\end{itemize}
		
		is uniformly definable in $\textbf C$. Notice that from the definition of $H$ we have that $K=E_w\in \ff(H)$ has level equal to $r$ and so the set $\mathcal D^2$ is of the form $E_w\times \bar U_w$ where $w$ has level $r$. It follows that the set $\mathcal D$ defined as the set of tuples $\{ U_w, E_w\times \bar U_w, E_w\times E_w^\perp\}$ where $w\in \Gamma^d$ has $d$-level $r$ and if $|E_w|=2$, then there is no $v'\in \Gamma^d$, $w\prec^d v'$ and $E_w<E_{v'}$ is uniformly definable in $\textbf C$.
		
		Then $\A_{r-1}\cup\mathcal{D}$ is precisely the set of tuples $\{ U_v, E_v\times \bar U_v, E_v\times E_v^\perp\}$ where $v\in \Gamma^d$ has $d$-level $s\leq r$ and if $|E_v|=2$, then there is no $v'\in \Gamma^d$, $v\prec^d v'$ and $E_v<E_{v'}$. Hence $\A_{r}=\A_{r-1}\cup\mathcal{D}$ and thus the former is uniformly definable in $\textbf C$, which completes the induction step.

		\begin{remark}\label{rem:two definable families}
			Notice that from {\rm Theorem \ref{thm: main isolation theorem}}, we deduce that not only $\mathcal A$ is definable but in fact the following two families are definable:
			$$
			\mathcal A_c=\{( U_v, E_v\times U_v, E_v\times E_v^\perp)\}_{v\in \tilde{\Gamma}^{d} \setminus \{v\mid E_v\in \mathcal K_d(E_v\times E_v^\perp)\}}
			$$
			and
			$$
			\mathcal A_d=\{( U_v, E_v\times \bar U_v, E_v\times E_v^\perp)\}_{v\in \tilde{\Gamma}^{d} \mid \{E_v\in \mathcal K_d(E_v\times E_v^\perp)\}}.
			$$
		\end{remark}
		
		In order to finish the proof of Theorem \ref{thm: almost interpretation of d-graph}, we need to define the vertex groups of $\Gamma^e$ of order 2 that are dominated by a dihedral subgroup. Recall that we ``lost" these groups in the level by level analysis since in Theorem \ref{p: main definability proposition}, we do not allow to take parameters in vertex groups that have been defined in previous levels. In order to do so, one needs to distinguish the involution that defines a singular vertex groups from other involutions (i.e. involutions that are not non-singular o belong to vertex groups of cardinality greater than 2). We prove this in the following Lemma:
		
		\begin{lemma}\label{lem: definability of Gamma^d}
			The set $\mathcal{B}=\{G_{v}\times U_{v}\,|\,v\in\Gamma^{c},|G_{v}|=2, G_v< E_w,  v \prec^d w, \hbox{ for some } w\in \tl\}$ is uniformly definable in $\textbf C$.
		\end{lemma}
		\begin{subproof}
			Let $DI$ be the set defined as follows: for all $g\in DI$, $g$ is an involution that satisfies the following conditions:
			
			\begin{itemize}
				
				\item (Singular) there exists $A\in \mathcal A$, $A= (U_w, E_w\times \bar U_w, E_w \times E_w^\perp)$ such that $g\in E_w \times \bar U_w$, $g\notin \bar U_w$ and $g\notin I(T(E_w\times E_w^\perp))$, see Lemma \ref{lem: I(t)};
				
				\item (it belongs to dihedrals) for all $A'\in \mathcal A$, $A'= (U_v, E_v\times \bar U_v, E_v \times E_v^\perp)$, if $g\in E_v \times \bar U_v$ and $g\notin U_v$, then either $A'\in \mathcal A_d$ (and so in particular $E_v$ is non-singular dihedral) or $A'\in \mathcal A_c$ and $E_v\times U_v/U_v$ is elementarily equivalent to an infinite dihedral group, see Remark \ref{rem:two definable families}.
			\end{itemize}
			
			Notice that from the definability of the set $\mathcal A$, one deduces the definability of $DI$.
			
			Let $g\in DI$. We show that $g$ is singular and there exists $v\in \Gamma^c$ such that $G_v=\subg{g}$ and $G_v\times U_v \in \mathcal B$.
			
			If $g$ is not singular, then by Lemma \ref{lem: I(t)} we have that $g\in I(T(E_w \times E_w^\perp))$ thus $g$ fails the first condition (singular). Therefore, if $g\in DI$, $g$ is singular.
			
			Suppose that $g\in G_v$ for some $v\in \Gamma^e$ and $|G_v|=2$. Then from the first property, there exists $w \in \tl$ such that $g\in U_w\times E_w$, $g\notin U_w$ and $g\notin I(T(E_w\times E_w^\perp))$. Since $g$ is singular and $g\notin U_w$, we have that $g\in E_w$. From the definition of $I(T(E_w\times E_w^\perp))$, we have that $C(g) < E_w\times E_w^\perp$ and so $v \prec^d w$.
			
			Assume that $g\in G_w$ for some $w\in \Gamma^e$ and $|G_w|>2$. Then from Theorem \ref{p: main definability proposition} we have that $G_w \times G_w^\perp \in \mathcal Z_{l}$ for some $1\leq l\leq n$ and thus $(U_w, U_w \times G_w, G_w\times G_w^\perp) \in \mathcal {A}_c$. Then from the second condition we have that $G_w$ is elementarily equivalent to $D_\infty$. As we have proven above, it follows from the first condition of the set $DI$ that $C(g) < E_w\times E_w^\perp$ and so $v \prec^d w$ for some $w\in \tl$. Furthermore, from the later and the definition of $\Gamma^d$, there exists $v\in \Gamma^c$ such that $G_v=\subg{g}$ and so $G_v < G_w$.
			
			Therefore, if $g\in DI$, we have that $G_v=\subg{g}$ for some $v\in \Gamma^c$ and $G_v\times U_v \in \mathcal B$.
			
			Conversely, if $G_v\times U_v\in \mathcal B$, then $G_v=\subg{g}$ and $g\in DI$.
			
			In order to conclude the proof, we show that the set $\mathcal B$ is definable. For all $g\in DI$, let $C(g)=G_v\times G_v^\perp$. Notice that since $G_v$ is abelian, $C(g) \in \mathcal Z$ and so we can consider the level decomposition of $C(g)$, see Definition \ref{defn:decomposition definable}. Then, from Lemma \ref{l: full isolation} applied to $C(g)=G_v\times G_v^\perp$, we have that the  $G_v\times U(G_v)$ are uniformly definable in $\textbf C$. From this and the definability of the set $DI$ we derive the definability of $\mathcal B$.
		\end{subproof}
	\end{proof}

	\begin{remark}
		Notice that in the previous result the class $\textbf{C}_{\Phi}^{r}$ from {\rm Definition} \ref{dfn: class section 9} can be replaced by $\textbf{C}_{\Phi}$.
	\end{remark}
	
	\begin{cor}
		The group-labelled graph $(\Gamma^d, \{E_v\}_{v\in V\Gamma^d})/\G$ is uniformly definable in $\textbf C$.
	\end{cor}
	\begin{proof}
		In order to interpret the graph $\Gamma^d$ it suffices to take as set of vertices the (different) sets $(U_v, E_v\times U_v, E_v\times E_v^\perp)$ from the family $\mathcal A'$ from Theorem \ref{thm: almost interpretation of d-graph} and as a set of edges the pairs of different vertices $((U_v, E_v\times U_v, E_v\times E_v^\perp), (U_w, E_w\times U_w, E_w\times E_w^\perp))$ such that $E_v$ and $E_w$ commute modulo $U_v\times U_w$, that is $[E_v\times U_v, E_w \times U_w] \subset U_v \times U_w$. Notice that by the properties of $U(Z)$ we have that $E_v$ and $E_w$ commute if and only if for all $K\in \ff(U_v\times E_v)$ and all $K'\in \ff(U_w\times E_w)$ either $K=K'$ or $K$ and $K'$ commute and the later is equivalent to $[E_v\times U_v, E_w \times U_w] \subset U_v \times U_w$, which is interpretable. Finally, to each vertex $(U_v, E_v\times U_v, E_v\times E_v^\perp)$, we associate the group $E_v\times U_v / U_v \simeq E_v$.
	\end{proof}

	\begin{lemma}\label{lem:definability of Gamma^c}
		The family of pairs of sets
		$$
		\mathcal{D}_c=\{(U_{v},E_{v}\times U_{v})\}_{v\in V^{c}}
		$$
		is uniformly definable  in $\textbf C$.
	\end{lemma}
	
	\begin{proof}
		By Theorem \ref{thm: almost interpretation of d-graph} we know that the family of triples of sets:
		$$
		\mathcal{A}'=\{(U_{v},E_{v}\times \bar U_{v},E_{v}\times E_{v}^{\perp})\}_{v\in V^{d}}
		$$
		is uniformly definable in $\textbf C$ as a union of finitely many definable families.
		
		Notice that our assumption that finite vertex groups are uniformly bounded for $\G\in\mc{C}$ implies that the set $V^{d}_{dyh}=\{v\in V^{d}\,|\,E_{v}\cong D_{\infty}\}$ is uniformly definable in $\textbf C$ and so is the set $\mathcal{A}_d'=\{(U_{v},E_{v}\times \bar U_{v},E_{v}\times E_{v}^{\perp})\}_{v\in \Gamma^d_{dyh}}$.
		
		Therefore, $\mathcal D_c = \mathcal A \setminus \mathcal A_d'$ is uniformly definable in $\textbf C$. Notice that for $(U_v, E_v\times \bar U_v, E_v\times E_v^\perp)\in \mathcal D_c$, we have that $U_v=\bar U_v$, see Remark \ref{rem:two definable families}.
	\end{proof}
	
	\begin{lemma}\label{lem: definability of Gamma^e}
		The family of pairs of sets
		$$
		\mathcal{D}=\{(U_{v},G_{v}\times U_{v})\}_{v\in \minCore^e \cup \{v\in \Gamma^c| |E_v|=2\}}
		$$
		is uniformly definable in $\textbf C$.
	\end{lemma}
	\begin{proof}
		\newcommand{\C}[0]{\textbf C}
		
		By Lemma \ref{lem:definability of Gamma^c} we know that the family of tuples of sets:
		$$
		\mathcal{D}_c=\{(U_{v},E_{v}\times U_{v})\}_{v\in V^{c}}
		$$
		is uniformly definable in $\textbf C$ as a union of finitely many definable families.
		
		All we need to show is that there is some formula that on any $\G\in\textbf C$ detects whether a parameter in each of such families corresponds to a triple $\Lambda_{v}=(U_{v},U_{v}\times E_{v})$ where either $v\in V^{e}\subseteq V^{c}$ and $v$ is not weak or $|E_v|=2$. As we will only use the information on the triple $\Lambda_{v}$ to achieve this, the aforementioned parameters can be avoided in the discussion that follows. In fact, by working on $\Th(\G)^{eq}$ one could think of the imaginaries $v\in V^{d}$ (equivalence classes of the disjoint union of the finitely many families of original parameters) as the actual parameters in the formulas giving the families.
		
		To begin with, notice that our assumption that finite vertex groups are uniformly bounded for $\G\in\C$ implies that the set $V^{c}_{cyc}=\{v\in V^{c}\,|\,E_{v}\equiv\Z\}$ is uniformly definable in $\textbf C$.
		
		We say that $v\in V^{c}_{cyc}$ satisfies the orthogonality condition if there is $w\in V^{d}$ such that
		\begin{enumerate}
			\item \label{cond1}$E_{w}\times U_{w}\nleq E_{v}\times E_{v}^{\perp}$
			\item \label{cond2}$E_{v}\times E_{v}^{\perp}\nleq E_{w}\times E_{w}^{\perp}$
			\item \label{cond3}$U_{v} < E_{w}\times E_{w}^{\perp}$
			\item \label{cond4}$E_{v}\times U_{v})^{\perp}\leq E_{w}\times E_{w}^{\perp}$ or equivalently: for any $g\in E_{v}\times E_{v}^{\perp}=(E_{v}\times U_{v})\times(E_{v}\times U_{v})^{\perp}$, there exists $h\in E_{v}\times U_{v}$ such that $gh \in E_{w}\times E_{w}^{\perp}$.
		\end{enumerate}
		Let $V^{c}_{ort}$ the collection of vertices in $V^{c}_{cyc}$ satisfying the orthogonality condition. Since the condition can be expressed in a first order way in terms of $\Lambda_{v}$ and $\Lambda_{w}$ it follows that the family
		$$
		\{(U_{v},E_{v}\times U_{v})\}_{v\in V^{c}\setminus (V^{c}_{ort}\cup V^{c}_{dyh})}
		$$
		is uniformly definable in $\textbf C$.
		
		Let us show that $v\in V^{c}_{cyc}\setminus V_{ort}^c$ if and only if
		$v\in V_{\minCore}^e$, i.e. $E_{v}$ is singular, and $v$ is not weak. Notice that in this case we would deduce that $V^{d}\setminus (V^{c}_{ort}\cup V^{c}_{dyh})=V^{e}_{\minCore}\cup \{ v\in \Gamma^c\mid |E_v|=2\}$.
		
		We start by noting that condition \ref{cond4} for a witness of the orthogonality condition can be replaced by the following two:
		\begin{enumerate}[label=(\alph*)]
			\item \label{condd1}$E_{w}\nleq E_{v}\times E_{v}^{\perp}$; and
			
			\item \label{condd2}$E_{v}^{\perp}\leq E_{w}^{\perp}$, or equivalently $E_{v}^{\perp}\perp E_{w}$.
		\end{enumerate}
		
		Conditions \ref{condd1} and \ref{condd2} imply the orthogonality condition. Indeed, Condition \ref{condd1} trivially implies  Condition  \ref{cond1}, while \ref{condd2} trivially implies \ref{cond3} and \ref{cond4}. On the other hand, if Condition \ref{cond2} does not hold, then \ref{condd2} implies that either $E_{v}\subseteq E_{w}$, which implies $E_{v}=E_{w}$, contradicting \ref{condd1} or $E_{v}\leq E_{w}^{\perp}$ or equivalently $E_{w}\leq E_{v}^{\perp}$ which also contradicts \ref{condd1}.
		
		For the opposite implication take $w$ witnessing the orthogonality condition for $v$. If condition \ref{condd1} fails, then as $|\ff(E_{w})|=1$ either $E_{w}\leq E_{v}$, which implies $E_{v}=E_{w}$ contradicting \ref{cond1}, or $E_{w}\leq E_{v}^{\perp}$. In the latter case there are two options. One is that $E_{w}\leq U_{v}$ and thus $E_{w}\times U_{w}\leq U_{v}$ by the transitivity of $\prec^{d}$, contradicting \ref{cond1}. The other is that $E_{w}\leq (E_{v}\times U_{v})^{\perp}$ which implies $E_{v}\times U_{v}\leq E_{w}^{\perp}$ which together with condition \ref{cond4} yields
		$$
		E_{v}\times E_{v}^{\perp}=(E_{v}\times U_{v})\times (E_{v}\times U_{v})^{\perp}\leq E_{w}\times E_{w}^{\perp}
		$$
		contradicting \ref{cond2}.
		Finally, \ref{cond3} and \ref{cond4} clearly imply \ref{condd2}.

		Assume $E_{v}$ is non-singular cyclic; by definition $\ff(E_{v})=\hh$ for some subgraph product $\hh$.
		Take any $w\in V^{e}$ such that $G_{w}=E_{w}\leq \hh$. Since $g\in G_{w}\setminus 1$  does not commute with
		$E_{v}$, clearly $G_{w}\nin E_{v}\times E_{v}^{\perp}$. However $E_{v}^{\perp}=\hh^{\perp}\leq G_{w}^{\perp}$. It follows that in this case $v\in V_{ort}^c$.
		
		Likewise, if $v\in V^{e}$ is weak, then by definition there is $w\in V^{e}$ such that $v,w\nin E^{e}$, but
		$\link{v}\subseteq \link{w}$, from which it follows that $v$ satisfies the orthogonality condition, that is $v\in V_{ort}^c$.
		
		To conclude, assume now for the sake of contradiction that some non weak $v\in V^{e}\cap V^{c}_{cyc}$ satisfies the orthogonality condition. Let $w\in\Gamma^{c}$ be a witness for this. From Condition \ref{condd1}, we have that $E_w\nleq E_v\times E_v^\perp$ and so $(v,w)\notin E(\Gamma)$. From Condition \ref{condd2}, $E_v^\perp < E_W^\perp$ and so $\st(v)\cap \st(w)=E_v^\perp=\link{v}$ and so $v$ is a weak vertex.
	\end{proof}
	
	\begin{cor}\label{cor:main interpretation}
		Let $\mc{G}=\mc{G}(\Gamma, \gvv)\in \mc{G}(\textbf{K}_N, \textbf{C})$. Then the group-labelled graphs $\minCore(\Gamma, \gvv)$ and $\ECore(\Gamma, \gvv)$ are interpretable in $\mathcal G$. This interpretation is uniform in the class $\mc{G}(\textbf{K}_N, \textbf{C})$.
	\end{cor}
	
	\begin{proof}
		
		Consider the (different) sets from the family $\mathcal D_c$ from Lemma \ref{lem:definability of Gamma^c}. Define the following equivalence relation:
		
		$$
		(U_v, E_v \times U_v) \sim (U_w, E_w \times U_w)
		$$
		
		if and only if there exists $g\in \mathcal G$ such that $(E_v \times U_v)^g = E_w \times U_w$. We denote the equivalence class as $[(U_v, E_v \times U_v)]$. Since conjugation respects the preorder $\prec^d$ and for all $v'\in \Gamma^e$ such that $G_{v'}<U_v$ we have that $v \prec^d v'$, we have that $(E_v \times U_v)^g = E_w \times U_w$ if and only if $E_v^g=E_w$.
		
		We set the set of vertices of the graph $\Lambda$ to be the different definable sets from $\mathcal D_c$ modulo the equivalence relation, that is $\{ [(U_v, E_v \times U_v)]\}$.
		
		The set of edges are defined as follows: $([(U_v, E_v \times U_v)], [(U_w, E_w \times U_w)]$ is an edge if and only if there is $(U_{v'}, E_{v'} \times U_{v'})\in [(U_v, E_v \times U_v)]$ and $(U_{w'}, E_{w'} \times U_{w'})\in [(U_w, E_w \times U_w)]$ such that $[E_{v'} \times U_{v'},E_{w'}\times U_{w'}] < U_{v'}\cdot U_{w'}$. Notice that the latter implies that $[E_{v'}, E_{w'}]=1$ and so from the definition of $U_{v'}$ and $U_{w'}$ we have that each factor of $U_{v'}$ and each factor of $U_{w'}$ are coincide or commute. In particular, the subgroup generated by $U_{v'}$ and $U_{w'}$ is a direct sum of singular vertex groups.
		
		Furthermore, for each vertex $[(U_v, E_v \times U_v)]$, we associate the group $E_v \times U_v / U_v$ if $E_v \times U_v / U_v\not\equiv \mathbb Z, D_\infty$; the group $\mathbb Z$ if $E_v \times U_v / U_v\equiv \mathbb Z$ and the group $D_\infty$ if $E_v \times U_v / U_v\equiv D_\infty$. Notice that for each $v\in \Lambda$ the vertex group is uniformly interpretable as so they are $E_v\times U_v$ and $U_v$.
		
		This proves that the group-labelled graph $(\Lambda, \{E_v\}_{v\in \Lambda})$ is uniformly interpretable. Let $(\Lambda_0, \{E_v\}_{v\in \Lambda_0})$ be the induced group-labelled subgraph of $(\Lambda, \{E_v\}_{v\in \Lambda})$ spanned by the vertices such that $E_v\not\equiv D_\infty$. Notice that in the class of groups $\textbf C_{\Phi}^r$, the latter condition is first-order definable. Then $(\Lambda_0, \{E_v\}_{v\in \Lambda_0})$ is uniformly interpretable.
		
		By Lemma \ref{lem: definability of Gamma^e}, the family $\mathcal D$ is uniformly definable in $\mathcal D_c$, and so the group-labelled graph $\minCore(\Gamma, \gvv)$ is uniformly interpretable in $(\Lambda, \{E_v\}_{v\in \Lambda_0})$: the vertices of $\minCore(\Gamma)$ correspond to vertices $[(U_v, E_v \times U_v)]$ such that $(U_v, E_v \times U_v) \in \mathcal D$.
		
		Finally, since the relation $\sim_2$ is defined in terms of links of vertices in $\Lambda_0$ and $\Lambda_0$ is uniformly interpretable, so is $\sim_2$. Furthermore, since the existence of a witness in $\minCore(\Gamma, \gvv)$ for a weak vertex is also definable in first-order, we have that the group-labelled graph $I(\Gamma, \gvv)$ and so the extended core $\ECore(\Gamma, \gvv)$ are uniformly interpretable in $\mathcal G$ inside the class $\textbf C$.
	\end{proof}
	
	\section{Applications}
	
	\subsection{Elementary equivalence of graph products of groups}
	
	In Section \ref{s: interpretation_of_base_structure_(first)} we proved the main interpretability result, namely, we showed that  given a graph product $\mc{G}=\mc{G}(\Gamma, \gvv)\in \mc{G}(\textbf{K}_N, \textbf{C})$, then the group-labelled graphs $\ECore(\Gamma,\gvv)$ and $\Core(\Gamma,\gvv)$  are uniformly interpretable in some given class $\textbf C$ of graph products, see Corollary \ref{cor:main interpretation}. The next theorem is an immediate consequence of this result:
	
	\begin{thm}\label{thm:elementary equivalence}
		Let $\mc{G}(\Gamma, \gvv)$ and $\mc{G}(\Delta,\{G_v'\}_{v\in \Delta})$ be two elementarily equivalent graph products of groups, where $\Gamma$ and $\Delta$ are finite and either:
		\begin{itemize}
			\item $\Gamma$,$\Delta$ are almost positive reduced graphs and each $G_{v}$ and $H_{u}$ satisfies a simple non-generic almost positive sentence
			\item $\Gamma$,$\Delta$ are positive reduced graphs and each $G_{v}$ and $G_{v}'$ satisfies a nontrivial positive sentence
		\end{itemize}
		Then there exists a graph isomorphism $f: \ECore(\Gamma,\gvv) \to \ECore(\Delta,\hww)$ such that $G_v$ is elementarily equivalent to $G_{f(v)}'$ for all $v\in \ECore(\Gamma)$ restricting to an isomorphism between representatives of $\Core(\Gamma,\gvv)$ and $\Core(\Delta,\{G_v'\}_{v\in \Delta})$ (as group-labelled graph ).
	\end{thm}
	\begin{proof}
		Finitness of the graphs implies that both graph products belong to some common class $\G(\textbf{K}_{N},\textbf{C}^{r}_{\Phi})$. The result then follows from Corollary \ref{cor:main interpretation}. 
	\end{proof}
	
	\begin{remark}
		We conjecture that the conditions on the graphs, namely that they be {\rm(}almost{\rm)} positive reduced are necessary.
		
		Let $\Gamma$ be a reduced graph that is not almost positive reduced and let $\Gamma'$ be an almost positive reduced quotient of $\Gamma$ obtained from $\Gamma$ by collapsing $w \cup S$ to a vertex $w'$, where $S$ is a graph with more than one vertex and it is not a join.
		
		Let $\mc{G}=\mc{G}(\Gamma, \gvv)$, where $G_v\simeq \mathbb{Z}$.
		
		Define $\mc{G}'=\mc{G}(\Gamma', \{H_v'\})$ where $H_v'=G_v'$ for all $v \notin S\cup \{w\}$ and  $H_{w'}'= \langle G_w'\rangle \times \langle \mathcal G(S), x,y \mid [x,y]=[w_1,w_2]\rangle$, where $w_i$ are $N$-small cancellation. Note that since $S$ is not a join (and it has more than one vertex) $\mathcal G(S)$ is a directly indecomposable RAAG and so it contains small cancellation elements, see, for instance, {\rm Corollary \ref{c: small cancellation}}.
		
		On the one hand, we conjecture that $\mc{G} \equiv \mc{G}'$. On the other hand, $H_{w'}'$ decomposes as a direct product $G_w' \times \langle \mathcal G(S),x,y \mid [x,y]=[w_1,w_2]\rangle$ and $\langle \mathcal G(S),x,y \mid [x,y]=[w_1,w_2]\rangle$ is directly indecomposable and does not admit a nontrivial decomposition as a graph product. Therefore, one cannot expect to interpret the core of $\Gamma$ in general terms.
	\end{remark}
	
	Combining the results above we deduce the following:
	
	\begin{conj}
		We conjecture that the converse also holds:
		
		Let $\mc{G}(\Gamma, \gvv)$ and $\mc{G}(\Delta, \{G_v'\}_{v\in \Delta})$ be graph products such that $\Gamma$ and $\Delta$ are finite and $G_v$ and $G_v'$ satisfy a simple non-generic almost positive sentence. 
		
		If there is an isomorphism $f:\Core(\Gamma) \to \Core(\Delta)$ and $G_v\equiv G_{f(v)}'$, then $\mc{G}(\Gamma, \gvv) \equiv \mc{G}(\Delta, \{G_v'\}_{v\in \Delta})$.
		
		In particular, $\mathbb{G}(\Gamma, \gvv)$ is elementarily equivalent to $\mathbb G(\Core(\Gamma, \gvv))$.
	\end{conj}
	
	Notice that if the conjecture holds, then one would deduce that $\mathbb G(\Core(\Gamma, \gvv))$ is the smallest graph product {\rm(}with vertex groups satisfying a simple non-generic positive sentence{\rm)} in the elementary class.
	
	\subsubsection{Further conditions of the vertex groups}
	
	Consider the class $\mathcal{UDF}$ of groups which have unique direct product decomposition: if $G \simeq G_1 \times \cdots \times G_k \simeq H_1 \times\dots\times H_m$, then $k = m$ and $G_i \simeq H_i$ (after a possible re-ordering of factors).
	
	This allows us to ``reverse'' the process of going to the (almost) positive reduced graph and recover the extended core of a maximal graph $\Gamma$ for which $\mathcal G$ is a graph product, that is, we define the graph $\Gamma$ and associated vertex groups as follows: if $\Delta$ is an (almost) positive reduced graph, for each vertex $v\in \Delta$, we consider the decomposition of $G_v$ as direct product of $k_v$ factors $G_v=G_{v_1} \times \dots \times G_{v_{k_v}}$, the set of vertices of $\Gamma$ is defined as the union $\bigcup\limits_{v\in \Delta} \{v_1, \dots, v_{k_v}\}$; the vertices $v_1, \dots, k_v$ are a clique in $\Gamma$ and there is an edge $(v_i, w)$ in $\Gamma$ if and only if there is an edge $(v,w)$ in $\Delta$; we assign to each vertex $v_i$ in $\Gamma$ the group $G_{v_{i}}$ $i=1,\dots, k_v$. It is easy to check that the map $\phi:\Gamma \to \Delta$ that sends $v_1, \dots, v_{k_v}$ to $v$ is a full epimorphism and so the corresponding graph products are isomorphic, see Lemma \ref{lem: full is iso}. Therefore, we have the following
	
	\begin{lemma}
		Let $\mathcal G=\mc{G}(\Delta, \{G_v'\}_{v\in \Delta})$ be a graph product where $G_v' \in \mathcal{UDF}$. Then $\mathcal G=\mc{G}(\Gamma, \{D_v\}_{v\in \Gamma})$ is a graph product where $D_v$ are directly indecomposable. Furthermore, $\Gamma$ is maximal and unique with the above property.
	\end{lemma}
	
	In this context, we obtain the following result:
	
	\begin{cor}\label{cor:iso graph}
		
		Let $\mathcal G=\mc{G}(\Gamma, \{D_v\})$ be a graph product where $\Gamma$ is finite, $D_v$ satisfies a simple non-generic almost positive sentence and $D_v\in \mathcal{UDF}$ is directly indecomposable. Then $\Core(\Gamma)$ is interpretable in $\mathcal G$ and if a graph product $\mathcal G'=\mc{G}(\Delta', \gvv)$ in the class is elementarily equivalent to $\mathcal G$, then $\mathcal G'=\mc{G}(\Gamma; D_v')$, there is an isomorphism  $\phi: \Core(\Gamma) \to \Core(\Gamma')$ and $D_v\equiv D_{\phi(v)}'$.
	\end{cor}
	
	\subsubsection{Applications to RAAGs}
	
	\def\GG{\mathbb G}
	
	Let $\GG=\GG(\Gamma)$ be a RAAG. Then $\GG$ is a graph product with infinite cyclic vertex groups and so vertex groups satisfy a nontrivial positive sentence. Let $\tilde \Gamma$ be a positive reduced graph, then $\GG=\mathcal G(\tilde \Gamma, \gvv)$ where $G_v$ are free abelian groups. Notice that free abelian groups belong to $\mathcal{UDF}$. Therefore, from Corollary \ref{cor:iso graph} applied to RAAGs we have:
	
	\begin{cor}
		Let $\GG(\Gamma)$ and $\GG(\Delta) $ be RAAGs. Then, if $\GG(\Gamma) \equiv \GG(\Delta)$, then $\Core(\Gamma)\simeq \Core(\Delta)$.
	\end{cor}
	
	\begin{conj}
		We conjecture that the converse also holds:
		
		If $\Core(\Gamma) \simeq \Core(\Delta)$, then $\GG(\Gamma)\equiv \GG(\Delta)$.
		
		Furhtermore, $\Core(\Gamma)$ is the RAAG with the smallest graph in its elementary class, that is if $\mathbb G(\Delta) \equiv \mathbb G(\Core(\Gamma))$, then $\Core(\Gamma) < \Delta$.
		
		If $\mathbb G(\Gamma) \equiv \mathbb G(\Delta)$, then $\Delta$ is obtained from $\Core(\Gamma)$ by adding redundant vertices.
	\end{conj}
	
	The above conjecture would state the following concrete results.
	
	\begin{example}
		\,
		\begin{itemize}
			\item Since the core of a non-abelian free group is the edgeless graph with 2 vertices, the conjecture would state that any two non-abelian free groups are elementarily equivalent. This is indeed the case as proven by Sela and Kharlampovich-Miasnikov.
			
			\item If $\mathbb G(\Gamma)$ is a RAAG with a nontrivial free product decomposition, then the conjecture asserts that $\mathbb G(\Gamma) \equiv \mathbb G(\Gamma) \ast F_k$, where $F_k$ is a free group. This result follows from Sela's work on the elementary theory of free products.
			
			\item More generally, if $\Gamma$ is a path with diameter greater than 2, then conjecture states that $\mathbb G(\Gamma) \equiv \mathbb G(\Gamma) \ast F_k$, where $F_k$ is a free group.
			
			\item Let $T$ be a tree, $\minCore(T)$ is the spanned by the non-leave vertices of $T$ and the core $T_0=\Core(\Gamma)$ is obtained from $\minCore(\Gamma)$ by adding precisely one leaf at each leaf of $\minCore(\Gamma)$ or 2 leaves if $\minCore(\Gamma)$ is one vertex {\rm(}that is if $T$ has diameter 2{\rm)}.
			By the conjecture, if $\mathbb G(\Delta) \equiv \mathbb G(T)$, then $\Delta$ is obtained from $T_0$ by adding leaves to any vertex from $\minCore(T)$ and adding isolated vertices.
			
			\item More generally, if $\mathbb G(\Gamma)$ is 2-dimensional, that is, $\Gamma$ does not contain triangles, $\Gamma$ coincides with its core and has diameter more than 2, then the conjecture states that $\mathbb G(\Delta) \equiv \mathbb G(\Gamma)$ if and only if $\Delta$ is obtained from $\Gamma$ by adding leaves in vertices of $\minCore(\Gamma)$ and isolated vertices.
		\end{itemize}
	\end{example}
	
	We next observe that having the same $\Core$ for RAAGs is sufficient to assure that they have the same universal theory.
	
	\begin{lemma} \label{lem:univequiv}
		If the cores of $\Gamma$ and $\Delta$ are isomorphic, $\Core(\Gamma)\simeq \Core(\Delta)$, then $\mathbb{G}(\Gamma)$ and $\mathbb{G}(\Delta)$ are universally equivalent.
	\end{lemma}
	\begin{proof}
		
		Notice that if $\Gamma < \Delta^e$ and $\Delta < \Gamma^e$, then there is a tame morphism from $\GG(\Gamma) \to \GG(\Delta)$ and from $\GG(\Delta) \to \GG(\Gamma)$ and so from \cite[Theorem 4.13]{2015casals}, we have that $\GG(\Gamma) \equiv_\forall \GG(\Delta)$.
		
		In order to establish the lemma, it suffices to show that $\Gamma$ embeds into the extension graph $\Core(\Gamma)^e$ of $\Core(\Gamma)$. Indeed, as then $\Gamma < \Core(\Gamma)^e \simeq \Core(\Delta)^e < \Delta^e$ and similarly, $\Delta < \Gamma^e$.

		Let $v_1, \dots, v_l$ be all the redundant vertices of  $\Gamma$ and let $W_i=\{w_{i,j}$, $i=1,\dots, l$, $j=1,\dots, k_l\}$ be the vertices of $\Core(\Gamma)$ so that $w_i\notin\st(v_i)$ and $\cap_{j=1,\dots,k_i}\st(w_{i,j}) = \link{v_i}$, $i=1,\dots, l$, that is $W_i$ is a witness for $v_i$.
		
		We identify vertices and generators of the RAAG. Notice that $w_{i1}, w_{i2}\dots w_{ik_l}$ do not commute as otherwise $w_{i1}\in \lk(v_i)$. Therefore, as in a RAAG two non-commuting elements generate a free group, they have an irreducible action in a tree. Using for instance Corollary \ref{c: small cancellation}, we can find words $u_i$ such that $\ff(w_{i1}^{u_i})=\{\mathcal G(w_{i1}, \dots, w_{ik_i})\}$ and $w_{i1}^{u_i} \ne w_{j1}^{u_j}$ for all $i\ne j \in \{1, \dots, l\}$. Consider the subgraph of $\Core(\Gamma)^e$ spanned by $\Core(\Gamma)$ and $w_{i1}^{u_i}$. It follows from the property of $u_i$ that there is an edge of the form $(u,w_{i1}^{u_i}) \in E(\Gamma^e)$, where $u$ is a vertex of $\Core(\Gamma)$ if and only if there is an edge $(u,v_i)$ in $\Gamma$. Therefore, $\Gamma < \Core(\Gamma)^e$.
	\end{proof}

	\subsection{Rigidity results}
	
	The main consequence of our results is the rigidity of some classes of groups up to elementary equivalence and isomorphism.
	
	\subsubsection{Isomorphism rigidity of graph products}
	
	We now establish a result on isomorphic rigidity for the appropriate class of graph products.

	\begin{cor}\label{cor:rigidity isomoprhism}
		Let $\Gamma, \Delta$ be two finite {\rm(}almost{\rm)} positive reduced graphs and $G_v, G_v'$ be groups with nontrivial {\rm(}almost{\rm)} positive theory. Then, $\mathcal G(\Gamma, \gvv)\simeq \mathcal G(\Delta, \{G_v'\}_{v\in \Delta})$ if and only if there is a graph isomorphism $f:\Gamma \to \Delta$ and $G_v\simeq G_{f(v)}'$.
	\end{cor}
	\begin{proof}
		Let $\mc{G}=\mathcal G(\Gamma, \gvv)$ and $\mc{G}'=\mathcal G(\Delta, \{G_v'\}_{v\in \Delta})$ be graph products and let $f:\mc{G} \to \mc{G}'$ be an isomorphism. In particular, we have that $\mc{G} \equiv \mc{G}'$. From the uniform definability of $\minCore(\Gamma^e)$ and $\Gamma^c$, see Lemma \ref{lem: definability of Gamma^e} and Lemma \ref{lem:definability of Gamma^c}, we have that $f$ induces isomorphisms $f:\minCore(\Gamma^e)\to \minCore(\Delta^e)$ and $f:\Gamma^e\to \Delta^e$ and more precisely, for all $v\in \Gamma^e$, we have that $f(G_v)=G_{v'}'$ for some $v'\in \Delta^e$ and for all $v\in \Gamma^c$, $f(E_v)=E_{v'}$ for some $v'\in \Delta^c$.
		
		So, up to composing by an inner automorphism if necessary, we can assume without loss of generality that $f$ induces an isomorphism $f:\minCore(\Gamma)\to \minCore(\Delta)$ and $f: \Lambda(\Gamma) \to \Lambda(\Delta)$, see Definition \ref{defn: graph I(Gamma)}.
		
		Let $v\in \Gamma\setminus \minCore(\Gamma)$ be a vertex with maximal link among the vertices in $\Gamma\setminus \minCore(\Gamma))$, that is, for all $w\in \Gamma \setminus \Core(\Gamma)$ we have that $\link{v} \not\subset \link{w}$. Consider $[v]$, that is the set of vertices of $\Lambda(\Gamma)$ with the same link as $\link{v}$. Let $\Gamma_{[v]}=\Gamma \cap [v]$ and let $\Gamma_{>[v]}=\{ w\in \Gamma \mid \link{v} \subsetneq \st(w)\}$. Notice that $\Gamma_{[v]}$ is an edgeless induced subgraph of $\Gamma$ generated by weak vertices with infinite cyclic associated groups, that is $\mc{G}(\Gamma_{[v]})$ is a free group. Indeed if $\link{v}=\link{w}$, then $(v,w)\notin E(\Gamma)$ (otherwise $w\in \link{v}=\link{w}$ a contradiction). From the maximality of $\link{v}$, we have that if $w\in \Gamma_{>[v]}$, then $v\in \minCore(\Gamma)$.
		
		We claim that for all $v_1\in [v]$, there is a conjugate of $E_{v_1}$ such that $E_{v_1} < \mc{G}(\Gamma_{[v]} \cup \Gamma_{>[v]})$.
		
		If $v_1\in [v]$ and $v_1\in \Gamma^e$, that is if $E_{v_1}$ is singular, then up to conjugacy, $v_1$ is conjugate to a vertex in $\Gamma_{[v]}$ (more precisely, there exists $g\in \mc{G}$ such that $E_{v_1}^g=E_{v_2}$ and $v_2\in \Gamma_{[v]}$).
		
		Assume now that $E_{v_1}$ is a non-singular cyclic subgroup. Then there exists $g\in \mc{G}$ such that $E_{v_1}^g < \mc{G}(\Gamma_1)$, $\Gamma_1<\Gamma$ is smallest subgraph with this property. As by definition, $\link{v}=\link{v_1}$, we have that $\link{v_1} = \Gamma_1^\perp$ and so $\Gamma_1 < \Gamma_{[v]} \cup \Gamma_{>[v]}$ and $E_{v_1}< \mc{G}(\Gamma_{[v]} \cup \Gamma_{>[v]})$.
		
		As $f$ induces isomorphisms $\Gamma^c \to \Delta^c$ and $\minCore(\Gamma^e)\to \minCore(\Delta^e)$, we have that $f([v])=[f(v)]$ and so $f(\mc{G}(\Gamma_{[v]} \cup \Gamma_{>[v]})) < \mc{G}(\Delta_{[f(v)]} \cup \Delta_{>[f(v)]})$. Symmetrically, we have that $f^{-1}(\mc{G}(\Delta_{[f(v)]} \cup \Gamma_{>[f(v)]}))<\mc{G}(\Gamma_{[v]} \cup \Gamma_{>[v]})$. Since $f$ is an isomorphism, $f^{-1}(f(\mc{G}(\Gamma_{[v]} \cup \Gamma_{>[v]})))=\mc{G}(\Gamma_{[v]} \cup \Gamma_{>[v]})$ and so we deduce that $f(\mc{G}(\Gamma_{[v]} \cup \Gamma_{>[v]})) = \mc{G}(\Delta_{[f(v)]} \cup \Delta_{>[f(v)]})$. Furthermore, since $\Gamma_{>[v]}\in \minCore(\Gamma)$, we have that $f(\Gamma_{>[v]})=\Delta_{>[f(v)]}$. Therefore, we have that:
		$$
		\mc{G}(\Gamma_{[v]}\cup \Gamma_{>[v]})/ \mc{G}(\Gamma_{>[v]}) \simeq \mc{G}(\Gamma_{[v]})\simeq \mc{G}(\Delta_{[f(v)]}) \simeq \mc{G}(\Delta_{[f(v)]}\cup \Delta_{>[f(v)})/ \mc{G}(\Delta_{>[f(v)}).
		$$
		
		As $\mc{G}(\Gamma_{[v]})$ and $\mc{G}(\Delta_{[f(v)]})$ are isomorphic free groups, they have the same rank. Then, the isomorphism $f:\minCore(\Gamma)\to \minCore(\Delta)$ can be extended to an isomorphism $f_{[v]}: \minCore(\Gamma) \cup \Gamma_{[v]} \to \minCore(\Delta) \cup \Delta_{[f(v)]}$. In fact, let $\Gamma_1 < \Gamma$ be the subset of weak vertices $[v]$ with infinite cyclic subgroups and maximal link and let $\Delta_1< \Delta$ be defined analogously. Then, the isomorphism $f:\minCore(\Gamma)\to \minCore(\Delta)$ can be extended to an isomorphism $f_1: \minCore(\Gamma) \cup \bigcup\limits_{v\in \Gamma_1} \Gamma_{[v]} \to \minCore(\Delta) \cup \bigcup\limits_{v\in \Delta_1} \Delta_{[f(v)]}$. We remark that the isomorphism $f_1$ is not induced from the isomorphism $f$.
		
		We proceed by induction on the height $m$ of the links of vertices in $\Gamma\setminus \minCore(\Gamma)$. Let $\Gamma_m < \Gamma$ be the subset of weak vertices $[v]$ with infinite cyclic subgroups and link of height less than or equal to $m$ and let $\Delta_m< \Delta$ be defined analogously. Assume by induction that there is an isomorphism $f_m$ that extends $f$ from $f_m: \minCore(\Gamma) \cup \bigcup\limits_{v\in \Gamma_m} \Gamma_{[v]} \to \minCore(\Delta) \cup \bigcup\limits_{v\in \Delta_m} \Delta_{[f(v)]}$.
		
		The induction step is analogous to the base of induction. Let $v\in \Gamma\setminus \minCore(\Gamma)$ be a vertex whose link has height $m+1$. Then $f([v])=[v']$ where $v'\in \Lambda(\Delta)$ and whose link also has height $m+1$.
		
		As above, for all $v_1\in [v]$, there is a conjugate of $E_{v_1}$ such that $E_{v_1} < \mc{G}(\Gamma_{[v]} \cup \Gamma_{>[v]})$ and we deduce that $f(\mc{G}(\Gamma_{[v]} \cup \Gamma_{>[v]})) = \mc{G}(\Delta_{[f(v)]} \cup \Delta_{>[f(v)]})$. Since $\Gamma_{>[v]}\in \minCore(\Gamma) \cup \bigcup\limits_{v\in \Gamma_m} \Gamma_{[v]}$, by induction we have that $f(\Gamma_{>[v]})=\Delta_{>[f(v)]}$. Therefore, we have that:
		$$
		\mc{G}(\Gamma_{[v]}\cup \Gamma_{>[v]})/ \mc{G}(\Gamma_{>[v]}) \simeq \mc{G}(\Gamma_{[v]})\simeq \mc{G}(\Delta_{[f(v)]}) \simeq \mc{G}(\Delta_{[f(v)]}\cup \Delta_{>[f(v)})/ \mc{G}(\Delta_{>[f(v)}).
		$$
		
		As $\mc{G}(\Gamma_{[v]})$ and $\mc{G}(\Delta_{[f(v)]})$ are isomorphic free groups, they have the same rank. Let $\Gamma_{m_1} < \Gamma$ be the subset of weak vertices $[v]$ with infinite cyclic subgroups and link of height $m+1$ and let $\Delta_{m+1}< \Delta$ be defined analogously. Then, the isomorphism $f:\minCore(\Gamma)\to \minCore(\Delta)$ can be extended to an isomorphism $f_{m+1}: \minCore(\Gamma) \cup \bigcup\limits_{v\in \Gamma_{m+1}} \Gamma_{[v]} \to \minCore(\Delta) \cup \bigcup\limits_{v\in \Delta_{m+1}} \Delta_{[f(v)]}$.
		
		Since the graphs in the class have property $AP_n$ (in particular if they are finite), we deduce that $f$ extends to an isomorphism of graphs $f_n: \Gamma \to \Delta$. Since for all $v\in \minCore(\Gamma)$ we have that $G_{v}\simeq G_{f(v)}'$ and for all $v\in \Gamma\setminus \minCore(\Gamma)$ we have that $G_v\simeq \mathbb Z \simeq G_{f_n(v)}'$, we have that $\mc{G}\simeq \mc{G}'$.
	\end{proof}
	
	This theorem recovers and generalises classical results on the classification of families of groups up to isomorphism. More precisely, it generalises the classification of RAAGs, RACGs and graph products of finite groups and it applies to graph products of groups with non-generic almost positive theory, for instance, graph products of nilpotent, solvable groups, groups satisfying an identity, groups with nontrivial center etc.
	\begin{cor}\
		\begin{itemize}
			\item Let $\GG(\Gamma)$ and $\GG(\Delta)$ be two RAAGs. Then $\GG(\Gamma)\simeq \GG(\Delta)$ if and only if $\Gamma\simeq \Delta$; in fact, if $\mc{G}$ and $\mc{G}'$ are graph products of groups in $\textbf C$ and in $\mathcal{UDF}$, then $\mc{G}=\mc{G}(\Gamma, \gvv)$ and $\mc{G}'=\mc{G}(\Delta, \{G_{v}'\}_{v\in \Delta})$ where $G_v$ and $G_{v}'$ are directly indecomposable, then $\mc{G}\simeq \mc{G'}$ if and only if there is a graph isomorphism $f:\Gamma \to \Delta$ and $G_v\simeq G_{f(v)}'$.
			
			\item Let $\mc{G}$ and $\mc{G}'$ be two graph products of finite groups. Then $\mc{G}=\mathcal G(\Gamma, \gvv)$, $\mc{G}'=\mathcal G(\Delta, \{G_{v}\}_{v\in \Delta}')$ where $\Gamma$ and $\Delta$ are reduced and $\mc{G}\simeq \mc{G}'$ if and only if there exists an isomorphism $f: \Gamma \to \Delta$ and $G_v \simeq G_{f(v)}'$.
			
			\item Let $\mc{G}$ and $\mc{G}'$ be two graph products of groups with nontrivial positive theory. Then $\mc{G}=\mathcal G(\Gamma, \gvv)$, $\mc{G}'=\mathcal G(\Delta, \{G_{v}\}_{v\in \Delta}')$, where $\Gamma$ and $\Delta$ are reduced and $\mc{G}\simeq \mc{G}'$ if and only if there exists an isomorphism $f: \Gamma \to \Delta$ and $G_v \simeq G_{f(v)}'$.
		\end{itemize}
	\end{cor}
	
	\begin{proof}
		Let $\GG(\Gamma)$ and $\GG(\Delta)$ be two isomorphic RAAGs. From Lemma \ref{lem: graph is positive reduced}, we have that $\GG(\Gamma)\simeq \mc{G}(\Gamma', \gvv)$ and $\GG(\Delta)\simeq \mc{G}(\Delta', \{G_{v'}'\})$ where $\Gamma'$ and $\Delta'$ are positive reduced and by construction, $G_v$ and $G_{v'}'$ are free abelian groups. From Corollary \ref{cor:rigidity isomoprhism}, we have that there exists a graph isomorphism $f:\Gamma' \to \Delta'$ and $G_v\simeq G_{f(v)}'$. The full morphism $\pi:\Gamma \to \Gamma'$ (resp. $\pi':\Delta\to \Delta'$) collapses cliques, that is, for all $v\in \Gamma'$, $\pi^{-1}(v)$ is a clique and the size of the clique is precisely the rank of the free abelian groups $G_v$. Since $G_v\simeq G_{f(v)}$ and they are free abelian groups, they have the same rank and so $\pi^{-1}(v)$ and ${\pi'}^{-1}(f(v))$ are cliques of the same dimension. This together the fact that $\Gamma'$ and $\Delta'$ are isomorphic imply that $\Gamma$ and $\Delta$ are also isomorphic.
		
		The proof for graph products of groups with groups in $\textbf C \cap \mathcal{UDF}$ is analogous.
	\end{proof}
	
	\subsubsection{Elementary rigidity of graph products}
	
	Our results applied to specific families of groups gives us rigidity of the elementary classification. Let $\textbf C$ be the class of group defined in Definition \ref{dfn: class section 9}.
	
	\begin{thm}[cf. \cite{CKR}, Rigidity of the class of graph products of finite groups]
		\
		
		Let $\textbf C_f$ be the class of finite groups. Let $G=\mathcal G(\Gamma, \gvv)$ and $H=\mathcal G (\Delta, \{H_v\}_{v\in \Delta})$ where $G_v \in \textbf C_f$, $ H_w \in \textbf C$. Then
		$$
		G \equiv H \Leftrightarrow G\simeq H.
		$$
		In particular, two right-angled Coxeter groups are elementarily equivalent if and only if they are isomorphic.
	\end{thm}

	Combining our results with the elementary classification of finitely generated nilpotent ($R$)-groups, see {\rm\cite{Oger, MiaSoh, CFKR}}, we deduce the following
	
	\begin{thm}[Rigidity of the class of graph products of (non-cyclic) fg nilpotent (well-structured nilpotent $R$-) groups]\
		
		Let $\textbf C_n$ be the class of {\rm(}non-cyclic{\rm)} finitely generated nilpotent groups. Let $G=\mathcal G(\Gamma, \gvv)$ and $H=\mathcal G (\Delta, \{H_v\}_{v\in \Delta}$, where $G_v\in \textbf C_n$ and $H_w \in \textbf C$. Then
		
		\begin{gather} 
			\begin{split}
				G \equiv H \Leftrightarrow & f:\Gamma \to \Delta \hbox{ is an isomorphism and } \\
				& G_v \times \mathbb Z \simeq H_{f(v)} \times \mathbb Z.
			\end{split}
		\end{gather} 
		
		\medskip 
		
		More generally, let $\textbf C_{n}(R)$ be the class of {\rm(}non-cyclic{\rm)} well-structured nilpotent $R$-groups. Let $G=\mathcal G(\Gamma, \gvv)$ and $H=\mathcal G(\Delta, \{H_v\}_{v\in\Delta})$ where $G_v \in \textbf C_{n}(R)$ and $H_w \in \textbf C$. Then,
		
		\begin{gather} 
			\begin{split}
				G \equiv H \Leftrightarrow  & f:\Gamma \to \Delta \hbox{ is an isomorphism and } \\
				& H_{f(v)} \hbox{ is an abelian deformation of } G_v.
			\end{split}
		\end{gather}
	\end{thm}

	\begin{thm}[Rigidity of the class of graph products of (some) solvable groups] \
		
		Let $\textbf C_{BSF}$ be the class of solvable Baumslag-Solitar groups and finitely generated free solvable groups. Let $G=\mathcal G(\Gamma, \gvv)$ and $H=\mathcal G (\Delta, \{H_v\}_{v\in \Delta})$ where $G_v\in \textbf C_{BSF}$ and $H_w \in \textbf C$. Then
		$$
		G \equiv H \Leftrightarrow G\simeq H.
		$$
		See {\rm\cite{Nies,CKBS, RSS}} for first-order rigidity results on the class of solvable Baumslag-Solitar groups and finitely generated free solvable groups.
	\end{thm}
	
	\begin{thm}[Rigidity of the class of graph products of classic linear groups]
		\
		
		Let $\textbf C_{L}$ be the class of classic linear groups over a field $R$ ($GL_n(R), SL_n(R), UT_n(R)$), where $n\ge 3$.
		
		Let $G=\mathcal G(\Gamma, \gvv)$ and $H=\mathcal G (\Delta, \{H_v\}_{v\in \Delta})$, where $G_v, H_w \in \textbf C_{L}$. Then
		$G \equiv H$ if and only if there exist an isomorphism $f:\Gamma \to \Delta$ and and $G_v=X(R)$, $H_{f(v)}=X(R')$ are linear groups of the same type and  $R\equiv R'$. 
		
		See {\rm\cite{Mal, Belegradek}} for first-order rigidity results on the corresponding linear groups.
	\end{thm}
	
	\subsection{Classification of finitely generated groups elementarily equivalent to Droms RAAGs}
	
	\begin{definition}\label{defn:structure Droms RAAGs}
		A \emph{right-angled Artin group} (RAAG for short) is a graph product $\mathcal G(\Gamma, \{G_v\}_{v\in \Gamma})$, where $G_v$ is infinite cylic for all $v\in \Gamma$.
		
		A RAAG is called \emph{Droms} if the defining graph $\Gamma$ does not contain induced cycles or paths of length more than 3. By convention, the length of a cycle and a path is determined by the number of vertices.
	\end{definition}
	
	\begin{definition}\label{defn:structure Droms RACGs}
		A \emph{right-angled Coxeter group} (RACG for short) is a graph product $\mathcal G(\Gamma, \{G_v\}_{v\in \Gamma})$, where $G_v$ is the group of order $2$ for all $v\in \Gamma$.
		
		As above, a RACG is called \emph{Droms} if the defining graph $\Gamma$ does not contain induced cycles or paths of length more than 3.
	\end{definition}
	
	Droms RAAGs admit different characterisations. On the one hand, they are precisely the RAAGs whose finitely generated subgroups are again RAAGs. They can also be described recursively as follows.
	
	\begin{definition} \label{defn:hheight}
		Let $D_1=\{\mathbb Z\}$, let $F_1$ be the class of free products of groups in $D_1$ (that is $D_1=\{F_n \mid n\in \mathbb N \cup \{0\}\})$ and let $Z_1$ be the class of groups of the form $\mathbb Z^m \times G$, where $G\in F_1$, $m\in \mathbb N$ (that is $Z_1=\{ \mathbb Z^m \times F_n\}$).
		
		Assume that $D_{i-1}$, $F_{i-1}$ and $Z_{i-1}$ have been defined. Then, we define $D_i=\bigcup\limits_{j=1, \dots i-1} Z_i$, $F_i$ is the free product of groups in $D_i$ and $Z_i$ is the class of groups of the form $\mathbb Z^m\times G$ where $G\in F_i$ and $m\in \mathbb N$.
	\end{definition}

	\begin{prop}[Droms, see \cite{Droms}]
		The class of Droms RAAGs coincides with the class of groups $\bigcup\limits_{i\in \mathbb N}D_i$.
	\end{prop}
	
	In the next theorem we use the recursive structure of Droms RAAGs given in Definition \ref{defn:structure Droms RAAGs} to describe finitely generated group elementarily equivalent to a Droms RAAG. The main tools are Sela's results on finitely generated groups elementarily equivalent to a free product and our interpretability results of the base structure.
	
	The classification of finitely generated groups elementarily equivalent to a free group or a free product is basically given in terms of a family of groups called hyperbolic towers. We refer the reader to \cite{sela2006diophantineVI,sela2010diophantine} for a precise definition. In our context, the reader can take the definition of hyperbolic tower over a free group (resp. a free product) to be a finitely generated group elementarily equivalent to a non-abelian free group (resp. elementarily equivalent to a given free product).
	
	\begin{thm}  \label{thm:eqdromsraags}
		We use the notation of {\rm Definition \ref{defn:structure Droms RAAGs}}. For Droms RAAGs of height $1$, we have
		\begin{itemize}
			\item A finitely generated group $H$ is elementarily equivalent to $G\in D_1$, i.e. $H\equiv \mathbb Z$, if and only if $H \simeq \mathbb Z$.
			\item A finitely generated group $H$ is elementarily equivalent to $G \in F_1$, i.e. $H\equiv F_n$, $n\in \mathbb N$, if and only if $H$ is a hyperbolic tower.
			\item A finitely generated group $H$ is elementarily equivalent to $\mathbb Z^m \times G$, $G\in F_1$ if and only if $H \simeq \mathbb Z^m \times T$, where $T$ is a hyperbolic tower.
		\end{itemize}
		
		Assume that for $G\in D_{i-1}, F_{i-1}, Z_{i-1}$ we have a description of finitely generated groups $H$ elementarily equivalent to $G$.   For Droms RAAGs of height $i$, we have
		\begin{itemize}
			\item Let $G\in D_i$. By definition $G\in Z_j$ for some $j< i$ and so by induction we have described $H \equiv G$.
			\item Let $G\in F_i$, that is
			$$
			G= \bfrp\limits_{j=1, \dots, i-1} \left( G_{j,1} \ast \dots \ast G_{j,k_j}\right),
			$$
			where $G_{j,l}\in D_{j}$ for all $l=1, \dots, k_j$. A finitely generated group $H \equiv G$  if and only if $H$ is a hyperbolic tower over
			$$
			\bfrp\limits_{j=2, \dots,i-1} \left( H_{j,1} \ast \dots \ast H_{j,k_j}\right)
			$$
			and $H_{j,l}\equiv G_{j,l}$ for all $l$.
			
			\item Let $G\in Z_i$, that is $G=\mathbb Z^m \times G'$, where $G'\in F_i$. A f.g $H \equiv G$ if and only if $H\simeq \mathbb Z^m \times H'$ and $H' \equiv G'$.
		\end{itemize}
	\end{thm}
	\begin{proof}
		The base of induction follows from results of Szmielew for abelian groups, see \cite{Szmielew} and Sela and Kharlampovich-Miasnikov for free groups, see \cite{sela2006diophantineVI} and \cite{kharlampovich2006elementary}. If $H \equiv \mathbb Z^m \times G$, then from Corollary A in \cite{casals2010elements} we have that $H\simeq Z(H) \times H'$ where $Z(H)\equiv \mathbb Z^m$ and $H'\equiv G$. As $H$ is finitely generated so are $Z(H)$ and $H'$. Then as we argued above, we have that $Z(H)\simeq \mathbb Z^m$ and $H'$ is a hyperbolic tower.
		
		For the induction step, suppose that $G \in F_i$.
		If
		$$
		H \equiv G= \bfrp\limits_{j=1, \dots, i-1} \left( G_{j,1} \ast \dots \ast G_{j,k_j}\right),
		$$
		by \cite{CK_free_products} it follows that $H$ is a hyperbolic tower over $\bfrp\limits_{j=1, \dots,i'-1} \left( H_{j,1} \ast \dots \ast H_{j,k_j'}\right)$, where $H_{l,j} \in D_j$. Sela shows in \cite{sela2010diophantine} that the hyperbolic tower is elementarily equivalent to the base and so we have that
		$$
		G= \bfrp\limits_{j=1, \dots, i-1} \left( G_{j,1} \ast \dots \ast G_{j,k_j}\right) \equiv \bfrp\limits_{j=1, \dots,i'-1} \left( H_{j,1} \ast \dots \ast H_{j,k_j'}\right).
		$$
		
		Since $G_{j,l}$ and $H_{j',l'}$ are cyclic for $j=1$ and have nontrivial center (and are non-cyclic) for $j>1$ (and so  they satisfy a simple non-generic almost positive sentence), it follows from Corollary \ref{cor:main interpretation} that $G_{j,l}$ and $H_{j',l'}$ are interpretable for $j,j'>1$ and so $i=i'$, $k_j=k_j'$, for $j\in \{2, \dots, i\}$ and up to reordering we have that $G_{j,l} \equiv H_{j,l}$ for $j\in \{2, \dots, i\}$ and $l\in \{k_1, \dots, k_j\}$. Furthermore, since the core of $\Gamma$ is interpretable, we have that if $G=G_{i,1} \ast \mathbb Z$, then $H=H_{i_1} \ast F_n$ where $n \ge 1$. If $G$ has more than one factor non-cyclic, $k_1' \ge 0$.
		
		The converse is a consequence of Sela's results. Indeed, if $G_{j,l} \equiv H_{j,l}$ for $j\ge 2$ and $l\in \{1, \dots, k_j\}$ and if $G=G_{i,1} \ast \mathbb Z$, then $H=H_{i_1} \ast F_n$ where $n \ge 1$ then
		$$
		G= \bfrp\limits_{j=1, \dots, i-1} \left( G_{j,1} \ast \dots \ast G_{j,k_j}\right) \equiv \bfrp\limits_{j=1, \dots,i'-1} \left( H_{j,1} \ast \dots \ast H_{j,k_j'}\right).
		$$
		
		Since by Sela's results $\bfrp\limits_{j=1, \dots,i'-1} \left( H_{j,1} \ast \dots \ast H_{j,k_j'}\right)$ is elementarily equivalent to a hyperbolic tower over it, the result follows.
	\end{proof}

	\begin{thm} \label{thm:eqdromsracgs}
		We use the notation of {\rm Definition \ref{defn:structure Droms RACGs}} and {\rm Definition \ref{defn:hheight}}, which is adapted to the case of RACG in an obvious way.
		
		For Droms right-angled Coxeter groups of height $1$, we have
		\begin{itemize}
			\item A finitely generated group $H$ is elementarily equivalent to $G\in D_1$, i.e. $H\equiv \mathbb Z_2$, if and only if $H \simeq \mathbb Z_2$.
			\item If a finitely generated group $H$ is elementarily equivalent to $G \in F_1$, $G=*_{i=1,\dots, k}\mathbb Z_2$, $k\ge 2$, then if $k=2$, then $H\simeq D_\infty$, otherwise,  $H\equiv *_{i=1,\dots, k}\mathbb Z_2$ is a hyperbolic tower over $G$.
			\item A finitely generated group $H$ is elementarily equivalent to $\mathbb Z_2^m \times G$, $G\in F_1$ if and only if $H \simeq \mathbb Z_2^{m} \times D_\infty$.
		\end{itemize}
		
		Assume that for $G\in D_{i-1}, F_{i-1}, Z_{i-1}$ we have a description of finitely generated groups $H$ elementarily equivalent to $G$.
		
		For Droms right-angled Coxeter groups of height $i$, we have
		\begin{itemize}
			\item Let $G\in D_i$. By definition, $G\in Z_j$ for some $j< i$ and so by induction we have described $H \equiv G$.
			\item Let $G\in F_i$, that is
			$$
			G= \bfrp\limits_{j=1, \dots, i-1} \left( G_{j,1} \ast \dots \ast G_{j,k_j}\right),
			$$
			where $G_{j,l}\in D_{j}$ for all $l=1, \dots, k_j$. A finitely generated group $H \equiv G$  if and only if $H$ is a hyperbolic tower over
			$$
			\bfrp\limits_{j=2, \dots,i-1} \left( H_{j,1} \ast \dots \ast H_{j,k_j}\right)
			$$
			and $H_{j,l}\equiv G_{j,l}$ for all $l$.
			
			\item Let $G\in Z_i$, that is $G=\mathbb Z_2^m \times G'$, where $G'\in F_i$. A finitely generated $H \equiv G$ if and only if $H\simeq \mathbb Z_2^m \times H'$ and $H' \equiv G'$.
		\end{itemize}
	\end{thm}
	\begin{proof}
		Proof is identical to that of Theorem \ref{thm:eqdromsraags} once we settle the case of finitely generated groups which are elementarily equivalent to the infinite dihedral group. If $G$ is finitely generated and elementarily equivalent to $D_\infty$, we claim that $G$ is isomorphic to $D_\infty$. Indeed, the formula $\phi(x):\ x=1 \vee (x^2\ne 1 \wedge \forall y \, x^y = x^{\pm 1})$ defines an index $2$ subgroup of $D_\infty$. In fact, $D_\infty$ is the nontrivial semidirect product of the infinite cyclic subgroup $\phi(D_\infty)$ and an element of order $2$. If $G$ is finitely generated and $G\equiv D_\infty$, we have that $\phi(G)\equiv \phi(D_\infty)\simeq \mathbb Z$ and $\phi(G)$ is of index $2$ in $G$. Since $G$ is finitely generated, so is $\phi(G)$ and as $\phi(G)\equiv \mathbb Z$, we have that $\phi(G)\simeq \mathbb Z$. Then $G$ is the nontrivial semidirect product of an infinite cyclic group and the group of order $2$ and so it is isomorphic to $D_\infty$.
	\end{proof}

	\section{Open questions}
	
	In this section we collect some of the questions that arise naturally from our work.
	
	\bigskip
	
	\paragraph{\textbf{From simple to all non-generic almost positive sentences}}
	
	The most general condition that we require to the vertex groups is that they satisfy a \emph{simple} non-generic almost positive sentence, and so in particular, we only allow for one inequality. Can one generalise the results to non-generic almost positive sentences and so consider finitely many inequalities? More precisely, we pose the following
	
	\begin{question}
		
		Let $\G(\Gamma, \gvv) \in \mc{G}(\textbf{F}_N, \textbf{C}_{\Phi}^r)$ {\rm(}see {\rm Definition \ref{d: K_N_and_C})}, where $N>0$ and $\Phi$ consists of non-generic almost positive sentences and $\Gamma$ is almost positive reduced for any $\G(\Gamma,\gvv)$ in $\textbf C$.
		
		Is the $\Core(\Gamma, \gvv)$ {\rm(}uniformly{\rm)} interpretable in $\G$?
	\end{question}
	
	\paragraph{\textbf{Characterisation of groups that have generic almost positive theory.}}
	
	Given a theory, that is a set of sentences, it is natural to try to describe the groups that are models of that theory.
	
	Makanin describe the positive sentences that are satisfied by non-abelian free groups - they are precisely the positive sentences that admit formal solutions. We then can wonder what are groups that satisfy the same positive sentences as a free group, i.e. what are the groups with trivial positive theory.
	
	A way to establish that a group has trivial positive theory is to show that it projects to a free group or more generally to any group with trivial positive theory, since if a group satisfies a positive sentence, so does any quotient.
	
	The most general tool up to date to show that a group has trivial positive theory is via small cancellation techniques. In \cite{casals2019positive}, we gave sufficient conditions for a finitely generated group acting on a simplicial tree to have trivial positive theory and in \cite{Simon} the authors show that in fact all acylindrically hyperbolic groups have trivial positive theory.
	
	As the definition of non-generic almost positive sentences is given by the existence of a (relative) formal solution, we also expect that small cancellation tools may be useful to describe classes of groups that do not satisfy any non-generic almost positive sentence.
	
	For instance, one can deduce from the deep work of Sela on the elementary theory of torsion-free hyperbolic groups, that these groups do not satisfy any non-generic almost positive sentence. One can then ask the following
	
	\begin{question}
		Can acylindrically hyperbolic groups satisfy a non-generic almost positive sentence?
	\end{question}
	
	We expect that the answer will be negative. As we mentioned, in \cite{casals2019positive}, we show that a finitely generated group acting on a tree with weak small cancellation elements has trivial positive theory. Roughly speaking, a tuple of elements is weakly small cancellation if it is ``small cancellation modulo the kernel of the action on the tree", see Definition 3.2 in \cite{casals2019positive}. Notice that having weak small cancellation elements is not enough to assure that a group cannot satisfy a non-generic almost positive sentence - indeed, $\mathbb Z\times F_2$ contains weak small cancellation elements but satisfies a non-generic almost positive sentence, namely it has a nontrivial center.
	
	Consider the class $\mathbf W$ of finitely generated groups acting on a simplicial tree that contain weakly small cancellation elements. For each finite tuple $W$ of weakly hyperbolic elements, let $K_W$ be the intersection of the kernels of the axis of the elements in $W$. Consider the subclass of groups $\mathbf K \subset \mathbf W$ that satisfy that the intersection of all the kernels $K_W$ where $W$ runs over all weak small cancellations tuples is trivial, i.e. $\bigcap\limits_{W \hbox{ weak small cancellation}} K_W= 1$. One can ask the following
	
	\begin{question}
		Do groups from the class $\mathbf K$ satisfy a non-generic almost positive sentence?
	\end{question}
	
	Notice that non-solvable Baumslag-Solitar groups belong to the class $\mathbf K$ so a concrete question is the following one.
	
	\begin{question}
		Do non-solvable Baumslag-Solitar groups satisfy a non-generic almost positive sentence?
	\end{question}
	
	\paragraph{\textbf{More general sentences.}}
	
	Our results apply to groups that satisfy a simple non-generic almost positive sentence (and conjecturally, to all non-generic almost positive sentences). However, we believe that the theory could be wider and contain more general sentences.
	
	We expect that non-solvable Baumslag-Solitar groups do not satisfy a non-generic almost positive sentence. However, they do satisfy the following sentence
	$$
	\phi \equiv \exists z\ne 1 \forall x \exists y\ne 1 \,\, ([z,y]=1 \wedge [y,y^x]=1.
	$$
	
	Although the sentence is not almost positive, it is in some sense ``non-generic": the only formal solution for the positive part of the sentence $([z,y]=1 \wedge [y,y^x]=1$ is $y=1$ and so formal solutions cannot ``witness" the inequalities. In this case, one can proof that in the free product of non-solvable Baumslag-Solitar groups the factors are interpretable.
	
	We believe that one should be able to define a theory of ``non-generic generalize positive" sentences and extend our results to groups that are models of this theory, namely, if one considers a graph product of groups in this class, the core of (an adjusted notion of) a reduced graph and the associated vertex groups should be interpretable. We formulate this in the following
	
	\begin{question}
		What is the maximal theory that allows to recover a (meaningful) base structure of the graph product if the vertex groups are models of that theory?
	\end{question}

	\paragraph{\textbf{The core is an elementary invariant.}}
	
	In the introduction we have conjectured that the core of the graph would determine the elementary class inside the family of RAAGs. In fact, we believe that it is going to characterise the elementary class of groups inside the appropriate class of graph products, more precisely, we conjecture:
	
	\begin{conj}
		Let $\G(\Gamma, \gvv), \G(\Gamma', \{G_v'\}) \in \mc{G}(\textbf{F}_N, \textbf{C}_{\Phi}^r)$ {\rm(}see {\rm Definition \ref{d: K_N_and_C})}, where $N>0$ and either:
		\begin{itemize}
			\item $\Phi$ consists only of nontrivial positive sentences and $\Gamma$ is positive reduced for any $\G(\Gamma,\gvv)$ in $\textbf C$.
			\item $\Phi$ consists non-generic almost positive sentences and $\Gamma$ is almost positive reduced for any $\G(\Gamma,\gvv)$ in $\textbf C$.
		\end{itemize} 
		
		Then $\G(\Gamma, \gvv) \equiv \G(\Gamma', \{G_v'\}_{v\in \Gamma'})$ if and only if there is an isomorphism $f:\Core(\Gamma) \to \Core(\Gamma')$ and $G_v \equiv G_{f(v)}'$.
	\end{conj}
	
	Notice that the above conjecture is a generalization of Vaught's conjecture on the elementary equivalence of free products of pairs of elementarily equivalent groups to the context of graph products in the class: if $G_{v}\equiv G_{v'}$ for $v\in \Gamma$, then $\G(\Gamma, \gvv)\equiv \G(\Gamma', \{G_v'\}_{v\in \Gamma'})$. In the case of free products, Zlil Sela gave a positive answer in \cite{sela2010diophantine}.
	
	\paragraph{\textbf{From finite to infinite graphs.}}
	
	Although we have formulated most of the main results assuming the graph $\Gamma$ to be finite, in fact, the main graph theoretical property that we need is the Property $K_n$, see Definition \ref{defn:propDn} (which implies property $AP_{2n}$). In this context, we can ask if the above conjecture generalises to infinite graphs. We state it in the following question:
	
	\begin{question}
		Let $\G(\Gamma, \gvv), \G(\Gamma', \{G_v'\}) \in \mc{G}(\textbf{K}_N, \textbf{C}_{\Phi}^r)$ {\rm(}see {\rm Definition \ref{d: K_N_and_C})}, where $N>0$ and either:
		\begin{itemize}
			\item $\Phi$ consists only of nontrivial positive sentences and $\Gamma$ is a (possibly infinite) positive reduced for any $\G(\Gamma,\gvv)$ in $\textbf C$.
			\item $\Phi$ consists non-generic almost positive sentences and $\Gamma$ is (possibly infinite) almost positive reduced for any $\G(\Gamma,\gvv)$ in $\textbf C$.
		\end{itemize} 
		
		If $\G(\Gamma, \gvv) \equiv \G(\Gamma', \{G_v'\}_{v\in \Gamma'})$, then $\Core(\Gamma) \equiv \Core(\Gamma')$.
	\end{question}
	
	One can also consider under which conditions we have a converse.

	A specific question is the following one.
	
	Let $P_\infty$ be the infinite $\mathbb Z$-chain. More precisely, $V(P_\infty)=\{v_n \mid n\in \mathbb Z\}$ and $(v_n,v_m)\in E(P_\infty)$ if and only if $|n-m|=1$. Let $Q_\infty$ be the disjoint union of two copies of $P_\infty$. One can show that $P_\infty \equiv Q_\infty$.

	\begin{question}
		Consider the graph products $\G_1=\mathcal G(P_\infty, \{\mathbb Z\}_{v\in P_\infty})$ and $\G_2=\mathcal G(Q_\infty, \{\mathbb Z\}_{v\in P_\infty})$. Are $\mathcal G_1 \equiv \G_2$?
	\end{question}
	
	\bibliography{references}
	\addcontentsline{toc}{section}{References}

\end{document}